\documentclass[12pt, reqno]{amsproc}
\usepackage{eurosym}
\usepackage{tabto}
\usepackage{enumerate}
\usepackage{amsmath}
\usepackage{enumitem}
\usepackage{amsfonts}
\usepackage{amssymb}
\usepackage{graphicx}
\usepackage{color}
\usepackage[pdfpagelabels]{hyperref}
\usepackage{tikz}
\usepackage[utf8]{inputenc}

\numberwithin{equation}{section}
\definecolor{mycolorred}{rgb}{1, 0, 0}

\definecolor{mycolorblue}{rgb}{0, 0, 1}

\definecolor{mycolorpink}{rgb}{1, 0, 1}

\newtheorem{theorem}{Theorem}[section]

\newtheorem{corollary}[theorem]{Corollary}

\newtheorem{definition}[theorem]{Definition}
\newtheorem{assumption}[theorem]{Assumption}

\newtheorem{example}[theorem]{Example}

\newtheorem{lemma}[theorem]{Lemma}

\newtheorem{proposition}[theorem]{Proposition}
\newtheorem{remark}[theorem]{Remark}

\def\<{\langle}
\def\>{\rangle}

\def\I{\mbox{\large \bf 1}}

\def\E{{\mathbb E}}
\def\R{{\mathbb R}}

\def\N{{\mathbb N}}

\def\cF{{\mathcal F}}
\def\cP{{\mathcal P}}
\def\ed{\mathbf{e}_d}

\def\betaMV{b}

\begin{document}
	\def\kp{{\kappa_p}}
	\def\X{\mathcal{X}}
	
	\title[Ergodic approximation for the invariant distribution]{Ergodic approximation for the invariant distribution: An abstract framework for law-dependent dynamics}
	\author{Aur\'elien Alfonsi}
\address{CERMICS, ENPC, Institut Polytechnique de Paris, CNRS, Marne-la-Vallée, France \& MathRisk team-project, Inria Paris, France.}
\email{aurelien.alfonsi@enpc.fr}
\author{Vlad Bally}
\address{Universit\'e Gustave Eiffel, LAMA (UMR CNRS, UPEMLV, UPEC), MathRisk INRIA,
  F-77454 Marne-la-Vall\'ee, France.}
\email{vlad.bally@univ-eiffel.fr}
\author{Lucia Caramellino}
\address{Dipartimento di Matematica and IN$\mathrm{d}$AM-GNAMPA, Universit\`a degli Studi di Roma Tor Vergata -- Via della Ricerca Scientifica 1, 00133 Roma, Italy}
\email{caramell@mat.uniroma2.it}
\author{Arturo Kohatsu-Higa}
\address{Department of Mathematical Sciences, Ritsumeikan University }
\email{khts00@fc.ritsumei.ac.jp}
		
	\begin{abstract}
		This paper studies the approximation of invariant distributions for a broad class of law-dependent dynamics, including McKean-Vlasov stochastic differential equations and Boltzmann-type equations. We consider discrete-time approximation schemes with decreasing time steps and analyse the convergence of their associated ergodic (or occupation) measure towards the invariant distribution of the underlying continuous-time process. Under a general coupling assumption, we prove convergence in the expected $p$-Wasserstein distance ($p \ge 1$) and derive explicit convergence rates. Our approach combines estimates on ergodic averages, regularization techniques for discrete measures, and a generalized discrete Gronwall lemma to control the error between the self-interacting scheme and the target invariant measure. We show that our framework applies to a wide range of models: McKean-Vlasov SDEs, a Boltzmann type equation, and a neuronal model.\bigskip
	\end{abstract}
	\renewcommand{\subjclassname}{2020 Mathematics Subject Classification}

	\maketitle
	\section{Introduction}
	
	Law-dependent dynamics provide a broad modeling framework for complex systems composed of many interacting agents or particles. In these dynamics, the evolution depends not only on the present state but also on its statistical distribution. They are thus especially suited to capture collective behaviour and mean-field interactions. Prominent examples are McKean-Vlasov Stochastic Differential Equations~\cite{McKean} and Boltzmann equations~\cite{Sznitman}. Law-dependent dynamics arise naturally in statistical physics, where they describe interacting particle systems and granular media~\cite{Sznitman,Malrieu}; in economics and finance~\cite{Carmona}, where they model systemic risk~\cite{CFS}, mean-field games~\cite{CaLe}, and the behavior of large markets; and in biology, where they appear in models of swarm dynamics~\cite{CFRT}, neural networks~\cite{DGLP,FoLo}, and population evolution. 
	
	A fundamental theoretical question for these dynamics is the existence of an invariant distribution, its uniqueness, as well as the weak convergence of the process toward this measure. This is also important in practice for models on an infinite time horizon, since the invariant measure then represents the long-time average behaviour of the particles or agents. These questions have been addressed for decades for Markov processes, see Meyn and Tweedie~\cite{MeTw} and \cite{Hairer} and the references therein, but also for McKean-Vlasov SDEs, see e.g. Tamura~\cite{Tamura}, Benachour et al.~\cite{BRTV}, Malrieu~\cite{Malrieu}, Eberle et al.~\cite{EGZ} or more recently Cormier~\cite{Cormier} to mention a few.  
	
	A very important practical question is then  the computation of the invariant distribution. This is usually done in two steps,  which can be performed in either order, as shown in the following schematic diagram (we have only indicated on this diagram the closest references to the present work).
	\usetikzlibrary{arrows.meta, positioning,calc}
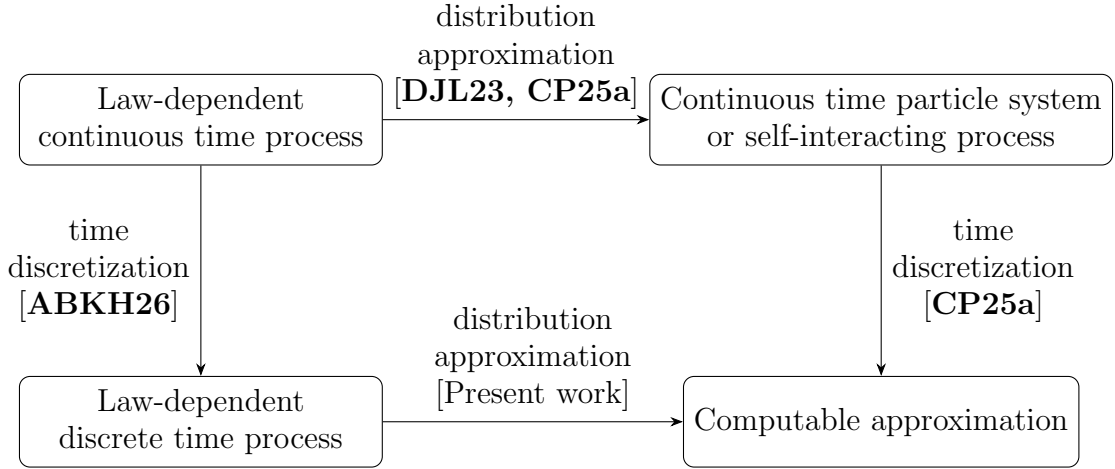
	\begin{center}
		\begin{tikzpicture}[ 
			box/.style={
				draw,
				rounded corners,
				align=center,
				minimum width=4.8cm,
				minimum height=1.2cm
			},
			>=Stealth
			]
			
			\node[box] (A) at (0,4)
			{Law-dependent\\ continuous time process};
			
			\node[box] (B) at (0,0)
			{Law-dependent\\ discrete time process};
			
			\node[box] (C) at (9,4)
			{Continuous time particle system\\ or self-interacting process};
			
			\node[box] (D) at (9,0)
			{Computable approximation};
			
			\draw[->]
			(A) -- node[left, draw=none, fill=none, align=center]
			{time\\ discretization\\ \cite{ABKH}} (B);
			
			\draw[->]
			(C) -- node[right, draw=none, fill=none, align=center]
			{time\\ discretization\\ \cite{ChPa}} (D);
			
			\draw[->]
			(A) -- node[above, draw=none, fill=none, align=center]
			{distribution\\ approximation \\ \cite{DJL,ChPa}} (C);
			
			\draw[->]
			(B) -- node[above, draw=none, fill=none, align=center]
			{distribution\\ approximation \\ {[Present work]}} (D);
			
		\end{tikzpicture}
	\end{center}
	
	The first step is to approximate the continuous time law-dependent process by a discrete one. This is often done by using a Euler-type scheme. For this step, there are two main approaches in the literature. The first one uses a constant (small) time step and analyses the distance between the invariant distribution of the continuous process and the invariant distribution of its approximation, see e.g.~\cite{Talay2002}, Mattingly et al.~\cite{MSH}, Brosse et al.~\cite{BDM} and Crisan et al.~\cite{CDO} to mention a few. The second one, pioneered by Lamberton and Pagès~\cite{LP02,LP03} and Lemaire~\cite{Lemaire} for SDEs, uses a decreasing sequence of time steps that gives directly the convergence towards the invariant measure of the continuous time process, see also the recent works of Pagès and Panloup~\cite{PaPa}, Bally and Qin~\cite{BaQi} or Panloup and Reygner~\cite{PaRe}. Recently, Alfonsi et al.~\cite{ABKH} have proposed an abstract framework under which they show the convergence of  law-dependent discrete time approximations toward the invariant distribution of the continuous time process. They also obtain a speed of convergence in terms of general Wasserstein distance.  
	
	The second step is the approximation of the state distribution. This is mainly done in two different ways. The first one consists of using an interacting particle system, and we refer to Méléard~\cite{Meleard} and the references therein. The study of the discretization of these particle systems in view of approximating the invariant distribution is an active field of research, see Chen et al.~\cite{CdRS} or Wang and Wu~\cite{WaWu} to mention a few. 
	However, the computational cost required by the simulation of a discretized interacting particle system may be important, and it may be more efficient from a computational point of view to consider self-interacting processes, as pointed out by AlRachid et al.~\cite{ABRS}. This consists of replacing the state distribution by the occupation measure in the law-dependent dynamics, which is of course only relevant for ergodic processes. Self interacting diffusions have been introduced by Benaim et al.~\cite{BLR}, and Du et al.~\cite{DJL} have recently analysed the convergence of the occupation measure toward the invariant distribution of the corresponding McKean-Vlasov process. They obtain in particular a convergence rate in terms of the expected $2$-Wasserstein distance between the occupation measure and the invariant distribution. Their results have then been extended by Du et al.~\cite{DRSW} to exponentially weighted occupation measure. Note however that these works only analyse the error with respect to another continuous time process which does not correspond exactly to a simulation method. Chassagneux and Pagès~\cite{ChPa} filled the gap toward a computable approximation and analyse both the error between the invariant measure of the McKean-Vlasov process and the occupation measure of the corresponding self-interacting process, and the error between the occupation measure of the corresponding self-interacting process and its discrete counterpart. The analysis is carried out both in expected $2$-Wasserstein distance and also in a.s. $2$-Wasserstein distance, and they provide rates of convergence.   
	
	The goal of the present work is to show the convergence of the occupation measure of a law-dependent approximation scheme towards the invariant distribution of a law-dependent process under a general framework that encompasses McKean-Vlasov and Boltzmann equations. In contrast with the work of Chassagneux and Pagès~\cite{ChPa}, we follow the other path on the above diagram. We first consider the approximation of the continuous dynamics by a discrete approximation scheme with decreasing time steps $(\gamma_k)_{k\ge 1}$. It has the following generic form
	$$ X_{k+1}= \psi_{k+1}(X_k, \mathcal{L}(X_k)).$$
	Then, we analyse the error between the invariant measure of this scheme and the empirical ergodic (or occupation) measure $\rho_k(Y)=\frac{1}{t_k}\sum_{i=1}^k \gamma_i \delta_{Y_i}$ with $t_k=  \sum_{i=1}^k \gamma_i$ and 
	$$ Y_{k+1}= \psi_{k+1}(Y_k, \rho_k(Y)),$$
	with $\mathcal{L}(Y_0)=\mathcal{L}(X_0)$. The analysis of the convergence of $\mathcal{L}(X_k)$ towards the invariant measure~$\mu_*$ of the corresponding law-dependent continuous time process is carried out in~\cite{ABKH} under a general condition on the marginal laws. However, this is not sufficient to deal with the ergodic sums that involve the joint law of $(X_1,\dots,X_k)$. This is why we introduce a related but different coupling assumption on the functions $\psi_k$ (see Assumption~\ref{Ass_SelfCoupling} below) that enables us to get the convergence 
	of $\mathcal{L}(X_k)$ towards~$\mu_*$ for the $p$-Wasserstein distance ($p\ge 1$). 
	The main contribution of the paper is then to analyse the convergence of the empirical ergodic measure $\rho_k(Y)$ towards $\mu_*$. This analysis is inspired by the recent work of Du et al.~\cite{DJL}, but there are two notable differences: our analysis is made in a discrete setting, and the abstract framework that we consider requires to handle carefully the coupling between the processes $X$ and $Y$. Namely, we first analyse the expected $p$-Wasserstein distance between the empirical ergodic measure $\rho_k(X)=\frac{1}{t_k}\sum_{i=1}^k \gamma_i \delta_{X_i}$ and $\mathcal{L}_k(X)=\frac{1}{t_k}\sum_{i=1}^k \gamma_i \mathcal{L}(X_i)$, which is the weighted average of the marginal laws. To do so, we use a regularisation of the empirical measure together with the Horowitz and Karandikar~\cite{HK} coupling used for the $p$-Wasserstein distance and we get an upper bound of the expected $p$-Wasserstein distance as a function of 
	$$\E\left[\left(\frac 1{t_k}\sum_{i=1}^k (\varphi(X_i) - \E[\varphi(X_i)]) \right)^2 \right],$$
	for test functions $\varphi \in C^1_b(\R^d)$ bounded with bounded derivatives. We then analyse directly this quantity and get a rate of convergence (see Theorem~\ref{thm_ergo}). Let us emphasize here that we develop in Appendix~\ref{app:HK} a regularization technique for random measures that is quite flexible and may be interesting for other applications, see Remark~\ref{Remark_regularization}. Last, we analyse the expected $p$-Wasserstein distance between the ergodic measures $\rho_k(X)$ and $\rho_k(Y)$, which relies in particular on a generalized discrete Gronwall lemma (see Lemma~\ref{lem_discrete_Gronwall}). Combining the two estimates, we get a convergence rate 
	towards zero for the expected $p$-Wasserstein distance between $\rho_k(Y)$ and $\mu_*$, see our main Theorem~\ref{MAIN0}. With respect to the existing works in the literature, our main contribution is to give a general framework for law-dependent dynamics that include McKean-Vlasov SDEs, but also many other ones possibly with  jumps and Boltzmann type interactions. Let us mention also that we work with $p$-Wasserstein distances, $p\ge 1$, while the works~\cite{DJL,ChPa} use the $2$-Wasserstein distance.

	The paper is structured as follows. Section~\ref{sect:results} presents our general framework, the main assumptions and the main results of the paper. Section~\ref{sect:proof-prop} is dedicated to the convergence of $\mathcal{L}(X_k)$ towards the invariant measure. Then, Section~\ref{sect:proof-thm} is devoted to showing the convergence of the ergodic measure $\rho_k(Y)$ towards the invariant measure. Last, Section~\ref{sect:examples} presents different applications of our general framework.  Namely, we show that it can be applied to McKean-Vlasov SDEs, and we compare our results to the recent work of Chassagneux and Pagès~\cite{ChPa}. We also show that it can be used for a Boltzmann type equation and for a Neuronal model~\cite{DGLP} with jumps and law-dependence in the drift.

	\section{Assumptions and main results}
	\label{sect:results}
	
	\subsection{Preliminaries and definitions}
	
	We define some chains which are appropriate to model Euler approximation
	schemes for equations which depend on the law of the solution - two
	important examples are the McKean-Vlasov equations and Boltzmann equations.
	
	We fix $d\in \mathbb{N}$ and $p\geq 1$ and we denote $\mathcal{P}_{p}(%
	\mathbb{R}^{d})$ the space of probability measures on $\mathbb{R}^{d}$ with
	the norm $\left\Vert \mu \right\Vert _{p}=(\int_{\mathbb{R}^{d}}\left\vert
	x\right\vert ^{p}\mu (dx))^{1/p}<\infty .$ If $\mu (dx)=\mu (\omega ,dx)$ is
	a random probability measure we denote $\left\Vert \mu \right\Vert _{p}^{p}={%
		\mathbb{E}}[\int_{\mathbb{R}^{d}}\left\vert x\right\vert ^{p}\mu (\omega
	,dx)].$ 
	
	In this article, we work with the \textbf{Wasserstein distance }$W_{p}$, $%
	p\geq 1$ defined by 
	\begin{equation*}
		W_{p}^{p}(\mu ,\nu )=\inf_{\Pi _{\mu ,\nu }}\int_{\mathbb{R}^{d}\times 
			\mathbb{R}^{d}}\left\vert x-y\right\vert ^{p}d\Pi _{\mu ,\nu }(x,y),
	\end{equation*}%
	where the infimum is taken over all the probability measures $\Pi_{\mu,\nu} \in  \mathcal{P}_{p}(\mathbb{R}^{d} \times \mathbb{R}^{d})$ with marginals $\mu $ and $\nu .$ If $X,Y$ are two random variables we make an abuse of notation and write 
	\begin{equation*}
		W_{p}(X,Y)=W_{p}(\mathcal{L}(X),\mathcal{L}(Y)).
	\end{equation*}
	
	\begin{definition}\label{def_moment}{\textbf{Moments}: We denote  $\Gamma _{p}(\mu ,\nu ):=1+\|\mu\|_p+\|\nu\|_p$  and similarly
			\begin{enumerate}
				\item $\Gamma _{p}(\mu )=\Gamma _{p}(\mu ,\mu )$
				\item For random variables $X,Y$ with laws $(\mathcal{L}(X),\mathcal{L}(Y))=(\mu ,\nu )$, $\Gamma _{p}(X,Y)=\Gamma _{p}(\mu ,\nu )$ and $\Gamma
				_{p}(X)=\Gamma _{p}(X,X)$.
				\item For two sequences of random variables $(X,\hat{X})=(X_{k},\hat{X}_{k})_{k\in \mathbb{N}}$ with laws $(\mathcal{L}(X),\mathcal{L}(\hat{X}))=(\mathcal{L}(X_{k}),\mathcal{L}(\hat{X}_{k}))_{k\in \mathbb{N}}$, we let $\Gamma _{p}(X,\hat{X}):=\max_{k}\Gamma _{p}(X_{k},\hat{X}_{k})$ and $\Gamma
				_{p}(X)=\Gamma _{p}(X,X)$.
			\end{enumerate}
		}
	\end{definition}
	
	Throughout the paper, a time discretization will be fixed, and the definitions used in the paper are given with respect to this discretization. We fix a non
	increasing sequence $\gamma _{k}\in (0,1)$ and we assume
	\begin{equation}\label{def_partition}
		t_0=0, \  t_{k}=\gamma
		_{1}+...+\gamma _{k}\to_{k\to \infty}\infty .
	\end{equation}
	
	\begin{assumption}
		\label{Sigma} Let $b>0$ and $\varepsilon>0$ be given. We assume that there
		exists a constant $C_{b,\varepsilon }>0$ and a sequence $\{\gamma _{k},k\in \mathbb{N^*}\}$ which satisfies 
		\begin{equation}
			\sigma _{b,\varepsilon }(k):=\sum_{i=1}^{k}e^{-b(t_{k}-t_{i})}\gamma
			_{i}^{1+\varepsilon }\leq C_{b,\varepsilon }\gamma _{k}^{\varepsilon
			}\quad\mbox{and}\quad \sum_{i=1}^{\infty }\gamma _{i}^{1+\varepsilon
			}<\infty .  \label{New0}
		\end{equation}
	\end{assumption}
	
	From now on, this assumption will be in force, without any further mention of it.
	
	\begin{example}
		If we choose $\gamma _{i}=Ci^{-\beta }$, $\beta \in (0,1]$ then the above
		conditions hold if $\beta (1+\varepsilon )>1$, 
		{see \cite[Lemma
			2.7 and Remark 2.8]{ABKH}.}
	\end{example}
	
	On a probability space $(\Omega ,\mathcal{F},\mathbb{P})$ we consider a filtration $%
	\mathcal{F}=(\mathcal{F}_{k},k\in \mathbb{N})$ and some $\sigma -$algebras $%
	\mathcal{G}_{k}\subset \mathcal{F}_{k}$ independent of $\mathcal{\ F}_{k-1}.$
	We think that $\mathcal{F}_{k-1}$ contains the noise before $k-1$ and $%
	\mathcal{G}_{k}$ contains the new noise (innovation). We assume that $(\Omega,\cF_0, \mathbb{P})$ and $(\Omega,\mathcal{G}_k, \mathbb{P})$, $k\ge 1$, are atomless probability spaces, so that we can define any random variable on them. The main object in our
	paper is a \textbf{sequence of random applications} $\psi =(\psi _{k})_{k\in 
		\mathbb{N}\text{ }}$ such that
	
	\begin{itemize}
		\item[$\blacktriangleright$] $\psi _{k}:\Omega \times \mathbb{R}^{d}\times 
		\mathcal{P}_{p}(\mathbb{R}^{d})\rightarrow \mathbb{R} ^{d}$ is $\mathcal{G}%
		_{k}\otimes \mathcal{B}(\mathbb{R}^{d})\otimes \mathcal{ B} (\mathcal{P}_{p}(\R^d))$
		measurable. Furthermore, we assume that for $k\neq m$, the random variables $%
		\psi _{k}$ and $\psi _{m}$ are independent: for every $x,y\in \mathbb{R}^{d}$
		and $\mu ,\nu \in \mathcal{P}_{p}(\mathbb{R}^{d}),$ the random variable $%
		\omega \mapsto \psi _{k}(\omega ,x,\mu )$ is independent of the random
		variable $\omega \mapsto\psi _{m}(\omega ,y,\nu )$;
		
		\item[$\blacktriangleright$] for every $(x,\mu )\in \mathbb{R}^{d}\times 
		\mathcal{P}_{p}(\mathbb{R}^{d}),$ $\mathbb{E}[\left\vert \psi _{k}(\omega
		,x,\mu )\right\vert ^{p}]<\infty .$
	\end{itemize}
	In what follows, we do not write $\omega $ in the notation, so $\psi
	_{k}(x,\mu )$ is shorthand for the random variable $\omega \mapsto \psi
	_{k}(\omega ,x,\mu ).$

	\begin{remark}\label{Rk_properties_psi}
		We give basic but useful properties of the sequence $\psi_{k}$.
		
		\smallskip
		
		\textbf{i) Markov property.} For each $k\in {\mathbb{N}}$, $X$ $\mathcal{F}%
		_{k}$-measurable,  $\mu \in \mathcal{P}_{p}({\mathbb{R}}^{d})$ and $h $ a measurable
		bounded function 
		\begin{align*}
			\mathbb{E}[h(\psi _{k+1}(X,\mathcal{\mu })) \mid \mathcal{F}_{k}]&=\mathbb{E}
			[h(\psi _{k+1}(X,\mathcal{\mu }))\mid X]=H(X) \\
			\mbox{with}\qquad H(x) &=\mathbb{E}[h(\psi _{k+1}(x,\mathcal{\mu }))].
		\end{align*}
		
		\smallskip
		
		\textbf{ii) Identity of laws.} Let $k\in \N$. Let $X,Y$ be independent of $\psi _{k+1}(x,\mu)$ such that $X$ has the same law as $Y$. Then, $\psi _{k+1}(X,\mu )$ has
		the same law as $\psi_{k+1}(Y,\mu ).$
		
		\smallskip
		
		\textbf{iii) Independence.} Let $X,Y\in \mathcal{F}_{k}$ be
		independent of $\psi _{k+1}(x,\mu )$. If $X$ is independent of $Y$, then $X$
		is independent of $\psi _{k+1}(Y,\mu )$.
		
	\end{remark}

	Given a random variable $X_{0}$ (initial value), we define by recurrence the following probabilistic representation  for  the $%
	\mathcal{F}$-adapted chain 
	\begin{equation}
		X_{k+1}=\psi _{k+1}(X_{k},\mathcal{L}(X_{k})).  \label{New1}
	\end{equation}%
	This is what we call a \textbf{law-dependent chain}  associated to $%
	(\psi ,X_{0})$. Usually the subscript notation refers to a time discretization as
	the one used in \cite{ABKH}. That is, $\pi
	=\{s=t_{0}<t_{1}<...<t_{n}<...\}$ such that $t_{n}\rightarrow \infty $. For
	simplicity of notation, we do not write the time dependency explicitly using
	instead the integer subscript.
	
	In order to emphasize the ``flow property" we will also use the following
	notation: for $r\leq k$ and $X$ a $\mathcal{F}_r$-measurable random variable  on $\R^d$, we construct $X_{r,r}(X)=X$ and 
	\begin{equation}
		X_{r,k+1}(X)=\psi _{k+1}(X_{r,k}(X),\mathcal{L}(X_{r,k}(X))).  \label{New1'}
	\end{equation}%
	Clearly, for $r<m<k$ we have by induction that the flow property is
	satisfied:%
	\begin{equation*}
		X_{r,k}(X)=X_{m,k}(X_{r,m}(X)).
	\end{equation*}

	\subsection{$(\protect\psi ,p)$-stationary measures}
	
	Our main assumption is the following.
	
	\begin{assumption}\label{Ass_SelfCoupling}
		Given $b>0,\alpha >0,c_{\ast }\geq 0$ and $\varepsilon >0$ we define the
		following \textbf{self coupling contraction property}: for every $k\in{%
			\mathbb{N}}$, $x,y\in \mathbb{R}^{d}$ and $\mu ,\nu \in \mathcal{P}_{p}(%
		\mathbb{R}^{d}),$ 
		\begin{align}
			W_{p}^{p}(\psi _{k}(x,\mu ),\psi _{k}(y,\nu ))\leq &(1-b\gamma
			_{k})\left\vert x-y\right\vert ^{p}+\alpha W_{p}^{p}(\mu ,\nu )\gamma
			_{k} \notag\\
			&+c_{\ast }{\left(\Gamma _{p}^{p}(\mu ,\nu ) +|x|^p +|y|^p \right)}\gamma _{k}^{1+\varepsilon }
			\label{New2}
		\end{align}
		If the above property holds, we say that $\psi =(\psi_{k})_{k\in N^*}$ is $%
		(b,\alpha ,\varepsilon )-$self coupled.
	\end{assumption}

	Under Assumption~\ref{Ass_SelfCoupling}, we have
	\begin{align}
		\E[|\psi_{k}(x,\mu)|^p]\le &2^{p-1} \E[|\psi_{k}(0,\delta_0)|^p] +2^{p-1}(|x|^p + \alpha \gamma_k \|\mu\|_p^p ) \notag \\&+2^{p-1}c_{\ast }(\Gamma _{p}^{p}(\mu ,\delta_0 )+|x|^p)\gamma _{k}^{1+\varepsilon }.\label{majo_momentp}
	\end{align}
	Therefore, we get that $\E[|\psi_{k}(X,\mu)|^p]<\infty$  when $X\in L^p$ is  $\mathcal{F}_{k-1}$-measurable (and thus independent of $\psi_k$). 
	We  can then define the application $\Theta _{k}:\mathcal{P}_{p}(\mathbb{R}^{d})\rightarrow 
	\mathcal{P}_{p}(\mathbb{R}^{d})$ by 
	\begin{equation}\label{def_thetak}
		\Theta _{k}(\mu )=\mathcal{L}(\psi _{k}(X,\mu ))\text{ with } \mu =\mathcal{L}(X) \text{ and } X \in \mathcal{F}_{k-1}. 
	\end{equation}%
	Note that using the above definition for the sequence $\{X_k\}_{k\in{\mathbb{N}}}$, defined in \eqref{New1} with $X_0 \in L^p$, we have 
	\begin{equation*}
		\mathcal{L}(X_{k+1})=\Theta _{k+1}(\mathcal{L}(X_{k})).
	\end{equation*}
	
	In the following, it is convenient to emphasize the link of these operators
	with the time grid $\pi ,$ given by the discretization $0=t_{0}<t_{1}<....<t_{n}<%
	\ldots$ For $r<k$ we define 
	\begin{equation*}
		\Theta _{t_{r},t_{k}}^{\pi }(\mu )=\Theta _{k}\circ ...\circ \Theta _{r+1}(\mu
		).
	\end{equation*}%
	In particular we identify $\Theta _{k}(\mu )=\Theta _{t_{k},t_{k+1}}^{\pi
	}(\mu ).$ And if $X_{r,k+1}(X)$ is defined as in (\ref{New1'}), then 
	\begin{equation*}
		\mathcal{L}(X_{r,k+1}(X))=\Theta _{t_{r},t_{k+1}}^{\pi }(\mathcal{L}(X)).
	\end{equation*}
	
	We now introduce a new assumption that will be important to ensure the existence of a stationary measure, see Theorem \ref{Thm_stationary} afterwards.
	\begin{assumption}
		\label{C} (Asymptotic commuting property) There exists a nonnegative sequence $\delta
		_{h,r,k}$, $h\leq r\leq k$, such that for any $\mu \in \cP_p(\R^d)$, 
		\begin{equation}
			\qquad W_{p}^{p}(\Theta _{t_{r},t_{k}}^{\pi }\circ \Theta
			_{t_{h},t_{r}}^{\pi }(\mu ),\Theta _{t_{h},t_{r}}^{\pi }\circ \Theta
			_{t_{r},t_{k}}^{\pi }(\mu ))\leq C(1+\left\Vert \mu \right\Vert
			_{p}^{p})\delta _{h,r,k},  \label{New6}
		\end{equation}
		and such that for any fixed $h$,  $\lim_{r\to\infty}\sup_{k\geq r}\delta
		_{h,r,k}=0$.
	\end{assumption}
	
	%
	
	\begin{example}
		Let $\theta $ be a time homogeneous flow. That is, $\theta$ consists in
		a family of continuous applications $\theta _{s,t}:\mathcal{P}_{p}(\mathbb{R}
		^{d})\rightarrow \mathcal{P}_{p}(\mathbb{R}^{d}),s<t$ such that $\theta
		_{s,t}=\theta _{r,t}\circ \theta _{s,r}$ for $s<r<t$ and $\theta
		_{s,t}=\theta _{0,t-s}$. Furthermore, the following commutativity property
		is satisfied: $\theta _{r,t}\circ \theta _{s,r}=\theta _{0,t-s}=\theta
		_{s,r}\circ \theta _{r,t}.$ Consider now $\Theta _{k}(\mu )=\theta
		_{t_{k},t_{k+1}}(\mu ).$ Then 
		\begin{equation*}
			\Theta _{k}\circ \Theta _{k+1}=\theta _{t_{k},t_{k+1}}\circ \theta
			_{t_{k+1},t_{k+2}}=\theta _{t_{k},t_{k+2}}=\theta _{t_{k+1},t_{k+2}}\circ
			\theta _{t_{k},t_{k+1}}=\Theta _{k+1}\circ \Theta _{k},
		\end{equation*}
		so the commutativity hypothesis \eqref{New6} is true with $\delta
		_{h,r,k}=0. $ 
	\end{example}
	
	Actually, one way to check \eqref{New6} is just to prove that,
	asymptotically as $k\to +\infty$, $\Theta_{t_0,t_k}^\pi$ is close to $%
	\theta_{t_{0},t_k}$, $\theta$ denoting a time homogeneous flow. We discuss this
	fact in Appendix \ref{app:C} and we use this criterion in the examples in
	Section \ref{sect:examples}.
	
	However, \eqref{New6} can be proved also directly, as it is done in the
	following example.

	\begin{example}
		\label{ex:1} \textrm{We construct now an example in which the commutativity
			property does not hold for fixed times but only asymptotically. This is
			based on a one dimensional non homogeneous flow already studied in \cite{ABKH}. Set  for a deterministic function $\zeta$ and $m>0$
			\begin{equation*}
				\psi  _{s,t}(x,\mu )=xe^{-m(t-s)}+\int_{s}^{t}\zeta (u)e^{-m(t-u)}du+\sigma 
				\int_{s}^{t}e^{-m(t-u)}dW_{u}  
			\end{equation*}
			and consider the associated flow $\theta _{s,t}(\mu )= \mathcal{L}(\psi
			_{s,t}(X,\mu ))$ when $X\sim \mu \in \mathcal{ P}_{2}({ \mathbb{R}})$.
			Let $h\le r\le k$. Explicit  calculations give:
			\begin{itemize}
				\item $\theta _{t_{r},t_{k}}\circ \theta _{t_{h},t_{r}}(\mu )$ is the law of $$X+\sigma \sqrt{\frac{1-e^{-2m(t_k-t_h)}}{2m}} G + \int_{t_h}^{t_k} e^{-m(t_k-u)} \zeta(u)du,$$
				with $X\sim \mu$ and $G\sim \mathcal{N}(0,1)$ independent of $X$.
				\item $\theta _{t_{h},t_{r}}\circ \theta _{t_{r},t_{k}}(\mu )$ is the law of $$X+\sigma \sqrt{\frac{1-e^{-2m(t_k-t_h)}}{2m}} G + \int_{t_h}^{t_r} e^{-m(t_r-u)} \zeta(u)du +e^{-m(t_r-t_h)}\int_{t_r}^{t_k} e^{-m(t_k-u)} \zeta(u)du .$$
			\end{itemize}
			Therefore we get
			\begin{align*}
				& W_{p}(\theta _{t_{r},t_{k}}\circ \theta _{t_{h},t_{r}}(\mu ),\theta
				_{t_{h},t_{r}}\circ \theta _{t_{r},t_{k}}(\mu )) \\
				=& \left\vert \int_{t_{h}}^{t_{r}}\zeta (u)\left(
				e^{-m(t_{r}-u)}-e^{-m(t_{k}-u)}\right) du+(e^{-m(t_{r}-t_{h})}-1)\int_{t_{r}}^{t_{k}}\zeta
				(u)e^{-m(t_{k}-u)}du\right\vert \\
				=& \left\vert (1-e^{-m(t_{k}-t_r)}) \int_{t_{h}}^{t_{r}}\zeta (u)
				e^{-m(t_{r}-u)} du+(e^{-m(t_{r}-t_{h})}-1)\int_{t_{r}}^{t_{k}}\zeta
				(u)e^{-m(t_{k}-u)}du\right\vert.
			\end{align*}%
			We first observe that if $\zeta(t)=\bar{\zeta}$ is a constant function, this distance is zero. 
			If we assume that $m>0$ and $\lim_{t\rightarrow \infty }\zeta (t)=\bar{\zeta}<\infty $,
			we get by the triangle inequality and subdividing the first integral into two integrals using the middle point $(t_r+t_h)/2$
			\begin{align*}
				& W_{p}(\theta _{t_{r},t_{k}}\circ \theta _{t_{h},t_{r}}(\mu ),\theta
				_{t_{h},t_{r}}\circ \theta _{t_{r},t_{k}}(\mu )) \\
				&\le \max_{u\ge t_h}|\zeta(u)-\bar{\zeta}| (1-e^{-m(t_k-t_r)})e^{-m(t_r-t_h)/2}(t_r-t_h)/2  \\
				&+ \max_{u\ge \frac{t_h+t_r}{2}}|\zeta(u)-\bar{\zeta}|(1-e^{-m(t_k-t_r)})\frac{1-e^{-m(t_r-t_h)/2}}{m}\\
				&+ \max_{u\ge t_r}|\zeta(u)-\bar{\zeta}| (1-e^{-m(t_r-t_h)})\frac{1-e^{-m(t_k-t_r)}}{m} \\
				&\le \max_{u\ge t_h}|\zeta(u)-\bar{\zeta}| e^{-m(t_r-t_h)/2}(t_r-t_h)/2  +\frac{2}{m}\max_{u\ge \frac{t_h+t_r}{2}} |\zeta(u)-\bar{\zeta}|.
			\end{align*}
			The above r.h.s. tends to 0 as $r,k\to\infty$, so \eqref{New6} is satisfied. \\
			In contrast, if we take $\zeta(u)=\cos(u)$, then evaluating the integrals  we have
			\begin{align*}
				& W_{p}(\theta _{t_{r},t_{k}}\circ \theta _{t_{h},t_{r}}(\mu ),\theta
				_{t_{h},t_{r}}\circ \theta _{t_{r},t_{k}}(\mu )) \\
				=& \bigg\vert \frac{1-e^{-m(t_{k}-t_r)}}{1+m^2} \left( m \cos (t_r)+\sin(t_r) - e^{-m(t_r-t_h)}(m \cos (t_h)+\sin(t_h))\right) \\
				&+\frac{e^{-m(t_{r}-t_{h})}-1}{1+m^2}\left( m \cos (t_k)+\sin(t_k) - e^{-m(t_k-t_r)}(m \cos (t_r)+\sin(t_r))\right) \bigg\vert,
			\end{align*}%
			that does not converge to $0$ for general $r,k\to \infty$. Indeed, the exponential terms go to $0$, but the sinusoids still oscillate. 
		}
	\end{example}

	In the framework of the chains that we discuss here, the notion of
	``invariant measure" is inappropriate. So we introduce:
	
	\begin{definition}
		We say that $\mu_{\ast }$ is a $(\psi ,p)$-stationary measure if for every $%
		\mu \in \mathcal{P}_{p}(\mathbb{R}^{d})$ and every $k\in \mathbb{N}$ 
		\begin{equation*}
			\lim_{n}W_{p}^{p}(\Theta_{t_{k},t_{n}}^{\pi }(\mu ),\mu _{\ast })=0.
		\end{equation*}
	\end{definition}
	
	Notice that the $(\psi ,p)$-stationary measure $\mu _{\ast }$ has to be
	unique. Moreover, due to uniqueness, if $p^{\prime }<p,$ the $(\psi ,p)$
	-stationary measure is also a $(\psi ,p^{\prime })$-stationary measure.
	
	Let us summarize the previous assumptions as follows. 
	
	\medskip

	\noindent {\bf Assumption} $\mathrm{\textbf{A}}_p(b,\alpha,\varepsilon)$:
	{\it	We assume that the  assumption~\eqref{New0} on the time discretization, the self-coupled property \eqref{New2}
		and the asymptotic commuting property  \eqref{New6} all hold.}
	
	\smallskip
	
	Then we have our first result.
	
	\begin{theorem}
		\label{Thm_stationary} Suppose that $\mathrm{\mathbf{A}}_p(b,\alpha,\varepsilon)$ is
		satisfied with $\overline{b}:=b-\alpha >0$ {and that $\sup_{k>r}\|\Theta^\pi_{t_r,t_k}(\mu) \|_p<\infty$ for any $\mu\in\mathcal{P}_p(\R^d)$.}
		Then the following statements
		hold.
		
		\medskip
		
		\textbf{A. } There exists a $(\psi ,p)-$stationary measure $\mu _{\ast }$
		for $\{X_{n}\}_{n}.$
		
		\smallskip
		
		\textbf{B. } Furthermore, for every $\mu$, there exists $C(\mu)\in \R_+$
		\begin{equation}
			W_{p}^{p}(\mathcal{L}(X_{n}),\mu _{\ast
			})=W_{p}^{p}(\Theta^\pi_{t_{0},t_{n}}(\mu ),\mu _{\ast })\leq C(\mu)(e^{-%
				\overline{b} t_{n}}+c_{\ast }\gamma _{n}^{\varepsilon
			})+\limsup_{m\rightarrow \infty }\delta _{0,n,m},  \label{New7}
		\end{equation}%
		where $\mu=\mathcal{L}(X_0)$. { Besides, if there exists $C\in \mathbb{R}_+$ such that for all $\mu$, $\sup_{0\le r \le k} \left\Vert \Theta_{t_{r},t_{k}}^{\pi
			}(\mu)\right\Vert _{p}^p\leq C(1+\left\Vert \mu\right\Vert
			_{p}^p) $ (this holds when the assumption~\eqref{New0bis} holds), we can take $C(\mu)\le \tilde{C}(1+\|\mu\|_p^p)$ for some constant $\tilde{C}\in \R_+$. }
	\end{theorem}
	
	The proof of Theorem \ref{Thm_stationary} is given in Section \ref%
	{sect:proof-prop}.
	Notice that Theorem \ref{Thm_stationary} requires the  boundedness of the moments of the chain. This requirement will also be needed later on. A sufficient condition is provided by the Foster--Lyapunov criterion stated below, whose proof follows immediately by a recurrence argument.

	\begin{proposition}
		\label{lem:29} 
		Suppose  that there exist $\lambda, C>0$ and $r\in \N$  such that for every $\mu \in 
		\mathcal{P}_{p}(\R^d)$and every $k>r$
		\begin{equation}  \label{New0bis}
			\qquad\qquad\qquad \left\Vert \Theta_{k}(\mu )\right\Vert _{p}^{p}\leq
			(1-\lambda \gamma _{k})\left\Vert \mu \right\Vert _{p}^{p}+C\gamma _{k}.
		\end{equation}
		Then for every $k\geq r$,
		\begin{equation*}  
			{\mathbb{E}}[|X_{r,k}(X)|^p]=\left\Vert \Theta_{t_{r},t_{k}}^{\pi
			}(\mu)\right\Vert _{p}^p\leq e^{-\lambda(t_{k}-t_r)}\left\Vert \mu\right\Vert
			_{p}^p+C_{\lambda}.
		\end{equation*}
		Here $\mathcal{L}(X)=\mu$ and $C_{\lambda}$ is a constant depending on $\lambda$.
	\end{proposition}

	\subsection{The ergodic measure as an approximation of the $(\protect\psi,p)$--stationary measure}

	In practice, in order to numerically compute the $(\psi ,p)-$stationary
	measure $\mu_*$, one should simulate the sequence $\{X_k\}_{k\in{\mathbb{N}}}$.
	However, the simulation of $X_{k+1}$ would involve the computation of the law of $%
	X_{k}$ that is in general not known. Usually approximation of this law is required in order to have a computable scheme. One possible way is to use particle systems. A second way to tackle this problem, which is computationally efficient, is to
	introduce ergodic sums, as pointed out in~\cite{ABRS}. This is what we study in the present paper. 
	
	With this in mind, we define a new chain $Y$, in which the dynamics use an ergodic approximation of $\mathcal{L}(X_{k})$  as follows. For $k\ge 0$ and a random variable $Y_0$ with finite $p$-moments
	we define
	\begin{equation}
		Y_{k+1}=\psi _{k+1}(Y_{k},\rho _{k}(Y)),\quad \mbox{with}\quad\rho _{k}(Y)=%
		\frac{1}{t_{k}}\sum_{i=1}^{k}\gamma _{i}\delta _{Y_{i}},  \label{Y}
	\end{equation}%
	with the convention  $\rho _{0}(Y)=\delta _{Y_{0}}$.
	The law of $Y_k$ is thus approximated by $\rho _{k}(Y)$ that can be computed.
	We stress that $\rho _{k}(Y)$ is a random probability measure that can be
	simulated in a direct way. Throughout this section we let Assumption~\ref{Ass_SelfCoupling} and $Y_0\in L^p$ hold for some $p\ge 1$. Then using~\eqref{majo_momentp} and noting that $\rho _{k}(Y)$ is independent of $\psi_{k+1}$ and $\|\rho _{k}(Y)\|_p^p=\frac{1}{t_{k}}\sum_{i=1}^{k}\gamma _{i} |Y_{i}|^p$, we get by induction that  $Y_k\in L^p$ for all $k\ge 0$.   
	Our main result (Theorem \ref{MAIN0}) states that, as $k\to\infty$, $%
	\rho _{k}(Y)$ is an approximation of $\mu_*$ in a $W_p$ sense, and we can also give an estimate of the error.
	
	In order to state the result, we need the following assumption. 
	
	\begin{assumption}
		\label{Ap} There exist $b>0$, $c_*>0$, $\varepsilon>0$ such that for every $%
		\mu \in \mathcal{P}_{p}({\mathbb{R}} ^{d})$ and $x,y\in {\mathbb{R}}^{d}$, 
		\begin{align}
			\mathbb{E}\left [\left\vert \psi _{i}(x,\mu )-\psi _{i}(y,\mu )\right\vert
			^{p}\right] &\leq (1-b\gamma _{i})\left\vert x-y\right\vert ^{p}+c_{\ast
			}(\Gamma_{p}^{p}(\mu )+|x|^p+|y|^p)\gamma _{i}^{1+\varepsilon }  \label{na1a} \\
			&\leq e^{-b\gamma _{i}}\left\vert x-y\right\vert ^{p}+c_{\ast
			}(\Gamma_{p}^{p}(\mu )+|x|^p+|y|^p)\gamma _{i}^{1+\varepsilon }.  \notag
		\end{align}
	\end{assumption}
	Notice that the basic contraction property (\ref{New2}) requires  an
	upper bound of $W_{p}^{p}(\psi _{i}(x,\mu ),\psi _{i}(y,\nu ))$ for every $%
	\mu $ and $\nu$. In contrast, we consider here the same probability measure $\mu =\nu$ and ask for an upper bound of the $L^{p}$ distance (instead of
	the Wasserstein distance). In particular, $\psi _{i}(x,\mu )$ and $\psi
	_{i}(y,\mu )$ are supposed to be defined on the same probability space.
	
	Then, our main result is the following.

	\begin{theorem}
		\label{MAIN0} Suppose that $\mathrm{\mathbf{A}}_p(b,\alpha,\varepsilon)$ is verified with $\overline{b}:=b-\alpha >0$ and $\delta_{h,r,k}=\gamma_r^\varepsilon$ in~\eqref{New6}.  Let also \eqref{na1a} hold with the
		same parameters $p$, $b$ and $\varepsilon$. Moreover, suppose that, for some $p'>p$,  $\sup_{k\ge r}\|\Theta^\pi_{t_r,t_k}(\mu) \|_{p'}<\infty$ for any $\mu\in\mathcal{P}_{p'}(\R^d)$.
		Let $\hat{p}=\frac{p(p^{\prime }-p)}{(d+p+1/2)(p^{\prime
			}-p)+(d+p)p}$.
		
		We also assume that there exists $\eta >0$ such that $\gamma
		_{k}=O(t_{k}^{-\eta })$. Then, for any $\zeta \in (0,1-\frac{\alpha }{b}
		)\cap \big(0,\min \left( \frac{\hat{p}}{2} ,\varepsilon \eta \right) \big]$ and $Y_0$ with law $\mu$, we have 
		\begin{equation}
			\mathbb{E}[W_{p}^{p}(\rho _{k}(Y),\mu _{\ast })]\leq Ct_{k}^{-\zeta },
			\label{na10-0}
		\end{equation}%
		where $\mu _{\ast }$ is the $(\psi ,p)$-stationary measure.
		
		Finally, suppose that $\gamma _{k}=\frac{1}{(1+k)^{\beta }}$ with $%
		1\geq\beta >\frac{1}{1+\varepsilon }$. Then, for every 
		\begin{equation*}
			\zeta \in (0,1-\frac{\alpha }{b})\cap \big(0,\frac{\hat{p}}{2}\big],
		\end{equation*}
		one has 
		\begin{align}
			\mbox{if $\beta=1$:}&\quad \mathbb{E}[W_{p}^{p}(\rho _{k}(Y),\mu _{\ast
			})]\leq \frac{C}{(\ln k)^\zeta},  \label{na10-01} \\
			\mbox{if $1>\beta>\frac 1{1+\varepsilon}$:} &\quad \mathbb{E}[W_{p}^{p}(\rho
			_{k}(Y),\mu _{\ast })]\leq \frac{C}{k^{(1-\beta)\zeta}}.  \label{na10-02}
		\end{align}
	\end{theorem}
	
	\section{Proof of Theorem \protect\ref{Thm_stationary} (Convergence of $\mathcal{L}(X_k)$).}
	
	\label{sect:proof-prop} We first study a preliminary result.
	
	\begin{lemma}
		Let $X$ and $Y$ be two random variables which are independent of the random
		variables $\psi _{k}(x,\mu )$ for every $x\in {\mathbb{R}}^{d}$ and $\mu \in 
		\mathcal{P}_{p}(\mathbb{R}^{d}).$ If (\ref{New2}) holds, then 
		\begin{align}
			W_{p}^{p}(\psi _{k}(X,\mu ),\psi _{k}(Y,\nu ))\leq &(1-b\gamma
			_{k})W_{p}^{p}(X,Y)+\alpha W_{p}^{p}(\mu ,\nu )\gamma _{k} \notag\\
			&+c_{\ast }\left(\Gamma
			_{p}^{p}(\mu ,\nu )+ \Gamma
			_{p}^{p}(X ,Y ) \right)\gamma _{k}^{1+\varepsilon }.  \label{New2'}
		\end{align}
		And, as an immediate consequence of (\ref{New2'}),%
		\begin{equation}
			W_{p}^{p}(\Theta _{k}(\mu ),\Theta _{k}(\nu ))\leq (1-(b-\alpha )\gamma
			_{k})W_{p}^{p}(\mu ,\nu )+2c_{\ast }\Gamma _{p}^{p}(\mu ,\nu )\gamma
			_{k}^{1+\varepsilon }.  \label{New2''}
		\end{equation}
	\end{lemma}
	
	\begin{proof}
		Let $\eta_{X,Y}(dx,dy)$ be an optimal coupling for $W_p$ of the laws
		of $X$ and $Y$. Moreover, for each fixed $x,y$ and $\mu ,\nu $ let $\Pi
		_{x,y}^{\mu ,\nu }(dz_{1},dz_{2})$ be the optimal coupling of the laws of $%
		\psi _{k}(x,\mu )$ and of $\psi _{k}(y,\nu )$. We define the probability
		measure 
		\begin{equation*}
			\Gamma (dz_{1},dz_{2})= \int \int \Pi _{x,y}^{\mu ,\nu }(dz_{1},dz_{2})\eta
			_{X,Y}(dx,dy).
		\end{equation*}%
		Note that by~\cite[Corollary 5.22]{Villani}, $(x,y)\mapsto
		\Pi _{x,y}^{\mu ,\nu }(dz_{1},dz_{2})$ is measurable and therefore $\Gamma (dz_{1},dz_{2})$ is well defined.
		
		Let us check that the marginals of $\Gamma $ are the laws of $\psi
		_{k}(X,\mu )$ and $\psi _{k}(Y,\nu ):$%
		\begin{align*}
			\int \int \Phi (z_{1})&\Gamma (dz_{1},dz_{2}) =\int \int \int \int \Phi
			(z_{1})\Pi _{x,y}^{\mu ,\nu }(dz_{1},dz_{2})\eta _{X,Y}(dx,dy) \\
			&=\int \int \E[\Phi (\psi _{k}(x,\mu ))]\eta _{X,Y}(dx,dy)=\E[\Phi (\psi
			_{k}(X,\mu ))].
		\end{align*}%
		We have used that the first marginal of $\Pi _{x,y}^{\mu ,\nu
		}(dz_{1},dz_{2})$ is the law of $\psi _{k}(x,\mu ).$\ Then we use the fact
		that the first marginal of $\eta _{X,Y}(dx,dy)$ coincides with the law of $%
		X, $ and moreover, the fact that $X$ is independent of $\psi _{k}(x,\mu ).$
		So $\Gamma $ is a coupling (maybe not optimal) of the laws of $\psi
		_{k}(X,\mu )$ and $\psi _{k}(Y,\nu )$. We use (\ref{New2}) and we write now%
		\begin{align*}
		W_{p}^{p}(\psi _{k}(X,\mu ),\psi _{k}(Y,\nu )) &\leq \int \int \left\vert
	z_{1}-z_{2}\right\vert ^{p}\Gamma (dz_{1},dz_{2}) \\
	&=\int \int \int \int \left\vert z_{1}-z_{2}\right\vert ^{p}\Pi _{x,y}^{\mu
		,\nu }(dz_{1},dz_{2})\eta _{X,Y}(dx,dy) \\
	&=\int \int W_{p}^{p}(\psi _{k}(x,\mu ),\psi _{k}(y,\nu ))\eta _{X,Y}(dx,dy).
		\end{align*}
		Then, we use Assumption~\ref{Ass_SelfCoupling} and get
		\begin{align*}
		& W_{p}^{p}(\psi _{k}(X,\mu ),\psi _{k}(Y,\nu ))\\
		&\leq \int \int (1-b\gamma _{k})\left\vert x-y\right\vert ^{p}\eta
		_{X,Y}(dx,dy)+\alpha W_{p}^{p}(\mu ,\nu ) \gamma
		_{k}+c_{\ast }(\Gamma _{p}^{p}(\mu ,\nu )+\Gamma _{p}^{p}(X ,Y ))\gamma
		_{k}^{1+\varepsilon } \\
		&=(1-b\gamma _{k})W_{p}^{p}(X,Y)+\alpha W_{p}^{p}(\mu ,
		\nu )\gamma _{k}+c_{\ast }(\Gamma _{p}^{p}(\mu ,\nu )+\Gamma _{p}^{p}(X ,Y ))\gamma
		_{k}^{1+\varepsilon },
		\end{align*}
		since $\eta_{X,Y}$ is an optimal coupling of $X$ and $Y$. 
	\end{proof}
	
	We are now ready for the
	
	\begin{proof}[Proof of Theorem~\protect\ref{Thm_stationary}]
		
		\textbf{Step 1:} $\{\Theta^\pi_{t_h,t_n}\}_n$ is a Cauchy sequence.
		
		\smallskip
		
		As $\psi $ is self-coupled, by (\ref{New2''}), for any $\mu ,\nu \in 
		\mathcal{P}_p(\mathbb{R}^{d})$ 
		\begin{align*}
			W_{p}^{p}(\Theta _{k}(\mu ),\Theta _{k}(\nu )) &\leq (1-\overline{b}\gamma
		_{k})W_{p}^{p}(\mu ,\nu )+2c_{\ast }\Gamma _{p}^{p}(\mu ,\nu )\gamma
		_{k}^{1+\varepsilon }  \notag \\
		&\leq e^{-\overline{b}\gamma _{k}}W_{p}^{p}(\mu ,\nu )+c_{\ast }\Gamma
		_{p}^{p}(\mu ,\nu )\gamma _{k}^{1+\varepsilon }.
\end{align*}
		Iterating this inequality and using (\ref{New0}) we get as well as the bounds on the norms of $\Theta _{t_{h},t_{k}}^{\pi }(\mu ),\Theta _{t_{h},t_{k}}^{\pi
		}(\nu )$ for $k\geq h$%
		\begin{align}
			&W_{p}^{p}(\Theta _{t_{h},t_{k}}^{\pi }(\mu ),\Theta _{t_{h},t_{k}}^{\pi
			}(\nu )) \notag\\&\leq e^{-\overline{b}(t_{k}-t_{h})}W_{p}^{p}(\mu ,\nu )+2c_{\ast
			}\sum_{i=h+1}^{k}e^{-\overline{b}(t_{k}-t_{i})}\Gamma_{p}^{p}(\Theta _{t_{h},t_{i-1}}^{\pi }(\mu ),\Theta _{t_{h},t_{i-1}}^{\pi
			}(\nu )) \gamma _{i}^{1+\varepsilon } \notag \\
			&\leq e^{-\overline{b}(t_{k}-t_{h})}(\left\Vert \mu \right\Vert
			_{p}^{p}+\left\Vert \nu \right\Vert _{p}^{p})+C(\mu,\nu)c_{\ast }\gamma
			_{k}^{\varepsilon } \label{estimate_Thetapi}.
		\end{align}
		Here, we let $C(\mu,\nu):=1+\sup_{k>r}\|\Theta^\pi_{t_r,t_k}(\mu) \|^p_p+\sup_{k>r}\|\Theta^\pi_{t_r,t_k}(\nu) \|^p_p$ and $C(\mu):=C(\mu,\mu)$.
		
		We fix $h\in{\mathbb{N}}$ and we take $m>n\geq h.$ The asymptotic
		commuting property \eqref{New6} and then the inequality %
		\eqref{estimate_Thetapi}  give 
		\begin{align*}
			&W_{p}^{p}(\Theta _{t_{h},t_{n}}^{\pi }(\mu ),\Theta _{t_{h},t_{m}}^{\pi
			}(\mu )) =W_{p}^{p}(\Theta _{t_{h},t_{n}}^{\pi }(\mu ),\Theta
			_{t_{n},t_{m}}^{\pi }\circ \Theta _{t_{h},t_{n}}^{\pi }(\mu )) \\
			&\leq W_{p}^{p}(\Theta _{t_{h},t_{n}}^{\pi }(\mu ),\Theta
			_{t_{h},t_{n}}^{\pi }(\Theta _{t_{n},t_{m}}^{\pi }(\mu )))+C(1+\left\Vert \mu \right\Vert
			_{p}^{p})\delta _{h,n,m} \\
			&\leq e^{-\overline{b}(t_{n}-t_{h})}(\left\Vert \mu \right\Vert
			_{p}^{p}+\left\Vert \Theta _{t_{n},t_{m}}^{\pi }(\mu ))\right\Vert
			_{p}^{p})+C(\mu,\Theta _{t_{n},t_{m}}^{\pi }(\mu )) c_{\ast }\gamma _{n}^{\varepsilon }+C(1+\left\Vert \mu \right\Vert
			_{p}^{p}) \delta _{h,n,m} \\
			&\leq C(\mu)(e^{-\overline{b}(t_{n}-t_{h})}+c_{\ast }\gamma _{n}^{\varepsilon
			})+C(1+\left\Vert \mu \right\Vert
			_{p}^{p}) \sup_{m\ge n}\delta _{h,n,m}\longrightarrow_{n\rightarrow \infty} 0 .
		\end{align*}
		So the laws $\Theta _{t_{h},t_{n}}^{\pi }(\mu )$, $n\in \mathbb{N}$ form a
		Cauchy sequence under $W_{p}.$ We define 
		\begin{equation*}
			\mu _{\ast }(h,\mu )=\lim_{n}\Theta _{t_{h},t_{n}}^{\pi }(\mu ).
		\end{equation*}
		Let us note here that when the property $\sup_{0\le r \le k} \left\Vert \Theta_{t_{r},t_{k}}^{\pi
			}(\mu)\right\Vert _{p}^p\leq C(1+\left\Vert \mu\right\Vert
			_{p}^p) $  holds, we can take $C(\mu)=\tilde{C}(1+\|\mu\|_p^p)$ for some constant $\tilde{C}$.
		
		\textbf{Step 2:} $\mu _{\ast }(h,\mu )$ does not depend on $h$ and $\mu$.
		
		\smallskip
		
		We use \eqref{estimate_Thetapi} and we get 
		\begin{align*}
W_{p}^{p}(\Theta _{t_{0},t_{n}}^{\pi }(\mu ),\Theta _{t_{h},t_{n}}^{\pi
}(\nu )) &=W_{p}^{p}(\Theta _{t_{h},t_{n}}^{\pi }(\Theta
_{t_{0},t_{h}}^{\pi }(\mu )),\Theta _{t_{h},t_{n}}^{\pi }(\nu )) \\
&\leq
e^{-\overline{b}(t_{n}-t_{h})}(\left\Vert \Theta _{t_{0},t_{h}}^{\pi
}(\mu ))\right\Vert _{p}^{p}+\left\Vert \nu \right\Vert _{p}^{p})+C(\Theta _{t_{0},t_{h}}^{\pi
}(\mu ),\nu)\gamma
_{n}^{\varepsilon }.
		\end{align*}
		Passing to the limit with $n\rightarrow \infty $ we obtain%
		\begin{equation*}
			W_{p}^{p}(\mu _{\ast }(0,\mu ),\mu _{\ast }(h,\nu
			))=\lim_{n}W_{p}^{p}(\Theta _{t_{0},t_{n}}^{\pi }(\mu ),\Theta
			_{t_{h},t_{n}}^{\pi }(\nu ))=0.
		\end{equation*}
		
		\textbf{Step 3:} proof of \textbf{B.}
		
		\smallskip
		
		Using \eqref{New6} and then \eqref{estimate_Thetapi}, 
		\begin{align*}
&W_{p}^{p}(\Theta _{t_{h},t_{n}}^{\pi }(\mu ),\mu _{\ast })
=W_{p}^{p}(\Theta _{t_{h},t_{n}}^{\pi }(\mu ),\lim_{m}\Theta
_{t_{h},t_{m}}^{\pi }(\mu )) \\
&\leq \lim_{m}W_{p}^{p}(\Theta _{t_{h},t_{n}}^{\pi }(\mu ),\Theta
_{t_{h},t_{n}}^{\pi }(\Theta _{t_{n},t_{m}}^{\pi }(\mu
))+C(1+\|\mu \|_p^p)\limsup_{m\rightarrow \infty }\delta _{h,n,m} \\
&\leq C(\mu)  (e^{-\overline{b}(t_{n}-t_{h})}+c_{\ast }\gamma _{n}^{\varepsilon
})+C(1+\|\mu\|_p^p)\limsup_{m\rightarrow \infty }\delta _{h,n,m}.
		\end{align*}
	
		%
		%
		%
		%
		%
		%
		%
		%
		%
		%
		%
	\end{proof}

	Now, we start looking into a functional that is closer to the averages that
	we will use for describing the simulation method. For this, as a first step,
	we consider for $X=(X_{k})_{k}$ defined in \eqref{New1} the ergodic sums:%
	\begin{equation*}
		\mathcal{L}_{n}(X)=\frac{1}{t_{n}}\sum_{i=1}^{n}\gamma _{i}\mathcal{L}%
		(X_{i}).
	\end{equation*}
	
	\begin{corollary}\label{cor_speeds_X}
		Assume that the hypotheses of Theorem \ref{Thm_stationary} are satisfied and let $%
		\mu_*$ denote the $(\psi,p)$--stationary measure. Suppose in addition that %
		\eqref{New6} is true with $\delta
		_{h,r,k} = \gamma _{r}^{\varepsilon }$. Then 
		\begin{equation}
			W_{p}^{p}(\mathcal{L}(X_{n}),\mu _{\ast })\leq C{\gamma }_{n}^{\varepsilon },
			\label{NEW8'}
		\end{equation}%
		and 
		\begin{equation}
			W_{p}^{p}(\mathcal{L}_{n}(X),\mu _{\ast })\leq \frac{C}{t_{n}}.
			\label{NEW9p}
		\end{equation}
	\end{corollary}
	
	\begin{proof}
		(\ref{NEW8'}) is a consequence of (\ref{New7}) and then, by (%
		\ref{New0}) 
		\begin{equation*}
			W_{p}^{p}(\mathcal{L}_{n}(X),\mu _{\ast })\leq \frac{C}{t_{n}}%
			\sum_{i=1}^{n}\gamma _{i}W_{p}^{p}(\mathcal{L}(X_{i}),\mu _{\ast })\leq 
			\frac{C}{t_{n}}\sum_{i=1}^{n}\gamma _{i}^{1+\varepsilon }\leq \frac{C}{t_{n}}%
			.
		\end{equation*}%
	\end{proof}

	\section{Time averages and proof of Theorem~\ref{MAIN0} (Convergence of $\rho_k(Y)$)}
	
	\label{sect:proof-thm}
	
	\subsection{Two types of ergodic measures}
	
	In the previous section, we obtained several properties for the sequence of $%
	(b,\alpha,\varepsilon)-$self coupled random applications $\psi=\{\psi
	_{k}\}_k$. But in order to go further, we need to assume \eqref{na1a}, 
	and we define the following two types of ergodic measures:%
	\begin{equation}
		\rho _{k}(X)=\frac{1}{t_{k}}\sum_{i=1}^{k}\gamma _{i}\delta _{X_{i}}\quad %
		\mbox{and}\quad \mathcal{L}_{k}(X)=\frac{1}{t_{k}}\sum_{i=1}^{k}\gamma _{i}%
		\mathcal{\ L}(X_{i}),  \label{na7}
	\end{equation}%
	where $\{X_k\}_{k\in{\mathbb{N}}}$ is the usual chain constructed through $%
	\psi $ as in \eqref{New1}. Notice that $\rho _{k}(X)$ and $\mathcal{L}_{k}(X)
	$ are probability measures, the former being random while the latter is not. Therefore $W_{p}^{p}(\rho _{k}(X),\mathcal{L}_{k}(X))$ is a random variable. This kind of quantities will appear from now.
	Our next aim is to estimate the expectation of this Wasserstein distance.
	
	\begin{theorem}
		\label{thm_ergo} Let $p^{\prime }>p\geq 1$. Suppose that \eqref{na1a} holds
		and $\Gamma_{p^{\prime }}(X)<\infty $ (see Definition~\ref{def_moment}). Then,
		we have for $\hat{p}:=\frac{%
			p(p^{\prime }-p)}{(d+p+1/2)(p^{\prime }-p)+(d+p)p}\in (0,1)$
		\begin{equation}
			{\mathbb{E}}[W_{p}^{p}(\rho _{k}(X),\mathcal{L}_{k}(X))]\leq
			C\Gamma_{p^{\prime }}(X)^{p^{\prime }}t_{k}^{-\hat{p}/2}.  \label{na6}
		\end{equation}%
		Furthermore, assume that $\psi =(\psi _{k})_{k\in \mathbb{N}}$ is $(b,\alpha
		,\varepsilon )-$self coupled, with $\overline{b}:=b-\alpha >0$,
		and suppose that \eqref{New0} 
		holds with the same parameters $b$ and $\varepsilon$. Finally, assume that
		the asymptotic commuting property \eqref{New6} is satisfied with $%
		\delta_{0,n,m}=\gamma _{n}^{\varepsilon }$. Then, as a consequence of (\ref{NEW9p}), 
		\begin{equation}
			{\mathbb{E}}[W_{p}^{p}(\rho _{k}(X),\mathcal{\mu }_{\ast })]=O\left( t_{k}^{-\hat{p}/2}
			\right)
			\label{na6'}
		\end{equation}
	\end{theorem}

	The main ingredient of the proof is the following lemma:
	
	\begin{lemma}\label{lemma_Wbar2}
		Let $p\ge 1$, $X_0 \in L^p$ and $\{X_k\}_{k\in \N}$ be the chain defined by~\eqref{New1}. 
		Suppose that \eqref{na1a}  
		holds and $\Gamma_{p}(X)=\left(1+\max_{k\in {\mathbb{N}}}\mathbb{E}%
		[\left\vert X_{k}\right\vert ^{p}]\right)^{1/p}<\infty $.\newline
		\textbf{A. }  Then,
		there exists a constant $C\in {\mathbb{R}}_{+}$ depending on the sequence $%
		(\gamma _{n})_n$, $c_{\ast }$ and $p$ such that for every $i<j$ and every $%
		\varphi \in C_{b}^{1}({\mathbb{R}}^{d})$ 
		\begin{equation*}
			\mathbb{E}[\left\vert \mathbb{E}[\varphi (X_{j})-\mathbb{E}[\varphi
			(X_{j})]\mid X_{i}]\right\vert] \leq C\left\Vert \nabla \varphi \right\Vert
			_{\infty }\Gamma_{p}(X) \left( e^{-\frac{b}{p}%
				(t_{j}-t_{i})}+ \gamma _{j}^{\varepsilon /p}\right)  
		\end{equation*}%
		\textbf{B. }Moreover, suppose that \eqref{New0}
		is satisfied with $\varepsilon$ replaced by $\varepsilon /p$ (in particular,
		this yields $\sum_{i=1}^{\infty }\gamma _{i}^{1+\varepsilon /p}<\infty $). Then, 
		\begin{equation}
			\mathbb{E}\left[ \left( \int_{{\mathbb{R}}^{d}}\varphi (x)(\rho _{k}(X)(dx)-%
			\mathcal{L}_{k}(X)(dx))\right) ^{2}\right] \leq C\Gamma_{p}(X)\left\Vert
			\varphi \right\Vert _{\infty }\left\Vert \nabla \varphi \right\Vert _{\infty
			}t_{k}^{-1},  \label{na4}
		\end{equation}%
		with a constant $C$ that only depends on $b$, $\varepsilon $, $c_{\ast }$,
		the sequence $(\gamma _{n})$ and $p$.
	\end{lemma}
	
	\begin{proof} \textbf{A.} 
		\textbf{Step 1.} We fix $i<j$. We may assume, possibly on an extension of
		the probability space $\Omega$ that we still denote by $\Omega$, for 
		notational simplicity, that we have a random variable $\widehat{X}_{i}$ which is
		independent of $\mathcal{F}$ and has the same law as $X_{i}$. We denote this
		common law by $\mu _{i}.$ We denote $X_{i,i}=\widehat{X}_{i}$ and we
		construct by recurrence the sequence $X_{i,k+1}=\psi _{k+1}(X_{i,k},\mathcal{%
			\ L}(X_{i,k}))$, $k\geq i $. Using Remark~\ref{Rk_properties_psi}~\textbf{ii)} we check that $X_{i,i+1}$
		has the same law as $X_{i+1},$ which we denote by $\mu _{i+1}.$ Using the
		same argument we check that for every $k>i,$ the law of $X_{i,k}$ coincides
		with the law of $X_{k},$ denoted by $\mu _{k}.$
		
		We fix now a bounded and measurable function $h_{j}:{\mathbb{R}}%
		^{d}\rightarrow {\mathbb{R}}$ and we construct $h_{j-1}(x)=\mathbb{E}%
		[h_{j}(\psi _{j}(x,\mu _{j-1}))].$ Using this recurrence we also define in
		the same way $h_{r},r=j-1,j-2,...,i.$
		
		For $k\geq i$ we denote $\mathcal{F}_{k}^{\prime }$ =$\mathcal{F}_{k}\vee
		\sigma (\widehat{X}_{i}).$\ Notice that $\psi _{j}(x,\mu _{j-1})$ is
		independent of $\mathcal{F}_{j-1}^{\prime }.$ \ Using Remark~\ref{Rk_properties_psi}~\textbf{i)} we check that 
		\begin{align*}
		\mathbb{E}[h_{j}(X_{i,j})\mid \mathcal{F}_{j-1}^{\prime }] &=\mathbb{E}%
		[h_{j}(\psi _{j}(X_{i,j-1},\mu _{j-1}))\mid \mathcal{F}_{j-1}^{\prime
		}]=h_{j-1}(X_{i,j-1}), \\
		\mathbb{E}[h_{j}(X_{j})\mid \mathcal{F}_{j-1}] &=\mathbb{E}[h_{j}(\psi
		_{j}(X_{j-1},\mu _{j-1}))\mid \mathcal{F}_{j-1}]=h_{j-1}(X_{j-1}).
		\end{align*}
		So, using recursively this argument 
		\begin{align*}
			\mathbb{E}[h_{j}(X_{i,j})\mid X_{i,i}] &=\mathbb{E}[\mathbb{E}%
			[h_{j}(X_{i,j})\mid \mathcal{F}_{j-1}^{\prime }]\mid X_{i,i}]=\mathbb{E}%
			[h_{j-1}(X_{i,j-1})\mid X_{i,i}] \\
			&=\ldots =h_{i}(X_{i,i}),
		\end{align*}
		and in a similar way $\mathbb{E}[h_{j}(X_{j})\mid X_{i}]=h_{i}(X_{i}).$
		Since $X_{i}$ has the same law as $X_{i,i}$ we conclude that $\mathbb{E}%
		[h_{j}(X_{i,j})\mid X_{i,i}]$ has the same law as $\mathbb{E}%
		[h_{j}(X_{j})\mid X_{i}].$ Since $X_{j}$ has the same law as $X_{i,j},$ we
		may define%
		\begin{equation*}
			h_{j}(x)=\varphi (x)-\mathbb{E}[\varphi (X_{j})]=\varphi (x)-\mathbb{E}%
			[\varphi (X_{i,j})],
		\end{equation*}%
		and using the above identity of laws we obtain 
		\begin{equation*}
			\mathbb{E}\left[\left|\mathbb{E}\big[\varphi (X_{j})-\mathbb{E}[\varphi (X_{j})]\mid
			X_{i}\big]\right|^{p}\right]=\mathbb{E}\left[\left|\mathbb{E}\big[\varphi (X_{i,j})-\mathbb{E}[\varphi
			(X_{i,j})]\mid X_{i,i}\big]\right|^{p}\right].
		\end{equation*}
		
		\textbf{Step 2.} We denote $$\xi _{j}=\varphi (X_{j})-\mathbb{E}[\varphi
		(X_{j})], \ \xi _{i,j}=\varphi (X_{i,j})-\mathbb{E}[\varphi (X_{i,j})].$$
		Since $X_{i}$ is independent of $\widehat{X}_{i},$ $\mathbb{%
			\ E}[\xi _{j}\mid X_{i}]$ and $\mathbb{E}[\xi _{i,j}\mid \widehat{X}_{i}]$
		are also independent and we have $\mathbb{E}[\mathbb{E}[\xi_{i,j}\mid 
		\widehat{X}_{i}]\mid X_i]=\mathbb{E}[\mathbb{E}[\xi_{i,j}\mid \widehat{X}%
		_{i}]]=0$. By Jensen's inequality, we get for $p\ge 1$ 
		\begin{align*}
				\mathbb{E}|[\mathbb{E}[\xi _{j} \mid X_{i}]|^{p}]&=\mathbb{E}[|
			\mathbb{E} [
			\mathbb{E}[\xi _{j} \mid X_{i}] - \mathbb{E}[\xi_{i,j}\mid 
			\widehat{X}_{i}]
			\mid X_{i}]|^{p}]  \notag \\
			&\le \mathbb{E}[\mathbb{E}[|\mathbb{E}[\xi _{j} \mid X_{i}] - \mathbb{E}
			[\xi_{i,j}\mid \widehat{X}_{i}] |^{p} \mid X_{i}]]  \notag \\
			&=\mathbb{E}[|\mathbb{E}[\xi _{j} \mid X_{i}] - \mathbb{E}[\xi_{i,j}\mid 
			\widehat{X}_{i}] |^{p}]. 
		\end{align*}
		Notice now that $X_{i}$ is independent of $\widehat{X}_{i}=X_{i,i}$ and is
		also independent of $\psi _{i+1}(x,\mu _{i})$. So it is independent of $%
		X_{i,i+1}=\psi _{i+1}(X_{i,i},\mu _{i})$, see Remark~\ref{Rk_properties_psi}~\textbf{iii)}. We also know that $X_{i}$ is $%
		\mathcal{F}_{i+1}$ measurable, so it is independent of $\psi _{i+2}(x,\mu
		_{i+1}).$ It follows that $X_{i}$ is independent of $X_{i,i+2}=\psi
		_{i+2}(X_{i,i+1},\mu _{i+1}).$ We continue like this and we conclude that $%
		X_{i}$ is independent of any $X_{i,k},k\geq i,$ and consequently of $\xi
		_{i,k}$, for $k\geq i$. It is also independent of  $\widehat{X}_{i}$, which  implies $\mathbb{E}[\xi _{i,j}\mid \widehat{X}%
		_{i}]=\mathbb{E}[\xi _{i,j}\mid X_{i},\widehat{X}_{i}]$. Besides, by definition $\widehat{X}_{i}$ is independent of $X_{i}$ and 
		any $\xi _{k}$, $k\geq i$, so $\mathbb{E}[\xi _{j}\mid X_{i}]=\mathbb{E}[\xi
		_{j}\mid X_{i},\widehat{X}_{i}]$. 

		We conclude that 
		\begin{align}
				\mathbb{E}[|\mathbb{E}[\xi _{j} \mid X_{i}]|^{p}]&\le\mathbb{E}[|\mathbb{E}
			[\xi _{j}\mid X_{i},\widehat{X}_{i}]-\mathbb{E}[\xi _{i,j}\mid X_{i}, 
			\widehat{X}_{i}]|^{p}]  \notag \\
			&\leq \mathbb{E}\left[\left\vert \xi _{j}-\xi _{i,j}\right\vert
			^{p}\right]\leq \left\Vert \nabla \varphi \right\Vert _{\infty }^{p}\mathbb{%
				E }\left[\left\vert X_{j}-X_{i,j}\right\vert ^{p}\right].  \label{ineq_p}
		\end{align}
				
		\textbf{Step 3.} Recall that $X_{j}$ has the same law as $X_{i,j}$ which is
		denoted by $\mu _{j}.$ We use  the
		contraction property \eqref{na1a}, and we obtain
		\begin{align*}
			\mathbb{E}\left[ \left\vert X_{j+1}-X_{i,j+1}\right\vert ^{p}\right] &=%
			\mathbb{E}\left[ \left\vert \psi _{j+1}(X_{j},\mathcal{L}(X_{j}))-\psi
			_{j+1}(X_{i,j},\mathcal{L}(X_{i,j}))\right\vert ^{p}\right] \\
			&=\mathbb{E}\left[ \left\vert \psi _{j+1}(X_{j},\mu _{j})-\psi
			_{j+1}(X_{i,j},\mu _{j})\right\vert ^{p}\right] \\
			&\leq e^{-b\gamma _{j+1}}\mathbb{E}\left[ \left\vert X_{j}-X_{i,j}\right\vert
			^{p}\right] +c_{\ast }{\E[\Gamma_p^p(\mu_j)+|X_j|^p+|X_{i,j}|^p]}\gamma _{j}^{1+\varepsilon } \\
			&\le e^{-b\gamma _{j+1}}\mathbb{E}\left[ \left\vert X_{j}-X_{i,j}\right\vert
			^{p}\right] +3 c_{\ast } \Gamma_p^p(X) \gamma _{j}^{1+\varepsilon }.
		\end{align*} 
		We iterate this inequality and we obtain 
		\begin{equation*}
			\mathbb{E}\left[ |X_{j}-X_{i,j}|^{p}\right] \leq e^{-b(t_{j}-t_{i})}\mathbb{E%
			}\left[ |X_{i}-\widehat{X}_{i}|^{p}\right] +3 c_{\ast } \Gamma_p^p(X)  \sum_{k=i}^{j}e^{-b(t_{j}-t_{k})}\gamma _{k}^{1+\varepsilon }.
		\end{equation*}%
		The hypothesis \eqref{New0} ensures that $\sum_{k=i}^{j}e^{-b(t_{j}-t_{k})}%
		\gamma _{k}^{1+\varepsilon }\leq C\gamma _{j}^{\varepsilon }$, and using
		that $\hat{X}_{i}\overset{d}{=} X_{i}$, we get 
		\begin{equation*}
			\mathbb{E}\left[ |X_{j}-X_{i,j}|^{p}\right] \leq 2^{p-1}e^{-b(t_{j}-t_{i})}%
			\mathbb{E}\left[ |X_{i}|^{p}\right] +3 C\times c_{\ast } \Gamma_p^p(X) \gamma
			_{j}^{\varepsilon }.
		\end{equation*}%
		We finally obtain using~\eqref{ineq_p} 
		\begin{align*}
			\mathbb{E}[\left\vert \mathbb{E}[\xi _{j}\mid X_{i}]\right\vert ] &\leq %
			\mathbb{E}[|\mathbb{E}[\xi _{j}\mid X_{i}]|^{p}]^{1/p}\leq \left\Vert \nabla
			\varphi \right\Vert _{\infty }(\mathbb{E}\left\vert X_{j}-X_{i,j}\right\vert
			^{p})^{1/p} \\
			&\leq \left\Vert \nabla \varphi \right\Vert _{\infty }(2^{1-1/p}e^{-\frac{b%
				}{p}(t_{j}-t_{i})}\mathbb{E}\left[ |X_{i}|^{p}\right] ^{1/p}+(3C
			c_{\ast })^{1/p} \Gamma_p(X)\gamma _{j}^{\varepsilon /p}),
		\end{align*}
		and the point \textbf{A} is finally proved.
		
		\textbf{B.} From~\eqref{na7}, we get
		\begin{align*}
			&\mathbb{E}\left[ \left( \int_{{\mathbb{R}}^{d}}\varphi (x)(\rho
			_{k}(X)(dx)-\mathcal{L}_{k}(X)(dx))\right) ^{2}\right] =\mathbb{E}\left[
			\left( \frac{1}{t_{k}}\sum_{i=1}^{k}\gamma _{i}\xi _{i}\right) ^{2}\right] \\
			&\leq 2\mathbb{E}\left[ \frac{1}{t_{k}}\sum_{i=1}^{k}\gamma _{i}\xi _{i}%
			\frac{1}{t_{k}}\sum_{j=i}^{k}\gamma _{j}\xi _{j}\right] \\
			&=2\mathbb{E}\left[ \frac{1}{t_{k}}\sum_{i=1}^{k}\gamma _{i}\xi _{i}\frac{1%
			}{t_{k}}\sum_{j=i}^{k}\gamma _{j}\mathbb{E}[\xi _{j}\mid X_{i}]\right] \\
			&\leq 4\left\Vert \varphi \right\Vert _{\infty }\frac{1}{t_{k}}%
			\sum_{i=1}^{k}\gamma _{i}\frac{1}{t_{k}}\sum_{j=i}^{k}\gamma _{j}\mathbb{E}%
			\left[ |\mathbb{E}[\xi _{j}\mid X_{i}]|\right] \\
			&\leq C\Gamma_{p}(X) \left\Vert \varphi \right\Vert _{\infty }\left\Vert \nabla \varphi
			\right\Vert _{\infty }\frac{1}{t_{k}}\sum_{i=1}^{k}\gamma _{i}\frac{1}{t_{k}}%
			\sum_{j=i}^{k}\gamma _{j}(e^{-\frac{b}{p}(t_{j}-t_{i})}+\gamma
			_{j}^{\varepsilon /p}),
		\end{align*}
		by using \textbf{A} for the last inequality. Since for $i\geq 1$, 
		\begin{equation*}
			\sum_{j=i}^{k}\gamma _{j}e^{-\frac{b}{p}(t_{j}-t_{i})}\leq
			\int_{t_{i-1}}^{t_{k}}e^{-\frac{b}{p}(s-t_{i})}ds\leq \frac{p}{b},
		\end{equation*}%
		we get 
		\begin{equation*}
			\frac{1}{t_{k}}\sum_{i=1}^{k}\gamma _{i}\frac{1}{t_{k}}\sum_{j=i+1}^{k}%
			\gamma _{j}e^{-\frac{b}{p}(t_{j}-t_{i})}\leq \frac{1}{t_{k}}\times \frac{p}{b%
			}.
		\end{equation*}%
		Moreover, since \eqref{New0} holds with $\varepsilon$ replaced by $%
		\varepsilon/p$, we get 
		\begin{equation*}
			\frac{1}{t_{k}}\sum_{i=1}^{k}\gamma _{i}\frac{1}{t_{k}}\sum_{j=i}^{k}\gamma
			_{j}^{1+\varepsilon /p}\leq \frac{C}{t_{k}}.
		\end{equation*}%
		so that we finally get%
		\begin{equation*}
			\mathbb{E}\left[ \left( \frac{1}{t_{k}}\sum_{i=1}^{k}\gamma _{i}\xi
			_{i}\right) ^{2}\right] \leq C \Gamma_{p}(X) \left\Vert \varphi
			\right\Vert _{\infty }\left\Vert \nabla \varphi \right\Vert _{\infty
			}t_{k}^{-1},
		\end{equation*}%
		where $C$ depends only on $b$, $\varepsilon $, $c_{\ast }$,
		the sequence $(\gamma _{n})$ and $p$.
	\end{proof}

	We can now proceed with the proof of Theorem~\ref{thm_ergo}. To this end, we will use a useful  generalization of the Horowitz-Karandikar estimate for the Wasserstein distance of random measures, which is postponed to Appendix~\ref{app:HK}. Since  Theorem~\ref{thm_ergo} deals with the ergodic measure $\rho_k(X)$, we also need to introduce some notation to deal with random variables on the set of probability measures. Let us consider two random  probability measures $\mu (\omega ,dx)$ and $\nu (\omega ,dx)$. This means that for each $\omega \in \Omega $, $\mu
	(\omega ,dx)$ and $\nu(\omega,dy)$ are probability measures. For $p\ge 1$, we define 
	$$\|\mu\|_p= \mathbb{E}%
	\left[\int_{{\ \mathbb{R}}^{d}}\left\vert x\right\vert ^{p}\mu (\omega ,dx)\right]^{1/p}.$$ 
	If $\|\mu\|_p$ is finite, then  $\mu(\omega,dx) \in \cP_p(\R^d)$ almost surely.
	We also introduce the following distance $\overline{W}_{p}$ between random  probability measures. For a bounded measurable function $\varphi \in C_{b}^{1}({\mathbb{R}}^{d})$ we
	denote%
	\begin{equation}
		S_{\mu ,\nu ,p}(\varphi )=\mathbb{E}\left[ \left\vert \int \varphi (x)\mu
		(\omega ,dx)-\int \varphi (x)\nu (\omega ,dx)\right\vert ^{p}\right] ^{1/p},
		\label{HK2}
	\end{equation}%
	and we define 
	\begin{equation}\label{def:HKp}
		\overline{W}_{p}(\mu ,\nu )=\sup \{S_{\mu ,\nu ,p}(\varphi ):\varphi \in
		C_{b}^{1}({\mathbb{R}}^{d})\text{ s.t. }\left\Vert \varphi \right\Vert
		_{\infty }\times \left\Vert \nabla \varphi \right\Vert _{\infty }\leq 1\}.
	\end{equation}%

	\begin{proof}[Proof of Theorem~\ref{thm_ergo}.] We will use the upper bound given by Lemma~\ref{lemma_reg_wasserstein} in the appendix with $\mu_{k}=\rho _{k}(X)$ and $\nu _{k}=\mathcal{%
			L}_{k}(X)$. Note that $\nu _{k}$ does not depend on $\omega $ while $\mu
		_{k} $ does. By assumption, we have $\Vert \mu _{k}\Vert _{p^{\prime
		}}^{p^{\prime }}={\mathbb{E}}\left[ \frac{1}{t_{k}}\sum_{i=1}^{k}\gamma
		_{i}|X_{i}|^{p^{\prime }}\right] \leq \Gamma_{p^{\prime }}(X)^{p^{\prime }}$
		and $\Vert \nu _{k}\Vert _{p^{\prime }}^{p^{\prime }}=\frac{1}{t_{k}}%
		\sum_{i=1}^{k}\gamma _{i}{\mathbb{E}}[|X_{i}|^{p^{\prime }}]\leq \Gamma_{p^{\prime }}(X)^{p^{\prime }}$. From~\eqref{HK3}, we get
		\begin{equation*}
			\mathbb{E}[W_{p}^{p}(\mu _{k},\nu _{k})]\leq C(\Gamma_{p^{\prime
			}}(X)^{p^{\prime }})^{1-\hat{p}}\overline{W}_{1}(\mu _{k},\nu _{k})^{\hat{p}}.
		\end{equation*}%
		On the other hand, Lemma~\ref{lemma_Wbar2} (Eq.~\eqref{na4}) gives 
		\begin{equation*}
			\overline{W}_{2}^{2}(\mu _{k},\nu _{k})\leq C\Gamma_{p}(X)t_{k}^{-1}.
		\end{equation*}%
		Using $\overline{W}_{1}(\mu _{k},\nu _{k})\leq \overline{W}%
		_{2}(\mu _{k},\nu _{k})$ and the two previous inequalities, we obtain using that $\hat{p}\in (0,1)$ 
		\begin{equation*}
			\mathbb{E}[W_{p}^{p}(\mu _{k},\nu _{k})]\leq C\Gamma_{p^{\prime
			}}(X)^{p^{\prime }}t_{k}^{-\hat{p}/2}.\qedhere
		\end{equation*}%
	\end{proof}

	\subsection{Proof of Theorem \protect\ref{MAIN0}}

	In order to prove our main result, a crucial step is the following
	comparison lemma. This is a discrete Gronwall type lemma that plays an
	analogous role as in ~\cite[Lemma 4.1]{DJL} in their continuous time setting.
	
	\begin{lemma}
		\label{lem_discrete_Gronwall} Let $(\gamma_k)_{k\ge 1}$ be a nonincreasing
		sequence of positive numbers such that $\gamma_k\to_{k\to \infty}0$ and $%
		t_k=\sum_{i=1}^k \gamma_i \to_{k\to \infty}+\infty$. Let $(f_{k})_{k\in 
			\mathbb{N}}\in {\mathbb{R}}_+^{\mathbb{N}}$ be a sequence that verifies for
		all $k\ge 0$, 
		\begin{equation}
			\frac{f_{k+1}-f_{k}}{\gamma _{k+1}}\leq -bf_{k}+\frac{a}{t_{k}}
			\sum_{i=1}^{k}\gamma _{i}f_{i}+Ct_{k}^{-\delta }  \label{na2'}
		\end{equation}
		for some $b>a>0$, $C\in{\mathbb{R}}_+$ and $\delta \in (0,1)$. Then, for any 
		$\zeta \in (0,\delta] \cap (0,1-\frac{a}{b})$, there exists a constant $%
		C^{\prime }\in {\mathbb{R}}_+$ such that $f_{k}\leq C^{\prime }t_{k}^{-\zeta
		} $.
	\end{lemma}
	
	Note that if $b(1-\delta)>a$, then $f_k=O(t_k^{-\delta})$.
	
	\begin{proof}
		Let us define $g_k=t_k^{-\zeta}$, with $\zeta \in (0,\delta] \cap (0,1-\frac{a}{b})$. We have 
		\begin{align*}
			g_{k+1}-g_k=\int_{t_k}^{t_{k+1}}\frac{-\zeta}{s^{1+\zeta}} ds \ge - \frac{%
				\zeta \gamma_{k+1}}{t_k^{1+\zeta}}= - \frac{%
				\zeta \gamma_{k+1}}{t_k}g_k.
		\end{align*}
		On the other hand, we have
		\begin{equation*}
			\frac 1{t_k} \sum_{i=1}^k \gamma_i g_i =\frac 1{t_k} \sum_{i=1}^k \gamma_i
			t_i^{-\zeta}\le \frac 1{t_k} \sum_{i=1}^k \int_{t_{i-1}}^{t_i} s^{-\zeta} ds
			= \frac{t_k^{-\zeta}}{1-\zeta}=\frac{g_k}{1-\zeta}.
		\end{equation*}
		
		Therefore, we get with $c_{a,b,\zeta}:=\frac 12(b-\frac{a}{1-\zeta})>0$
		$$-b g_k + \frac a{t_k} \sum_{i=1}^k \gamma_i
		g_i + c_{a,b,\zeta} g_k \le \left(-b+\frac{a}{1-\zeta}+c_{a,b,\zeta}\right) g_k = -c_{a,b,\zeta} g_k.$$

		Since $\frac{\zeta }{t_k} \to 0$,  we get that there exists $K$
		such that for $k\ge K$,
		\begin{equation}  \label{ineq:gk}
			\frac{g_{k+1}-g_k}{\gamma_{k+1}}\geq  -b g_k + \frac a{t_k} \sum_{i=1}^k \gamma_i
			g_i + c_{a,b,\zeta} t_k^{-\zeta}.
		\end{equation}
		We may
		assume also without loss of generality that $b\gamma_K<1$.
		
		Now, let $C^{\prime }\in {\mathbb{R}}_+$ be such that $C^{\prime }g_k>f_k$
		for $k\in \{1,\dots,K\}$ and $C^{\prime }c_{a,b,\zeta} t_k^{-\zeta} \ge C
		t_k^{-\delta}$ for all $k\ge K$ (we use here $\zeta\le \delta$). We define $\Delta_k=C^{\prime }g_k-f_k$.
		From~\eqref{na2'} and~\eqref{ineq:gk} we have for $k\ge K$: 
		\begin{equation*}
			\frac{\Delta_{k+1}-\Delta_k}{\gamma_{k+1}} \ge -b\Delta_k +\frac {a}{t_k}%
			\sum_{i=1}^k \gamma_i \Delta_i,
		\end{equation*}
		since $C^{\prime }c_{a,b,\zeta} t_k^{-\zeta} - C t_k^{-\delta}\ge 0$. By
		definition of $C^{\prime }$, $\Delta_k>0$ for $k\in \{1,\dots,K\}$ and since 
		$b\gamma_k<1$ for $k\ge K$, we easily get by induction on $k$ that $%
		\Delta_k> 0$ for all $k\ge K$ since we have $\Delta_{k+1} \ge (1-\gamma_{k+1} b)
		\Delta_k$.
	\end{proof}
	
	We will need also the following lemma. The proof follows a similar argument.
	
	\begin{lemma}
		\label{lem_Gronwall_bounded} Let $(\gamma_k)_{k\ge 1}$ be a nonincreasing
		sequence of positive numbers such that $t_k=\sum_{i=1}^k \gamma_i \to_{k\to
			\infty}+\infty$ and $\gamma_k\to_{k\to \infty}0$. Let $(f_{k})_{k\in \mathbb{%
				N}}\in {\mathbb{R}}_+^{\mathbb{N}}$, such that there exists $K\in {\mathbb{N}%
		}$, such that for all $k\ge K$, 
		\begin{equation*}
			\frac{f_{k+1}-f_{k}}{\gamma _{k+1}}\leq -bf_{k}+\frac{a}{t_{k}}
			\sum_{i=1}^{k}\gamma _{i}f_{i}+C  
		\end{equation*}
		for some $b>a>0$, $C\in{\mathbb{R}}_+$. Then, $\sup_{k\in {\mathbb{N}}%
		}f_k<\infty$.
	\end{lemma}
	
	\begin{proof}
		Without loss of generality, we may assume also $b\gamma _{K}<1$. Let $%
		M=\max_{0\leq k\leq K}f_{k}+\frac{C}{b-a}+1$ and set $\Delta _{k}=M-f_{k}$.
		We have for $k\geq K$, 
		\begin{equation*}
			\frac{M-\Delta _{k+1}-(M-\Delta _{k})}{\gamma _{k+1}}\leq -b(M-\Delta _{k})+%
			\frac{a}{t_{k}}\sum_{i=1}^{k}\gamma _{i}(M-\Delta _{i})+C,
		\end{equation*}%
		and thus 
		\begin{align*}
			\frac{\Delta _{k+1}-\Delta _{k}}{\gamma _{k+1}}& \geq -b\Delta _{k}+\frac{a}{%
				t_{k}}\sum_{i=1}^{k}\gamma _{i}\Delta _{i}+(b-a)M-C \\
			& \geq -b\Delta _{k}+\frac{a}{t_{k}}\sum_{i=1}^{k}\gamma _{i}\Delta _{i},
		\end{align*}
		since $M\geq C/(b-a)$. By definition of $M$, we have $\Delta _{k}>0$ for $%
		k\in \{0,\dots K\}$. As in the previous lemma, we get by induction on $k$ that $\Delta _{k}>0$
		for $k\geq K$.
	\end{proof}
	
	Coming back to the chains $\{X_k\}_k$ and $\{Y_k\}_k$ defined in \eqref{New1}
	and \eqref{Y} respectively, we have the following result.
	
	\begin{lemma}\label{lem_couplingchain}
		Let $X_0,Y_0\in L^p$ and $\{X_k\}$ (resp. $\{Y_k\}$) be the chain defined by~\eqref{New1} (resp~\eqref{Y}). Suppose that $\psi$ is $(b,\alpha ,\varepsilon )-$self coupled
		(see \eqref{New2}). Then, for every $k\in {\mathbb{N}}$ one may construct
		two random vectors $(\overline{X}_{0},...,\overline{X}_{k})$ and $(\overline{%
			Y}_{0},...,\overline{Y}_{k})$ which have the same law as $(X_{0},...,X_{k})$
		and $(Y_{0},...,Y_{k})$ respectively, and which verify for $k\geq 0$ 
		\begin{equation}  \label{Ap2}
			\begin{array}{rl}
				{\mathbb{E}}[|\overline{X}_{k+1}-\overline{Y}_{k+1}|^{p} ]\leq & (1-b\gamma
				_{k+1}){\mathbb{E}}[|\overline{X}_{k}-\overline{Y}_{k}|^{p}] +\alpha \gamma
				_{k+1}{\mathbb{E}}[W_{p}^{p}(\mathcal{L}(\overline{X}_{k}),\rho _{k}(%
				\overline{Y}))]\smallskip \\ 
				& +c_{\ast }\gamma _{k+1}^{1+\varepsilon }\E[\Gamma _{p}^{p}(\mathcal{L}(%
				\overline{X}_{k}),\rho _{k}(\overline{Y}))].%
			\end{array}%
		\end{equation}%
	\end{lemma}
	As it is usually the case in this work, the vectors $(\overline{X}_{0},...,\overline{X}_{k})$ and $(\overline{Y}%
	_{0},...,\overline{Y}_{k})$ are constructed on a possibly extended probability space.
	
	\begin{proof}
		We prove our property by using a coupling argument and recurrence. For $k=0$, we simply take $\overline{X}_{0}=X_0$ and $\overline{Y}_{0}=Y_0$. Now, we
		suppose that the induction hypothesis is true for $k$, i.e.  we have constructed $(\overline{X}_{0},...,\overline{%
			X}_{k})$ and $(\overline{Y}_{0},...,\overline{Y}_{k})$ respectively distributed as $(X_{0},...,X_{k})$ and $(Y_{0},...,Y_{k})$ satisfying~\eqref{Ap2}, and we produce $%
		\overline{X}_{k+1}$ and $\overline{Y}_{k+1}$.  We fix $x,y \in \R^d$, $\mu,\nu \in \cP_p(\R^d)$
		and we take $\Pi^{\mu,\nu}_{x,y}(dz_{1},dz_{2})$ to be the optimal coupling
		of $\psi _{k+1}(x,\mu )$ and $\psi _{k+1}(y,\nu ).$  Then, we take $
		(X_{k+1}(x,y,\mu ,\nu ),Y_{k+1}(x,y,\mu ,\nu ))$ of law $\Pi^{\mu,\nu}_{x,y}(dz_{1},dz_{2})$ and independent of $(\overline{X}_{0},\dots,\overline{%
			X}_{k},\overline{Y}_{0},\dots,\overline{Y}_{k})$. We can construct these random variables so that the function $(\omega,x,y,\mu,\nu) \mapsto 
		(X_{k+1}(x,y,\mu ,\nu )(\omega),Y_{k+1}(x,y,\mu ,\nu )(\omega))$ is jointly measurable by using~\cite[Corollary 5.22]{Villani} and~\cite[Lemma 3.3]{AB}. We define 
		\begin{equation*}
			\overline{X}_{k+1}=X_{k+1}(\overline{X}_{k},\overline{Y}_{k},\mathcal{L}(%
			\overline{X}_{k}),\rho _{k}(\overline{Y}))\quad \mbox{and}\quad \overline{Y}%
			_{k+1}=Y_{k+1}(\overline{X}_{k},\overline{Y}_{k},\mathcal{L}(\overline{X}%
			_{k}),\rho _{k}(\overline{Y})).
		\end{equation*}
		
		\textbf{Step 1.} We prove \eqref{Ap2}.
		
		Let $\mathcal{V}_k$ denote the joint law of $(\overline{X}_k, \overline{Y}_k,%
		\mathcal{L}(\overline{X}_k), \rho_k(\overline{Y}))$. From the independence between $
		(X_{k+1}(x,y,\mu ,\nu ),Y_{k+1}(x,y,\mu ,\nu ))$  and $(\overline{X}_k, \overline{Y}_k,%
		\mathcal{L}(\overline{X}_k), \rho_k(\overline{Y}))$, and 
		since $\Pi_{x,y}^{\mu,\nu}$ is the $W_p$-optimal coupling of $\psi_{k+1}(x,\mu)$ and $%
		\psi_{k+1}(y,\nu)$, we get 
		\begin{align*}
			&{\mathbb{E}}[|\overline{X}_{k+1}-\overline{Y}_{k+1}|^p] =\int
			W_p^p(\psi_{k+1}(x,\mu),\psi_{k+1}(y,\nu))d\mathcal{V}_k(x,y,\mu,\nu) \\
			&\leq \int [(1-b\gamma _{k+1})\left\vert x-y\right\vert ^{p}+\alpha
			W_{p}^{p}(\mu ,\nu )\gamma _{k+1}+\Gamma _{p}^{p}(\mu ,\nu )c_{\ast }\gamma
			_{k+1}^{1+\varepsilon } ] d\mathcal{V}_k(x,y,\mu,\nu) \\
			&=(1-b\gamma _{k+1}){\mathbb{E}}|\overline{X}_{k}-\overline{Y}_{k}|^{p}
			+\alpha \gamma _{k+1}{\mathbb{E}}[W_{p}^{p}(\mathcal{L}(\overline{X}%
			_{k}),\rho _{k}(\overline{Y})) ]\\
			&\qquad +c_{\ast }\gamma _{k+1}^{1+\varepsilon } \E[\Gamma _{p}^{p}(\mathcal{L}(%
			\overline{X}_{k}),\rho _{k}(\overline{Y}))],
		\end{align*}
		where we have used the contraction property \eqref{New2}. Therefore 
		\eqref{Ap2} is proved.
		
		\medskip
		
		\textbf{Step 2.} We prove equality of laws. We define $\mathcal{F}_{k}=\sigma (X_{0},...,X_{k},Y_{0},...,Y_{k})$ and $%
		\overline{\mathcal{F}}_{k}=\sigma (\overline{X}_{0},...,\overline{X}_{k},%
		\overline{Y}_{0},...,\overline{Y}_{k})$ and notice that $\psi_{k+1}$ (resp. $Y_{k+1}(x,y,\mu ,\nu )$) is
		independent of $\mathcal{F}_{k}$ (resp. $\overline{\mathcal{F}}_{k}$).
		
		Let $\Phi _{0},\dots,\Phi _{k+1}: \R^d \to \R$ be a  bounded measurable functions  and $F_{k+1}(y,\nu )={\mathbb{E}}[\Phi _{k+1}(\psi _{k+1}(y,\nu ))]$. By construction of $Y_{k+1}$, the random variables $Y_{k+1}(x,y,\mu ,\nu )$ and $\psi _{k+1}(y,\nu )$ have the same law, and thus  $F_{k+1}(y,\nu )={\mathbb{E}}(\Phi
		_{k+1}(Y_{k+1}(x,y,\mu ,\nu )))$ for any $x \in \R^d$, $\mu \in \mathcal{P}_p(\R^d)$. Therefore, we have 
		\begin{align*}
			&{\mathbb{E}}[\Phi _{k+1}(\overline{Y}_{k+1}) \mid \overline{\mathcal{F}}%
			_{k}] =F_{k+1}(\overline{Y}_{k},\rho _{k}(\overline{Y})).
		\end{align*}%
		We compute now 
		\begin{align*}
				{\mathbb{E}}\left(\prod_{i=1}^{k+1}\Phi _{i}(\overline{Y}_{i})\right) &={\mathbb{E}}\left({%
				\mathbb{E}}(\Phi _{k+1}(\overline{Y}_{k+1})\mid \overline{\mathcal{F}}%
			_{k})\prod_{i=1}^{k}\Phi _{i}(\overline{Y}_{i})\right) \\
			&={\mathbb{E}}\left(F_{k+1}(\overline{Y}_{k},\rho _{k}(\overline{Y}))\prod_{i=1}^{k}\Phi _{i}(\overline{%
				Y}_{i})\right).
		\end{align*}
		Using the identity of laws from the recurrence hypothesis (up to $k)$ the
		above expression is equal to 
		\begin{align*}
			{\mathbb{E}}\left(F_{k+1}(Y_{k},\rho
		_{k}(Y))\prod_{i=1}^{k}\Phi _{i}(Y_{i})\right) &={\mathbb{E}}\left({\mathbb{E}}(\Phi
		_{k+1}(Y_{k+1})\mid \mathcal{F}_{k})\prod_{i=1}^{k}\Phi _{i}(Y_{i}) \right) \\
		&={\mathbb{E}}\left(\prod_{i=1}^{k+1}\Phi _{i}(Y_{i})\right).
		\end{align*}
		The proof for $(\overline{X}_0,\dots,\overline{X}_{k+1})$ is similar.
	\end{proof}
	
	We can finally prove our main theorem.
	
	\begin{proof}[Proof of Theorem~\protect\ref{MAIN0}]
		Let $\{X_k\}_{k\in \N}$ be the chain defined by $X_0=Y_0$ and~\eqref{New1}. 	
		We use Lemma~\ref{lem_couplingchain} to construct $\{\overline{X}_k\}_{k\in \N}$ and $\{\overline{Y}_k\}_{k\in \N}$  and we denote
		$f_{k}={\mathbb{E}}[|\overline{X}_{k}-\overline{Y}_{k}|^{p}]$.
		We get 
		\begin{equation}
			f_{k+1}\leq (1-b\gamma _{k+1})f_{k}+\alpha \gamma _{k+1}{\mathbb{E}}[W_{p}^{p}(%
			\mathcal{L}(\overline{X}_{k}),\rho _{k}(\overline{Y}))]+c_{\ast }\gamma _{k+1}^{1+\varepsilon }{%
				\mathbb{E}}[\Gamma _{p}^{p}(\mathcal{L}(\overline{X}_{k}),\rho _{k}(\overline{Y}))].
			\label{ineq_for_Gronwall}
		\end{equation}
		Let us prove first that ${\mathbb{E}}[\Gamma _{p}^{p}(\mathcal{L}%
		(X_{k}),\rho _{k}(Y))]={\mathbb{E}}[\Gamma _{p}^{p}(\mathcal{L}(\overline{X}%
		_{k}),\rho _{k}(\overline{Y}))]$ is a bounded sequence. We know that $\sup_{k\geq 0}{\mathbb{E}}%
		[|X_{k}|^{p^{\prime }}]<\infty $,  and thus $\sup_{k\geq 0}\Gamma
		_{p}^{p}(X_{k})<\infty $. Since $$\Gamma _{p}^{p}(\mathcal{L}(X_{k}),\rho
		_{k}(Y))\leq { 2^p  }\Gamma _{p}^{p}(X_{k})+%
		{2^{p-1}}W_{p}^{p}(\mathcal{L}(X_{k}),\rho _{k}(Y)),$$ we get
		\begin{align*}
			f_{k+1}\leq &(1-b\gamma _{k+1})f_{k}+ \gamma _{k+1} (\alpha + c_*2^{p-1}\gamma _{k+1}^\varepsilon  ) {\mathbb{E}}
			[W_{p}^{p}(\mathcal{L}(X_{k}),\rho _{k}(Y))]\\&+c_{\ast }{ 2^p }%
			\gamma _{k+1}^{{ 1+\varepsilon }}\sup_{k\geq 0}\Gamma 
			_{p}^{p}(X_{k}) .
		\end{align*}
		Let $\alpha ^{\prime }\in (\alpha ,b)$.
		For $x,y\geq 0$, we have using the triangular inequality and  $(x+y)^{p}\leq \frac{\alpha ^{\prime }}{\alpha }%
		x^{p}+Cy^{p}$ for some constant $C\geq 1$ \footnote{
			We take for example $C=\sup_{z\geq (\frac{\alpha'}{\alpha})^{ 1/p}-1}\frac{(1+z)^{p}}{z^{p}}\geq 1$. If $x=0$, the inequality is clear. Otherwise, it is
			equivalent to $(1+z)^{p}\leq \frac{\alpha ^{\prime }}{\alpha }+Cz^{p}$: this
			inequality is true for $z\in \lbrack 0, (\frac{\alpha'}{\alpha})^{ 1/p}-1]$ since $%
			(1+z)^{p}\leq \frac{\alpha^{\prime }}{\alpha }$, and also for $z>(\frac{\alpha'}{\alpha})^{ 1/p}-1$ since $(1+z)^{p}\leq Cz^{p}$.}, we get 
		$$W_{p}^{p}(\mathcal{L}(\overline{X}_{k}),\rho _{k}(\overline{Y}))\leq 
		\frac{\alpha ^{\prime }}{\alpha }W_{p}^{p}(\rho_k(\overline{X}),\rho
		_{k}(\overline{Y}))+CW_{p}^{p}(\mathcal{L}(\overline{X}_{k}), \rho_k(\overline{X})).$$
		Note that $\mathcal{L}(\overline{X}_{k})=\mathcal{L}(X_{k})$ and $\rho_k(\overline{X})$ has the same distribution as $\rho_k(X)$ since $(\overline{X}_{0},...,\overline{X}_{k})$ has the same law as $(X_{0},...,X_{k})$. This gives  $$\E[W_{p}^{p}(\mathcal{L}(\overline{X}_{k}), \rho_k(\overline{X}))]=\E[W_{p}^{p}(\mathcal{L}(X_{k}), \rho_k(X))].$$
		Since $\gamma_k=O(t_k^{-\eta})$, we get by applying \eqref{NEW8'} in  Corollary~\ref{cor_speeds_X} and \eqref{na6'} in  Theorem~\ref{thm_ergo}, $\E[W_{p}^{p}(\mathcal{L}(X_{k}), \rho_{k}(X))]=O(t_k^{-\delta})$ for some $\delta \in (0,1)$. Let $\alpha''\in (\alpha',b)$.  Combining the two previous previous inequalities gives the existence of $K\in \N$ and $C \in \R$, such that for all $k\ge K$, 
		$$ \frac{f_{k+1}-f_k}{\gamma_{k+1}}\le -bf_k +\alpha'' \E[W_{p}^{p}(\rho_k(\overline{X}),\rho
		_{k}(\overline{Y}))] + C t_k^{-\delta}.  $$
		Since $W_{p}^{p}(\rho_k(\overline{X}),\rho
		_{k}(\overline{Y}))\le \frac 1{t_k} \sum_{i=1}^k \gamma_i |\overline{X}_i-\overline{Y}_i|^p$, we get $\E[W_{p}^{p}(\rho_k(X),\rho
		_{k}(Y))] \le  \frac 1{t_k} \sum_{i=1}^k \gamma_i f_i$. Finally, Lemma~\ref{lem_Gronwall_bounded}, ensures that $
		(f_{k})_k$ is bounded.  This gives $\sup_{k\geq 0}{\mathbb{E}}%
		[|Y_{k}|^{p}]<\infty $ and thus $\sup_{k\geq 0}{\mathbb{E}}[\Gamma
		_{p}^{p}(\rho _{k}(Y))]<\infty $ since $\Gamma _{p}^{p}(\rho _{k}(Y))=1+%
		\frac{2}{t_{k}}\sum_{i=1}^{k}\gamma _{i}|Y_{i}|^{p}$.
		
		Using this boundedness, we now get from~\eqref{ineq_for_Gronwall} 
		\begin{equation*}
			f_{k+1}\leq (1-b\gamma _{k+1})f_{k}+\alpha \gamma _{k+1}{\mathbb{E}}[W_{p}^{p}(%
			\mathcal{L}(\overline{X}_{k}),\rho _{k}(\overline{Y}))]+C\gamma _{k+1}^{1+\varepsilon }.
		\end{equation*}%
		Let $\alpha ^{\prime }\in (\alpha ,b)$ be arbitrarily close to $\alpha $. We use again the triangular inequality and  $(x+y)^{p}\leq \frac{\alpha ^{\prime }}{\alpha }%
		x^{p}+Cy^{p}$ to get $$W_{p}^{p}(\mathcal{L}(\overline{X}_{k}),\rho _{k}(\overline{Y}))\leq 
		\frac{\alpha ^{\prime }}{\alpha }W_{p}^{p}(\rho _{k}(\overline{X}),\rho
		_{k}(\overline{Y}))+CW_{p}^{p}(\mathcal{L}(\overline{X}_{k}),\rho _{k}(\overline{X})).$$ Besides, we use  the triangle inequality and then~\eqref{NEW8'} and~\eqref{na6} to obtain \begin{align*}
		W_{p}^{p}(\mathcal{L}(\overline{X}_{k}),\rho _{k}(\overline{X}))&\le2^{p-1}(W_{p}^{p}(\mathcal{L}(\overline{X}_{k}),\mu_*) + W_{p}^{p}(\mu_*,\rho _{k}(\overline{X})) ) \\&\le C (\gamma_k^\varepsilon +t_k^{-\frac{\hat{p}}2}).
		\end{align*}
		We use then again the
		estimate $\gamma_k\leq Ct_k^{-\eta}$ to get
		\begin{align*}
			f_{k+1}& \leq (1-b\gamma _{k+1})f_{k}+\alpha ^{\prime }\gamma _{k+1}{\mathbb{E}}%
			[W_{p}^{p}(\rho _{k}(\overline{X}),\rho _{k}(\overline{Y}))]+C\gamma _{k+1} ( \gamma_k^{\varepsilon }+t_{k}^{-\frac{\hat{p}}2 } )\\
			& \leq (1-b\gamma _{k+1})f_{k}+\alpha ^{\prime }\gamma _{k+1}  \frac{1}{t_{k}}%
			\sum_{i=1}^{k}\gamma _{i}f_{i}+C\gamma _{k+1}t_{k}^{-\min \left( \frac{\hat{p}}2,\varepsilon \eta
				\right) }.
		\end{align*}%
		From Lemma~\ref{lem_discrete_Gronwall}, we obtain $f_{k}\leq C^{\prime
		}t_{k}^{-\zeta }$.
		
		To prove~\eqref{na10-0}, we write $W_{p}^{p}(\rho _{k}(\overline{Y}),\mu _{\ast })\leq
		2^{p-1}W_{p}^{p}(\rho _{k}(\overline{X}),\mu _{\ast })+2^{p-1}W_{p}^{p}(\rho
		_{k}(\overline{X}),\rho _{k}(\overline{Y}))$. By Theorem~\ref{thm_ergo}, we have $\E[W_{p}^{p}(\rho _{k}(\overline{X}),\mu
		_{\ast })]=\E[W_{p}^{p}(\rho _{k}(X),\mu
		_{\ast })]\leq Ct_{k}^{-\frac{\hat{p}}2 }$. On the other hand, we have $$\E[W_{p}^{p}(\rho
		_{k}(\overline{X}),\rho _{k}(\overline{Y})) ]\leq \frac{1}{t_{k}}%
		\sum_{i=1}^{k}\gamma _{i}f_{i}\leq C^{\prime }\frac{1}{t_{k}}%
		\sum_{i=1}^{k}\gamma _{i}t_{i}^{-\zeta }\leq \frac{C^{\prime }}{1-\zeta }%
		t_{k}^{-\zeta }.$$ This gives the claim since $\E[W_{p}^{p}(\rho _{k}(\overline{Y}),\mu _{\ast })]=\E[W_{p}^{p}(\rho _{k}(Y),\mu _{\ast })]$. 
		
		To get~\eqref{na10-01}, it is sufficient to observe that $\gamma_k=\frac{1}{1+k}=O(t_k^{-\eta})$ for any $\eta>0$. For \eqref{na10-02}, we have $t_k \sim C k^{1-\beta}$ for some $C>0$ and thus $\gamma_k=O(t_k^{-\frac{\beta}{1-\beta}})$. This gives $\eta=\frac{\beta}{1-\beta}$ and $\varepsilon \eta>1$, and then~\eqref{na10-02}.
	\end{proof}
	
	\section{Examples}
	
	\label{sect:examples}

	\subsection{McKean-Vlasov Equations}
	We consider here the following McKean-Vlasov process
	$$X_{s,t}=X_{s,s}+\int_s^t \betaMV(X_{s,r},\mathcal{L}(X_{s,r}))dr+\int_s^t \sigma(X_{s,r},\mathcal{L}(X_{s,r}))dW_r,$$
	with $\betaMV:  \mathbb{R}^d \times \mathcal{P}_2(\mathbb{R}^d ) \to \mathbb{R}^d$,  $\sigma:  \mathbb{R}^d \times \mathcal{P}_2(\mathbb{R}^d ) \to \mathbb{R}^{d\times d}$ and $W$ is a $d$-dimensional Brownian motion\footnote{We abuse the notation slightly by using $b$ for the drift. The contraction constants are denoted by $\bar{b}$.}. We make the following assumption: there exists $C,k, L\ge 0$ such that for all $x,y \in \R$ and $\mu,\nu \in \mathcal{P}_2(\R^d)$,
	\begin{enumerate}[label=(M\arabic*),ref=M\arabic*]
		\item \label{M1}
		$|\betaMV(x,\mu)-\betaMV(y,\nu)|^2+|\sigma(x,\mu)-\sigma(y,\nu)|^2
		\le C (|x-y|^2 +W_2^2(\mu,\nu))$,
		\smallskip
		\item \label{M2}
		$2 \langle \betaMV(x,\mu)-\betaMV(y,\nu),x-y\rangle+|\sigma(x,\mu)-\sigma(y,\nu)|^2
		\le L W_2^2(\mu,\nu) -\bar{b} |x-y|^2$.
	
	\end{enumerate}
	Then, \cite[Theorem 2.1]{Wang} gives the strong existence and uniqueness for $X$, and we denote by $\theta_{s,t}(\mu)$ the law of $X_{s,t}$ when $X_{s,s}\sim \mu$. Besides, \cite[Theorem 3.1]{Wang} gives the existence and uniqueness of the invariant probability measure $\mu_*$.
	The corresponding Euler scheme is given by 
	$$\mathcal{X}_{s,t}(x,\mu)=x+\betaMV(x,\mu)(t-s)+\sigma(x,\mu)(W_t-W_s),$$
	and we define the map $\psi_k(x,\mu)=\mathcal{X}_{t_{k-1},t_k}(x,\mu)$ and $\Theta_k(\mu)$ by~\eqref{def_thetak}.
	
	\begin{lemma}\label{lemma_McKean}
		Let \eqref{M1} and \eqref{M2} hold with $\bar{b} >L$. Then, the assumptions $\mathrm{\mathbf{A}}_2(\bar{b} ,L,1)$ and~\eqref{na1a} are satisfied, and the property~\eqref{New6} holds with $\delta_{h,r,k}=\gamma_r^{1/2}$.
	\end{lemma}
	\begin{proof}
		We first prove that $\psi_k$ is self-coupled (Assumption~\ref{Ass_SelfCoupling}). Let $x,y\in \R$ and $\mu,\nu \in \mathcal{P}_2(\R)$. We note $\delta_{s,t}=\mathcal{X}_{s,t}(x,\mu)-\mathcal{X}_{s,t}(y,\nu)$ and have by Itô formula
		\begin{align*}
			d|\delta_{s,t}|^2=&\left(2 \langle\delta_{s,t} ,\betaMV(x,\mu)-\betaMV(y,\nu) \rangle +\mathrm{Tr}\left[(\sigma(x,\mu)-\sigma(y,\nu))(\sigma(x,\mu)-\sigma(y,\nu))^T \right] \right)dt \\&+2\langle \delta_{s,t},(\sigma(x,\mu)-\sigma(y,\nu))dW_t \rangle.
		\end{align*}
		Taking the expectation, writing $\delta_{s,t}=\delta_{s,t}-(x-y)+(x-y)$, and using \eqref{M2} leads to
		\begin{align*}
			\E[|\delta_{s,t}|^2]\leq &|x-y|^2(1-\bar{b} (t-s))+L W_2^2(\mu,\nu)(t-s)\\&+2\langle \betaMV(x,\mu)-\betaMV(y,\nu),  \int_s^t \E[(\delta_{s,u}-(x-y))]du \rangle\\
			=&|x-y|^2(1-\bar{b} (t-s))+L W_2^2(\mu,\nu)(t-s)\\&+|\betaMV(x,\mu)-\betaMV(y,\nu)|^2 (t-s)^2,
		\end{align*}
		since $\E[(\delta_{s,u}-(x-y))]=(u-s)(\betaMV(x,\mu)-\betaMV(y,\nu))$.
		
		We use (M1) and get that $(\psi_k)$ is $(\bar{b},L,1)$ self coupled. We also get~\eqref{na1a} by taking $\nu=\mu$.
		
		We now check  the asymptotic commutativity and that~\eqref{New6} holds with $\delta_{h,r,k}=\gamma_r^{1/2}$. This is a consequence of Proposition~\ref{prop:commuting}: on the one hand, we have the self-coupling property \eqref{New2} with $\bar{b}>L$, and Assumption~\eqref{ass:C}  has been checked in~\cite[page 7]{ABKH} with $p=2$ and $\varepsilon=1/2$ and on the other hand, the bound estimate~\eqref{uniform_bound_momp} has been proved in \cite[Theorem 3.1 (2)]{Wang} for $\theta$ while the bound for $\Theta$ is a consequence of the one for $\theta$ together with~\eqref{ass:C} by using~\eqref{eq:New6a}.
	\end{proof}
	
	From Theorem~\ref{MAIN0} and Lemma~\ref{lemma_McKean}, we get the following result.
	\begin{proposition}
		Let $\betaMV$ and $\sigma$ satisfy the conditions~\eqref{M1} and~\eqref{M2} above with $\bar{b} >L$. 
		Let  $\gamma _{k}=\frac{1}{(1+k)^{\beta }}$ with $%
		\beta \in (\frac{2}{3},1)$, and  suppose that, for some $p'>2$,  $\sup_{k\ge r}\|\Theta^\pi_{t_r,t_k}(\mu) \|_{p'}<\infty$ for any $\mu\in\mathcal{P}_{p'}(\R^d)$.
		Let $\hat{p}=\frac{2(p^{\prime }-2)}{(d+5/2)(p^{\prime
			}-2)+(d+2)2}$.  Then, for every 
		\begin{equation*}
			\zeta \in (0,1-\frac{L }{\bar{b} })\cap \big(0,\frac{\hat{p}}{2}\big],
		\end{equation*}
		there exists $C\in \R_+$ such that for all $k\ge 1$,  $\mathbb{E}[W_{2}^{2}(\rho
		_{k}(Y),\mu _{\ast })]\leq \frac{C}{k^{(1-\beta)\zeta}}$, or equivalently $\mathbb{E}[W_{2}^{2}(\rho
		_{k}(Y),\mu _{\ast })]\leq C t_k^{-\zeta}$.
	\end{proposition}

	\smallskip
	
	Recently, Chassagneux and Pagès~\cite{ChPa} have studied the rate of convergence under the same assumptions (1) and (2) in Theorem 1.2.~(i)\footnote{Note that the rate given by \cite[Theorem 1.1]{ChPa2} is used with $a=2p^*-2$, and one should read $2(p^*-1)$ instead of $2p^*-1$ in~\cite[Theorem~1.2~(i)]{ChPa}.} where they obtain essentially a rate of convergence in $\mathbb{E}[W_{2}^{2}(\rho
	_{k}(Y),\mu _{\ast })]=O(t_k^{-\tilde{\zeta}})$ with $\tilde{\zeta}\in (0,1-\frac{L }{\bar{b} })\cap \big(0,\frac{2p^* -2}{(d+3)(2p^*
		-2)+(d+2)2}\big]$, where $p^*$ is such that $\int_{\R^d} |y|^{2p^*} \mu_*(dy)<\infty$. The assumption  $\sup_{k\ge r}\|\Theta^\pi_{t_r,t_k}(\mu) \|_{p'}<\infty$  implies $\int_{\R^d} |y|^{p'} \mu_*(dy)<\infty$. Taking thus $p'=2p^*$, we see that we obtain a very similar rate, slightly improved ($5/2$ instead of $3$). Let us mention here that~\cite{ChPa} obtain also sharper rates under further assumptions. In contrast, the goal of the present paper is to propose a general framework, and it is encouraging to obtain in the case of the McKean-Vlasov equation an accurate estimate of the convergence speed.

	It is possible to go further and give sufficient assumptions  to get the boundedness of moments larger than~$2$. Let us take $p'\ge 4$ even and set $q'=p'/2$. We write 
	$$|\psi_\ell(x,\mu)|^{2}= |x+\gamma_\ell \betaMV(x,\mu)|^{2} + 2\langle x+\gamma_\ell \betaMV(x,\mu), \sigma(x,\mu) W_{\gamma_\ell} \rangle + |\sigma(x,\mu) W_{\gamma_\ell} |^2$$ and get
	\begin{align*}
		\E[|\psi_\ell(x,\mu)|^{p'}]&=\sum_{q=0}^{q'} \binom{q'}{q} |x+\gamma_\ell \betaMV(x,\mu)|^{2(q'-q)}\E[(2\langle x+\gamma_\ell \betaMV(x,\mu), \sigma(x,\mu) W_{\gamma_\ell} \rangle +|\sigma(x,\mu) W_{\gamma_\ell} |^2)^q]. 
	\end{align*}
	Note that~\eqref{M1} gives 
	$|\betaMV(x,\mu)|^2+|\sigma(x,\mu)|^2
	\le C (1+|x|^2 +\|\mu\|_2^2)$ by taking $y=0$ and $\nu=\delta_0$. Collecting terms in powers of $\gamma_\ell$, we have  
	\begin{align*}
		\E[|\psi_\ell(x,\mu)|^{p'}] \le &|x|^{p'}+\gamma_\ell|x|^{p'-2}(p' \langle x,\betaMV (x,\mu) \rangle + q' \mathrm{Tr}( \sigma(x,\mu)\sigma(x,\mu)^T) \\
		& + 2q'(q'-1) |\sigma(x,\mu)^T x|^2) + C \gamma_\ell^2( 1+ |x|^{p'}+ \|\mu\|_2^{p'}),  
	\end{align*}
	by noting that the terms of order one in $\gamma_\ell$ only appear for $q=0$, $q=1$  and $q=2$. 
	We now make the following assumption
	$$\textup{(M3)}\  2  \langle x,\betaMV (x,\mu) \rangle +  \mathrm{Tr}( \sigma(x,\mu)\sigma(x,\mu)^T) +(p'-2) |\sigma(x,\mu)^T x|^2 \le - \tilde{b}|x|^2 +\tilde{L} \|\mu\|_2^2 +\tilde{C},$$
	for some $\tilde{b},\tilde{L},\tilde{C}\in \R_+$ such that $\tilde{b}>\tilde{L}$. Then, we prove the Foster-Lyapunov criterion~\eqref{New0bis}.
	Indeed, we have for $X\sim \mu$,
	\begin{align*}\E[|\psi_\ell(X,\mu)|^{p'}]\le& \E[|X|^{p'}]+q' \gamma_\ell \E[|X|^{p'-2} ( - \tilde{b}|X|^2 +\tilde{L} \|\mu\|_2^2 +\tilde{C})]+C \gamma_\ell^2(1+\E[|X|^{p'}]+\|\mu\|_2^{p'}) \\
		\le &\|\mu\|_{p'}^{p'} +q' \gamma_\ell (-\tilde{b}\|\mu\|_{p'}^{p'} +\tilde{L}\|\mu\|_{p'}^{p'}+\tilde{C}(\varepsilon \|\mu\|_{p'}^{p'}+ \bar{C}_{\varepsilon,p'} ) ) + C \gamma_\ell^2 (1+\|\mu\|_{p'}^{p'}),
	\end{align*}
	where $\bar{C}_{\varepsilon,p'}$ is a constant such that $|x|^{p'-2}\le \varepsilon |x|^{p'}+ \bar{C}_{\varepsilon,p'}$ for all $x\in \R^d$. Since $\tilde{b}>\tilde{L}$ and $\gamma_\ell \to_{\ell\to \infty} 0$, we can take $\varepsilon$ sufficiently small and $\ell$ large enough so that 
	$$ \E[|\psi_\ell(X,\mu)|^{p'}]\le \|\mu\|_{p'}^{p'} (1-\gamma_\ell(\tilde{b}-\tilde{L})/2) + C \gamma_\ell,$$
	which proves that the Foster-Lyapunov criterion~\eqref{New0bis} holds. We then get the next corollary.
	\begin{corollary}
		Let $b,\sigma$ satisfy \eqref{M1}, \eqref{M2} and \textup{(M3)}. Let  $\gamma _{k}=\frac{1}{(1+k)^{\beta }}$ with $%
		\beta \in (\frac{2}{3},1)$ and  $\hat{p}=\frac{2(p^{\prime }-2)}{(d+5/2)(p^{\prime
			}-2)+(d+2)2}$.  Then, for every 
		\begin{equation*}
			\zeta \in (0,1-\frac{L }{\bar{b} })\cap \big(0,\frac{\hat{p}}{2}\big],
		\end{equation*}
		there exists $C\in \R_+$ such that for all $k\ge 1$,  $\mathbb{E}[W_{2}^{2}(\rho
		_{k}(Y),\mu _{\ast })]\leq \frac{C}{k^{(1-\beta)\zeta}}$.
	\end{corollary}

	\subsection{Boltzmann type equations in the martingale regime ($W_{2}$)}
	
	We consider the $d$-dimensional stochastic equation of Boltzmann type: 
	\begin{equation}\label{EX1.SDE}
		\begin{array}{rl}
			X_{s,t}(X)=&X+\int_{s}^{t}b(X_{s,r}(X))dr\smallskip\\
			&+\int_{s}^{t}
			\int_{\mathbb{R}^{d}\times E}
			c(v,z,X_{s,r-}(X))\widetilde{N}_{\mathcal{L}%
				(X_{s,r})}(dv,dz,dr)
		\end{array}
	\end{equation}%
	where $N_{\mathcal{L}(X_{s,r})}(dv,dz,dr)$ is a Poisson point measure on $%
	\mathbb{R}^{d}\times E\times \R_{+}$ with compensator $\widehat{N}_{\mathcal{L%
		}(X_{s,r})}(dv,dz,dr)=\mathcal{L}(X_{s,r})(dv){\rho} (dz)dr$  and $\widetilde{N}%
	_{\mathcal{L}(X_{s,r})}=N_{\mathcal{L}(X_{s,r})}-\widehat{N}_{\mathcal{L}%
		(X_{s,r})}.$ Here, as usual, $\mathcal{L}(X_{s,r})$ denotes the law of the
	random variable $X_{s,r}$ and the initial condition $X$ is square integrable and independent of $%
	N_{\mathcal{L}(X_{s,r})}$. Last, $E$ is a measurable space endowed  with the measure {$\rho$}.
	
	We will work under the following hypotheses:
	\begin{align}
		&\left\langle x-y,b(x)-b(y)\right\rangle \leq -\bar{b}\left\vert
		x-y\right\vert ^{2}\mbox{ and }\left\vert b(x)-b(y)\right\vert ^{2}\leq
		L_{b}\left\vert x-y\right\vert ^{2}  \label{VL4} \\
		&\left\vert c(v,z,x)-c(v^{\prime },z,x^{\prime })\right\vert ^{2} \leq 
		\overline{c}^{2}(z)(\left\vert v-v^{\prime }\right\vert ^{2}+\left\vert
		x-x^{\prime }\right\vert ^{2})\label{VL5} \\
		&\left\vert c(0,z,0)\right\vert ^{2} \leq \overline{c}^{2}(z),
		\quad \mbox{with }  {Q}_2:=\int_{E}\overline{c}^{2}(z)\rho (dz)<\infty .  \label{VL5'}
	\end{align}
	
	Under these hypotheses the above equation \eqref{EX1.SDE} has a unique solution (see \cite[Theorem 3.1]{ABC}).
	
	Notice that, as an immediate consequence of our hypothesis (\ref{VL4}) and (%
	\ref{VL5}) we have, for every $\delta >0$ and some $K>0$
	\begin{align}
		&\left\langle x,b(x)\right\rangle \leq (-\bar{b}+\delta )\left\vert
		x\right\vert ^{2}+{ L_b  }\delta ^{-1}\left\vert b(0)\right\vert ^{2}, \label{vl1} \\
		&\left\vert c(v,z,x)\right\vert ^{2} \leq {2}\overline{c}^{2}(z)(1+\left\vert
		v\right\vert ^{2}+\left\vert x\right\vert ^{2})\mbox{ and } \left\vert
		b(x)\right\vert \leq K(1+\left\vert x\right\vert ).  \label{vl2}
	\end{align}%

	We associate to the stochastic equation \eqref{EX1.SDE} the following one step Euler scheme: for $%
	x\in \mathbb{R}^{d}$ and $\mu \in \mathcal{P}_{2}(\mathbb{R}^{d})$ 
	\begin{equation*}
		\mathcal{X}_{s,t}(x,\mu )=x+\int_{s}^{t}b(x)dr+\int_{s}^{t}\int_{\mathbb{R}%
			^{d}}\int_{E}c(v,z,x){\widetilde{N}_{{\mu}}}(dv,dz,dr),
	\end{equation*}%
	where $N_{\mu}(dv,dz,dr)$ is the Poisson point measure with
	compensator $\widehat{N}_{\mu}(dv,dz,dr)=\mu(dv)\rho
	(dz)dr$ and $\widetilde{N}_{{\mu}}=N_{{\mu}}-\widehat{N}%
	_{{\mu}}.$
	
	Moreover, we define the following maps: for $\mathcal{L}(X)=\mu \in \mathcal{%
		P}_{2}(\mathbb{R}^{d})$ 
	\begin{equation*}
		\theta _{s,t}(\mu )=\mathcal{L}(X_{s,t}(X)),\quad \Theta _{s,t}(\mu )=%
		\mathcal{L}(\mathcal{X}_{s,t}(X,\mu )).
	\end{equation*}
	
	Recall the following notation employed in the present paper. We have the
	time grid $\pi =\{t_{0}<\cdots <t_{n}<\cdots\}$ and we define the applications%
	\begin{equation}\label{Ex1-psi}
		\psi _{k}(x,\mu )=\mathcal{X}_{t_{k-1},t_{k}}(x,\mu ),\quad k\geq 0.
	\end{equation}
	Following the definition in Section \ref{sect:results}, $\Theta_k(\mu)=\mathcal{L}(\psi _{k}(X,\mu ))$  with $\mu=\mathcal{L}(X)$, hence
	\begin{equation}\label{Ex1-Theta}
		\Theta_k(\mu)=\Theta_{t_{k-1},t_k}(\mu),
		\qquad \Theta
		_{t_{0},t_{k}}^{\pi }(\mu)=
		\Theta _{t_{k-1},t_{k}}\circ \cdots \circ \Theta
		_{t_{0},t_{1}}(\mu).
	\end{equation}

	\begin{remark}\label{rem:Boltz}
		In Theorem~C.4, Appendix C.2 of the long version of~\cite{ABKH}, see {\tt https://arxiv.org/pdf/2509.03971}, it is proved that, under the hypotheses \eqref{VL4}--\eqref{VL5'} and $Q_2<\bar{b}$, there exists a unique invariant measure $\mu_*$ for $\theta$.
		Moreover, it is proved that $\theta$ and $%
		\Theta $ satisfy a Foster-Lyapunov criterion for  $\|\cdot \|_2$ and  are $(2,b_{\ast },\frac{1}{2})$--coupled in the sense of {\cite[Definition 2.2]{ABKH}} for $b_*<2(\bar{b}-Q)$,
		which exactly means that the condition (\ref{New6'}) holds with $p=2$ and $\varepsilon=1/2$. Consequently, by Proposition \ref{prop:commuting} the asymptotic commuting property \eqref{New6} holds with $\delta_{h,r,k}=\gamma_r^{1/2}$. Moreover, by \cite[Lemma~2.6]{ABKH}, we get that $W_2(\Theta^\pi_{t_0,t_k}(\mu),\mu_*)\to_{k\to \infty} 0$ for any $\mu \in \mathcal{P}_2(\R^d)$, when $(t_k)$ satisfies~\eqref{def_partition}.
	\end{remark}
	
	We prove now in the following lemma that Assumptions~\ref{Ass_SelfCoupling} and~\ref{Ap} hold under our framework for Boltzmann equations.
	
	\begin{lemma}\label{LemmaEx1}		
		If the hypotheses \eqref{VL4}--\eqref{VL5'} and $t-s\in [0,1]$ hold then for every {$x_1,x_2\in\R^d$ and} $\mu_1,\mu_2{,\mu}\in \mathcal{P}_2(\R^d)$ one has
		\begin{equation}\label{EX1.1}
			\begin{array}{ll}
				W_{2}^{2}(\mathcal{X}_{s,t}(x_1,\mu_1 ),\mathcal{X}_{s,t}(x_2,\mu_2 )) \leq
				&\left\vert x_1-x_2\right\vert ^{2}(1-({ 2  }\bar{b}-{Q}_2)(t-s)) \\
				&+{Q}_2W_{2}^{2}(\mu_1 ,\mu_2 )(t-s)\\
				&+{C(1+\sum_{i=1}^2|x_i|^2+\|\mu_i\|_2^2)(t-s)^{3/2}},
			\end{array}%
		\end{equation}
		and
		\begin{equation}\label{EX1.2}
			\begin{array}{ll}
				\E[|\mathcal{X}_{s,t}(x_1,\mu )-\mathcal{X}_{s,t}(x_2,\mu )|^2] \leq
				&\left\vert x_1-x_2\right\vert ^{2}(1-({2}\bar{b}-{Q}_2)(t-s)) \\
				&{+C(1+\sum_{i=1}^2|x_i|^2+\|\mu\|_2^2)(t-s)^{3/2}}.
			\end{array}%
		\end{equation}
	\end{lemma}
	
	\begin{proof}
		We start by proving \eqref{EX1.1}.
		
		\textbf{Step 1: Coupling }Let ${\Pi _{\mu_1,\mu_2}}\in \mathcal{P}%
		_{2}(\R^{d}\times \R^{d})$ be the optimal coupling in $W_{2}$ of ${\mu _1
		}$ and $\mu_2 .$ We construct an application $\tau =(\tau _{1},\tau
		_{2}):(0,1)\rightarrow \mathbb{R}^{d}\times {\mathbb{R}^{d}}$ such that for
		every measurable and bounded function $\varphi $ 
		\begin{equation*}
			\int_{\mathbb{R}^{d}\times \mathbb{R}^{d}}\varphi (v_{1},v_{2}){\Pi _{\mu_1,\mu_2}}(dv_{1},dv_{2})=\int_{0}^{1}\varphi (\tau (w))dw.
		\end{equation*}%
		We consider a Poisson point measure $N(dw,dz,dr)$ on $(0,1)\times E\times
		\R_{+}$ with compensator $\widehat{N}(dw,dz,dr)=dw\times \rho (dz)\times dr$.
		
		Then we construct for $i=1,2$
		\begin{align*}
			\overline{\mathcal{X}}_{s,t}^{i}(x_i)
			=x_i+\int_{s}^{t}b(x_i)dr+\int_{s}^{t}{\int_0^1}\int_{E}c(\tau
			_{i}(w),z,x_i){{\widetilde{N}}}(dw,dz,dr),
		\end{align*}
		and we denote
		\begin{equation*}
			\Delta b=b(x_1)-b(x_2),\quad \Delta c(w,z)=c(\tau _{1}(w),z,x_1)-c(\tau
			_{2}(w),z,x_2).
		\end{equation*}
		
		\textbf{Step 2: }We use It\^{o}'s formula and the identity $\left\vert
		y+\Delta c\right\vert ^{2}-\left\vert y\right\vert ^{2}-2\left\langle
		y,\Delta c\right\rangle =\left\vert \Delta c\right\vert ^{2},$ in order to
		get%
		\begin{align*}
		\mathbb{E}\left[\left\vert \overline{\mathcal{X}}_{s,t}^{1}(x_1)-\overline{\mathcal{X%
		}}_{s,t}^{2}(x_2)\right\vert ^{2}\right] &=\left\vert x_1-x_2\right\vert
		^{2}+2\E\int_{s}^{t}\left\langle \overline{\mathcal{X}}_{s,r}^{1}(x_1)-%
		\overline{\mathcal{X}}_{s,r}^{2}(x_2),\Delta b\right\rangle dr \\
		&+\E\int_{s}^{t}\int_{0}^{1}\int_{E}\left\vert {\Delta
			(c)}(w,z)\right\vert ^{2}dwd{\rho}(z)dr.
		\end{align*}
		Standard estimates give $\mathbb{E}\left[\left\vert \overline{\mathcal{X}}%
		_{s,t}^{i}(x_i)-x_i\right\vert ^{2}\right]\leq C {(1+|x_i|^2+\|\mu_i\|_2^2)} (t-s)$ where $C$ depends on $Q_2$ and $L_b$. Thus, using the contraction
		property also, we get  for $C=C_0(1+\sum_{i=1}^2|x_i|+\|\mu_i\|_2)|x_1-x_2|$ 
		\begin{align*}
			\E\int_{s}^{t}\left\langle \overline{\mathcal{X}}_{s,r}^{1}(x_1)-\overline{%
			\mathcal{X}}_{s,r}^{2}(x_2),\Delta b\right\rangle dr &\leq
		C(t-s)^{3/2}+{2}\int_{s}^{t}\left\langle x_1-x_2,\Delta b\right\rangle dr \\
		&\leq C(t-s)^{3/2}-{2}\bar{b}(t-s)\left\vert x_1-x_2\right\vert ^{2}.
		\end{align*}
		And by (\ref{VL5}) 
		\begin{equation*}
			\mathbb{E}[\left\vert \Delta (c)(w,z)\right\vert ^{2}]\leq \overline{c}%
			^{2}(z) \int_0^1 (\left\vert \tau _{1}-\tau_{2}\right\vert ^{2}+\left\vert
			x_1-x_2\right\vert ^{2}) dw.
		\end{equation*}%
		Using the definition of $\tau $, we have 
		\begin{equation*}
			\int_{0}^{1}\left\vert \tau _{1}(w)-\tau _{2}(w)\right\vert ^{2}dw=W_{2}^{2}(\mu_1
			,\mu_2 ).
		\end{equation*}%
		We conclude that%
		\begin{equation*}
			\int_{s}^{t}\int_{0}^{1}\int_{E}\left\vert \Delta c(r,w,z)\right\vert
			^{2}dwd\rho (z)dr\leq {Q}_2(W_{2}^{2}(\mu_1 ,\mu_2 )+\left\vert x_1-x_2\right\vert
			^{2})(t-s).
		\end{equation*}%
		Putting all this together
		\begin{align*}
			&W_{2}^{2}(\mathcal{X}_{s,t}(x_1,\mu_1 ),\mathcal{X}_{s,t}(x_2,\mu_2 )) 
			\leq\mathbb{E}\left[\left\vert \overline{\mathcal{X}}_{s,t}^{1}(x_1)-\overline{%
				\mathcal{X}}_{s,t}^{2}(x_2)\right\vert ^{2}\right]\\
			&\leq \left\vert x_1-x_2\right\vert
			^{2}(1-({ 2  }\bar{b}-{Q}_2)(t-s)) 
			+{Q}_2W_{2}^{2}(\mu_1 ,\mu_2 )(t-s)\\&+C_0(1+\sum_{i=1}^2|x_i|^2+ {\Gamma_2^2(\mu_1,\mu_2)})(t-s)^{3/2}.
		\end{align*}

		Finally, we get \eqref{EX1.2} following the same proof, but working directly with the measure $N_\mu$ instead of $N$.	
	\end{proof}

	In order to study the boundedness of the moments,  we have the following result.
	
	\begin{assumption}\label{ass:Boltz}
		There exists an even number $p\geq 2$, such that for every $q=2,\ldots,p$, $q$ even, it holds
		\begin{equation}
			\label{vlq}
			\begin{array}{l}
				\left\vert c(v,z,x)\right\vert ^{q} \leq 2^{q-1}{\overline{c}^q(z)}(1+\left\vert
				v\right\vert ^{q}+\left\vert x\right\vert ^{q})\mbox{ with }\smallskip\\ 
				\displaystyle
				Q_{q}:=\int_{E}{\overline{c}^q(z)}\rho (dz)<\infty.
			\end{array}
		\end{equation}
	\end{assumption}

	\begin{lemma} \label{lemma:BoltzB}	
		Suppose that \eqref{VL4}--\eqref{VL5'} are in force. Moreover, suppose that Assumption~\ref{ass:Boltz} holds and, for the even integer $p$ therein, set 
		$
		\kappa_{p}=2^{{p}-3}(p-1)(2Q_2+2^{p-1}Q_{p}).
		$	
		Let $X$ be $\mathcal{F}_s$-measurable with $\mathcal{L}(X)=\mu\in\mathcal{P}_p(\R^d)$. Then for $t-s\leq 1$ and even integer $p$ one has
		\begin{equation}\label{estX3}
			\begin{array}{rl}
				\E[ |\X_{s,t}(X,\mu)|^p]
				\leq &
				\displaystyle \Big(1-p\Big(\bar b_{p,\delta}-C(t-s)\Big)(t-s)\Big)\|\mu\|_p^p\\
				&+C(1+\|\mu\|_{p-2}^{p})(t-s),
			\end{array}
		\end{equation}
		where,  for $\delta>0$,  $ 
		\bar b_{p,\delta}:=	\bar{b}-\delta-\kappa_{p}$, 
		$C$ is a positive constant depending only on $p,b,\delta$ and $\|\mu\|_0:=1$.
	\end{lemma}
	
	The proof of Lemma \ref{lemma:BoltzB} is postponed in Appendix \ref{app:BoltzB}.

	We recall now the two chains of interest in the present paper: given $X_{0}$
	and $Y_{0}$ we construct by recurrence %
	\begin{equation*}
		X_{k+1}=\psi _{k+1}(X_{k},\mathcal{L}(X_{k}))\qquad \mbox{and}\qquad Y_{k+1}=\psi
		_{k+1}(Y_{k},\rho _{k}(Y))
	\end{equation*}%
	with $\rho _{k}(Y)=\frac{1}{t_{k}}\sum_{i=1}^{k}\gamma _{i}\delta _{Y_{i}}.$ 
	
	We are now ready to state our approximation result for the invariant measure $\mu_*$ of the flow $\theta$ associated with the Boltzmann equation \eqref{EX1.SDE}.
	
	\begin{theorem}
		Suppose that the hypotheses \eqref{VL4}--\eqref{VL5'} are in force. Suppose moreover that Assumption \ref{ass:Boltz} is satisfied with $p'\geq 4$ and that
		for every $p=2,\ldots,p'$, $p$ even,
		\begin{equation}
			\label{kappa-p}
			\bar{b}>\kappa_{p},\quad \mbox{where}\quad \kappa_{p}=2^{{p}-3}(p-1)(2Q_2+2^{p-1}Q_{p}).
		\end{equation}	
		Let $\gamma _{k}=\frac{1}{(1+k)^{\beta }}$  with $%
		1\geq\beta >\frac{2}{3 }$.
		Then, for every  
		\begin{equation*}
			\zeta \in \Big(0,1-\frac{Q{ _2  }}{\bar{b}}\Big)\cap \Big(0,\frac{p^{\prime }-2}{
				(d+5/2)(p^{\prime }-2)+2(d+2)}\Big],
		\end{equation*}
		one has 
		\begin{align*}
			\mbox{if $\beta=1$:}&\quad \mathbb{E}[W_{2}^{2}(\rho _{k}(Y),\mu _{\ast
			})]\leq \frac{C}{(\ln k)^\zeta}, \\
			\mbox{if $1>\beta>\frac 2{3}$:} &\quad \mathbb{E}[W_{2}^{2}(\rho
			_{k}(Y),\mu _{\ast })]\leq \frac{C}{k^{(1-\beta)\zeta}}, \end{align*}
		$\mu_*$ denoting the invariant measure associated with the stochastic equation \eqref{EX1.SDE}.
	\end{theorem}
	
	\begin{proof}
		As already observed in Remark \ref{rem:Boltz}, there exists a unique  invariant measure $\mu_*$ for $\theta$  and it satisfies $\lim_{k\rightarrow \infty }W_2(\Theta^\pi_{t_r,t_k}(\mu),\mu_*)=0$ for all $\mu \in \mathcal{P}_2(\R^d)$. So, it only remains to check that the hypotheses of Theorem \ref{MAIN0} are verified.
		
		Consider the random applications  $\psi_k$ defined in \eqref{Ex1-psi}. We apply Lemma \ref{LemmaEx1}:
		by \eqref{kappa-p}, one has $\bar{b}>\kappa_2={2}Q_2$, hence \eqref{EX1.1} ensures that the weak self-contraction property in \eqref{New2} is satisfied with $p=2$, $b={2}\bar{b}$, $\alpha={Q}_2$ and $\varepsilon=1/2$. The asymptotic commuting property \eqref{New6}
		holds with $\delta_{h,r,k}=\gamma_r^{1/2}$ by Remark~\ref{rem:Boltz}. This means that $\mathbf{A}_2({2\bar{b}, Q_2}, 1/2)$ holds.
		And \eqref{EX1.2}  gives that the stronger contraction in \eqref{na1a} holds with $p=2$ and the same parameters 
		$(\bar{b}, Q_2, 1/2)$. 	So, it remains to prove that $\sup_{k>r}\E[|X_{r,k}(X)|^{p'}]=\sup_{k>r}\|\Theta^\pi_{t_r,t_k}(\mu)\|_{p'}^{p'}<+\infty$, where $\Theta^\pi_{t_r,t_k}$ is defined in \eqref{Ex1-Theta}.
		
		We use Lemma \ref{lemma:BoltzB}. We choose $\delta>0$ and $r$ such that for every $k>r$ one has $\bar b_{p,\delta}-C\gamma_k=\bar{b}-\delta-\kappa_p-C\gamma_k>0$ for every $p$ even, $p=2,\ldots p'$, where $C$ denotes the constant appearing in  \eqref{estX3}.
		When $p=2$, \eqref{estX3} fulfills  the Foster-Lyapunov condition \eqref{New0bis},  and Lemma \ref{lem:29} gives that $\sup_{k>r}\|\Theta^\pi_{t_r,t_k}(\mu)\|_2^2<+\infty$. Then, for $p=4,\ldots,p'$, $p$ even, the result follows by an iterative application of Lemma \ref{lem:29}: simply apply \eqref{estX3} and use the fact that
		$\sup_{k>r}\|\Theta^\pi_{t_r,t_k}(\mu)\|_{p-2}^{p-2}<+\infty$. \end{proof}

	\subsection{A Neuron model}
	We consider $X$ to be the solution of the following mean field
	type stochastic equation in dimension one: 
	\begin{equation}\label{SJE_Neuron}
		X_{s,t}=X_{s,s}+\int_{s}^{t}b(X_{s,r})dr+J
		\int_{s}^{t}\E[f(\mathcal{X}_{s,u})]du-\int_{s}^{t}\int_{\mathbb{R}_{+}}%
		{X}_{s,u-}1_{z\leq f({X}_{u}-)}dN(u,z).
	\end{equation}%
	Here,  $J\geq 0$, $N$ is random Poisson measure with compensator
	measure given by the Lebesgue measure on $\mathbb{R}_{+}\times \mathbb{R}%
	_{+} $, and $X_{s,s}\geq 0$ is an adapted random variable. We assume that $f:\mathbb{R}_{+}\rightarrow \mathbb{R}_{+}$ is
	bounded and Lipschitz continuous with constant $C_{f}$ and $f(0)=0$. 

	We consider the following approximation scheme for $x\geq 0$
	\begin{equation*}
		\mathcal{X}_{s,t}(x,\mu )={x}+\left(b({x})+J\int fd\mu \right)(t-s)-\int_{s}^{t}\int_{\mathbb{R}%
			_{+}}\mathcal{X}_{s,u}(x,\mu )1_{z\leq f(\mathcal{X}_{s,u}(x,\mu ))}dN(u,z).
	\end{equation*}
	Notice that the above expression is not
	amenable to simulation because the term $\mathcal{X}_{s,u}(x,\mu ),s\leq u\leq t$ appears in the integral with respect to the Poisson point measure $%
	dN$.  However, since the Poisson
	point measure has a finite number of jumps in this interval of time, the
	solution of the above equation may be explicitly constructed (so it is possible to simulate it).
	This specific form of the approximation scheme is convenient for some estimates (for example to handle the quantity $I_3$ in the proof of Lemma \ref{lemma_neuron}).
	 We work under the following hypotheses.
	\begin{assumption}\label{assumptions_Neurons}
		\begin{itemize}
			\item $f:\mathbb{R}_{+}\rightarrow \mathbb{R}_{+}$ is
			bounded, satisfies $f(0)=0$ and Lipschitz continuous with constant $C_{f}$.
			\item $b:\R_+\to \R_+$ is Lipschitz continuous with Lipschitz constant $C_{b} $.
			\item There exists $\bar{b}>0$ such that for every $x,y\in \mathbb{R}%
			_{+}$%
			\begin{align}
			&\mathrm{sgn}(y-x)(b(y)-b(x))+F(x,y)\leq -\bar{b}\left\vert x-y\right\vert
			\label{Neuron1}
			\end{align}
			with  $F(x,y) =-|y-x|f(y)\vee f(x)+x(f(y)-f(x))^{+}+y(f(x)-f(y))^{+}$.
		\end{itemize}
	\end{assumption}
	Taking $y=0$ in~\eqref{Neuron1}, the above conditions imply that there exists $K>0$ such that $|b(x)|\leq
	K(1+x)$, $x\in \mathbb{R}_{+}$. Under Assumption~\ref{assumptions_Neurons}, $X_{s,t}$ and $\mathcal{X}_{s,t}(x,\mu)$
	have a unique solution and are nonnegative. This equation is related to
	the neuronal model introduced and discussed in~\cite{DGLP}, \cite{FoLo}, and~\cite{CTV}.

	We define 
	\begin{align*}
			\theta _{s,t}(\mu ) &=\mathcal{L}(X_{s,t})\quad \mathcal{L}%
		(X_{s,s})=\mu , \\
		\Theta _{s,t}(\mu ) &=\mathcal{L}(\mathcal{X}_{s,t}(X,\mu ))\quad \mathcal{L}(X)=\mu ,
	\end{align*}
	and in order to use our framework, we let
	\begin{equation*}
		\psi _{k}(x,\mu )=\mathcal{X}_{t_{k-1},t_{k}}(x,\mu ).
	\end{equation*}
	Our aim is to prove the self-coupling property (\ref{New2}) and the estimate~\eqref{na1a} for $p=1$, which is a consequence of the next lemma.

	\begin{lemma} \label{lemma_neuron}
		Let Assumption~\ref{assumptions_Neurons} hold. Then we have for $x,y\ge 0$ and $\mu,\nu \in \cP_1(\R_+)$,
		\begin{align}
			&W_{1}(\mathcal{X}_{s,t}(x,\mu ),\mathcal{X}_{s,t}(y,\nu ))  \label{New2poisson}\\
			&\leq (1-\overline{b}(t-s))\left\vert x-y\right\vert +JC_f (t-s) W_{1}(\mu ,\nu
			)+C(1+|x|+|y|)(t-s)^{2} \notag
		\end{align}
		and  {
			$$ \E[|\mathcal{X}_{s,t}(x,\mu )-\mathcal{X}_{s,t}(y,\mu )|]\le (1-\overline{b}(t-s))\left\vert x-y\right\vert +C(1+|x|+|y|)(t-s)^{2} .$$
		}
	\end{lemma}

	\begin{proof}
		Since $\left\vert b(x)\right\vert \leq K(1+\left\vert
		x\right\vert )$ and $f$ is bounded, we first easily check that 
		\begin{equation}\label{estim_neuron_short_time}
			\E[\left\vert \mathcal{X}_{s,t}(x,\mu )-x\right\vert ]\leq C(t-s)(|x|+1).	
		\end{equation}
		We now define $\mathcal{Y}_{s,t}:=\mathcal{X}_{s,t}(x,\mu )-\mathcal{X}_{s,t}(y,\nu
		) $ which satisfies 
		\begin{equation*}
			\mathcal{Y}_{s,t}=x-y+(b(x)-b(y))(t-s)+J(t-s)\int f(d\mu -d\nu
			)-\int_{s}^{t}\int_{\mathbb{R}_{+}}c(r,z)dN(r,z),
		\end{equation*}%
		with $c(r,z)=\mathcal{X}_{s,r}(x,\mu )1_{z\leq f(\mathcal{X}_{s,r}(x,\mu ))}-\mathcal{X}_{s,r}(y,\nu
		)1_{z\leq f(\mathcal{X}_{s,r}(y,\nu ))}.$
		
		Then, using It\^{o}'s formula for $\phi (x)=\left\vert x\right\vert $ we get%
		\begin{align*}
			|\mathcal{Y}_{s,t}|& =|x-y|+\int_{s}^{t}\mathrm{sgn}(\mathcal{Y}%
			_{s,r})(b(x)-b(y))dr \\
			& +J\int_{s}^{t}\mathrm{sgn}(\mathcal{Y}_{s,r})\int f(d\mu -d\nu )dr \\
			& +\int_{s}^{t}\int_{\mathbb{R}_{+}}|\mathcal{Y}_{s,r-}-c(r,z)|-|\mathcal{Y}%
			_{s,r-}|dN(r,z) \\
			& =:|x-y|+\sum_{j=1}^{3}I_{j}(s,t).
		\end{align*}
		We denote $\widehat{x}_{r}=\mathcal{X}_{s,r}(x,\mu ),\widehat{y}_{r}=\mathcal{X}_{s,r}(y,\nu )$. From the Lipschitz property of $b$ and~\eqref{estim_neuron_short_time}
		we get
		\begin{equation*}
			\mathbb{E}[I_{1}(s,t)]\leq \int_{s}^{t}\mathrm{sgn}(\mathcal{Y}_{s,r})(b(\widehat{x}%
			_{r})-b(\widehat{y}_{r}))dr+C(1+|x|+|y|)(t-s)^{2}.
		\end{equation*}%
		The Lipschitz property of $f$ gives
		\begin{equation*}
			\mathbb{E}[I_{2}(s,t)]\leq JC_{f}W_{1}(\mu ,\nu )(t-s).
		\end{equation*}
		
		We consider the estimate for $\E[I_{3}(s,r)]$. We have 
		\begin{align*}
			\mathbb{E}[I_{3}(s,r)]& =\mathbb{E}\left[\int_{s}^{t}\int_{\mathbb{R}_{+}}|\mathcal{Y}%
			_{s,r-}-c(r,z)|-|\mathcal{Y}_{s,r-}|dN(r,z)\right] \\
			& =\mathbb{E}\left[\int_{s}^{t}\int_{\mathbb{R}_{+}}|\mathcal{Y}_{s,r-}-c(r,z)|-|\mathcal{Y%
			}_{s,r-}|dzdr\right].
		\end{align*}
		
		Our aim is to give an explicit expression for $\int_{0}^{\infty }|\mathcal{Y}%
		_{s,r-}-c(r,z)|-|\mathcal{Y}_{s,r-}|dz.$
		
		We fix $r\in (s,t),$ we denote $\widehat{x}=\widehat{x}_{r}=\mathcal{X}_{s,r}(x,\mu ),%
		\widehat{y}=\widehat{y}_{r}=\mathcal{X}_{s,r}(y,\nu )$ and then%
		\begin{equation*}
			\int_{0}^{\infty }|\mathcal{Y}_{s,r-}-c(r,z)|-|\mathcal{Y}%
			_{s,r-}|dz=\int_{0}^{\infty }|\widehat{x}-\widehat{y}-\widehat{x}1_{z\leq f(%
				\widehat{x})}+\widehat{y}1_{z\leq f(\widehat{y})}|-|\widehat{x}-\widehat{y}%
			|dz.
		\end{equation*}%
		Suppose first that $f(\widehat{x})\leq f(\widehat{y}).$ Then the above
		quantity is equal to $I+J$ with 
		\begin{align*}
			I &=\int_{0}^{f(\widehat{x})}(|\widehat{x}-\widehat{y}-\widehat{x}+\widehat{y}|-|\widehat{x}-\widehat{y}%
			|)dz=-|\widehat{x}-\widehat{y}|f(\widehat{x}) \\
			J &=\int_{f(\widehat{x})}^{f(\widehat{y})}(|\widehat{x}-\widehat{y}+\widehat{%
				y}|-|\widehat{x}-%
			\widehat{y}|)dz=(|\widehat{x}|-|\widehat{x}-\widehat{y}|)(f(\widehat{y})-f(%
			\widehat{x})).
		\end{align*}
		So we obtain 
		\begin{align*}
			I+J &=-|\widehat{x}-\widehat{y}|f(\widehat{x})+(|\widehat{x}|-|\widehat{x}-%
			\widehat{y}|)(f(\widehat{y})-f(\widehat{x})) \\
			&=-|\widehat{x}-\widehat{y}|f(\widehat{y})+|\widehat{x}|(f(\widehat{y})-f(\widehat{x})) =F(\widehat{x},\widehat{y}).
		\end{align*}
		And, by symmetry, this equality is also true for $f(\widehat{x})\geq f(\widehat{y})$ as
		well. We conclude that%
		\begin{align*}
			\mathbb{E}[I_{3}(s,t)] &=\mathbb{E}\left[\int_{s}^{t}\int_{\mathbb{R}_{+}}|\mathcal{Y}%
			_{s,r-}-c(r,z)|-|\mathcal{Y}_{s,r-}|dzdr \right]\\
			&=\mathbb{E}[\int_{s}^{t}F(\widehat{x}_{r},\widehat{y}_{r})dr].
		\end{align*}
		Using our hypothesis, it follows from \eqref{estim_neuron_short_time} that
		\begin{align*}
			\mathbb{E}[I_{1}(s,t)+I_{3}(s,t)] &\leq \mathbb{E}\int_{s}^{t}(\mathrm{sgn}(\mathcal{Y}%
			_{s,r})(b(\widehat{x}_{r})-b(\widehat{y}_{r}))+F(\widehat{x}_{r},\widehat{y}%
			_{r}))dr+C(1+|x|+|y|)(t-s)^{2} \\
			&\leq -\mathbb{E}\int_{s}^{t}\overline{b}\left\vert \widehat{x}_{r}-\widehat{y}%
			_{r}\right\vert dr+C(1+|x|+|y|)(t-s)^{2} \\
			&\leq -\overline{b}\left\vert x-y\right\vert (t-s)+C(1+|x|+|y|)(t-s)^{2}.
		\end{align*}
		We conclude that
		\begin{align*}
			&	W_{1}(\mathcal{X}_{s,t}(x,\mu ),\mathcal{X}_{s,t}(y,\nu )) \\\leq &\mathbb{E}|\mathcal{Y}_{s,t}| \\
			\leq &|x-y|(1-\overline{b}(t-s))+JC_f(t-s)W_{1}(\mu ,\nu )+C(1+|x|+|y|)(t-s)^{2}.
		\end{align*}
		so (\ref{New2poisson}) is proved. Taking $\mu=\nu$, we also get the other estimates.
	\end{proof}

	We consider first some moment bounds for $\X_{s,t}(x,\mu)$.
	
	\begin{lemma}\label{lemmaNeu1} 
		Let $t-s<1$, $x\in\R^+$ and $\mu\in\mathcal{P}_0(\R^+)$ (i.e.  a probability measure on $\R_+$). For $q\geq 1$, the following $L^q$ estimates hold:
		\begin{align}
			&\big|\E[\X_{s,r}(x,\mu)^ q]- x^ q\big|\leq C_q(1+x^ q)(t-s)\label{neu3}\\
			&\E[|\X_{s,r}(x,\mu)-x|^q]\leq C_q(1+x^ q)(t-s)\label{neu4}
		\end{align}
		where $C_q$ depends on $q$, the drift $b$, $J$ and $\|f\|_\infty$.
		
	\end{lemma}
	\begin{proof} The proof of  both statements is similar. The proof is done by induction on $q$ and first uses It\^o's formula for the $X_{s,r}(x,\mu)^ q- x^ q $ in the first case and for $|X_{s,r}(x,\mu)- x|^ q $ in the second. Then using the Lipschitz and boundedness properties of the coefficients one finishes by using Gronwall's lemma. 
		
		Let $x\in\R^ +$ and {$\mu\in \mathcal{P}_0(\R_+)$ (it is not necessary that $\mu\in \mathcal{P}_1(\R_+)$ because $f$ is bounded)}. We write $\X_{s,t}$ in place of $\X_{s,t}(x,\mu)$. We also use the notation $\mu(f)=\int fd\mu\leq \|f\|_\infty$. By applying  It\^o's formula,
		$$
		\X_{s,t}^ q=x^ q+\int_s^t q\X_{s,r}^ {q-1}(b(x)+J\mu(f))dr
		+\int_s^ t\int_{\mathbb{R}_+} \big[(\X_{s,r}-\X_{s,r}\I_{z\leq f(\X_{s,r})})^ q-\X_{s,r}^ q\Big]dN(z,r)
		$$
		Since $(\xi -\xi\I_{z\leq c})^ q-\xi^ q=-\xi^ q\I_{z\leq c}$, we get
		\begin{align*}
			\E[\X_{s,t}^ q]
			&=x^ q+\int_s^t q\E[\X_{s,r}^ {q-1}](b(x)+J\mu(f))dr-\int_s^ t \E[\X_{s,r}^qf(\X_{s,r})]dr. 
		\end{align*}
		So, there exist constants which may depend on $\|f\|_\infty $ so that 
		\begin{align*}
			\big|\E[\X_{s,t}^ q]-x^ q\big|
			\leq& C_q(1+x)\int_s^t \E[\X_{s,r}^ {q-1}]dr+\|f\|_\infty x^q(t-s)
			\\
			\leq& C_q(1+x)\int_s^t \big|\E[\X_{s,r}^ {q-1}]-x^{q-1}\big|dr+C_q(1+x^q)(t-s)\\
			&+C\int_s^ t \big|\E[\X_{s,r}^q]-x^q\big|dr.
		\end{align*}

		We can now prove \eqref{neu3} by induction on $q$. If $q=1$ then
		\eqref{neu3} follows by the Gronwall's lemma. We consider now $2\leq q\leq p$. Since \eqref{neu3} holds with $q$ replaced by $q-1$, we use it in the above inequality and apply 
		Gronwall's lemma which  gives \eqref{neu3} for general $q$. We leave the proof of the second statement to the reader.
	\end{proof}
	
	We can now state a Foster-Lyapunov type  condition.
	
	\begin{lemma}\label{lemmaNeu}
		In addition to Assumption \ref{assumptions_Neurons}, we assume that there exists an integer $p\geq 1$ such that for every $q=1,\ldots,p$ and $x\in\R^+$,
		\begin{align}
			qx^{q-1}(b(x)+J\|f\|_\infty)-x^qf(x)\leq -\bar{b}_qx^q+C_q(1+x^{q-1}).\label{neu1}
		\end{align}
		where $C_q$ denotes a positive constant. Then for $t-s\leq 1$, for every $q=1,\ldots,p$, $x\in\R_+$ and $\mu \in \mathcal{P}_0(\R_+)$ one has
		\begin{equation}\label{neu6}
			\E[\X_{s,t}^q(x,\mu)]
			\leq \Big(1-\big(\bar{b}_q-C_q(t-s)^{\frac 1q}\big)(t-s)\Big)x^ q+C_q(1+x^ {q-1})(t-s).
		\end{equation}
		Moreover, if $\mu\in\mathcal{P}_p(\R_+)$ and $X$ is $\mathcal{F}_s$-measurable with $\mathcal{L}(X)=\mu$, then for every $q=1,\ldots,p$ one has
		\begin{equation}\label{neu7}
			\E[\X_{s,t}^q(X,\mu)]
			\leq \Big(1-\big(\bar{b}_q-C_q(t-s)^{\frac 1q}\big)(t-s)\Big)\|\mu\|_q^ q+C_q(1+\|\mu\|_{q-1}^ {q})(t-s).
		\end{equation}
		In \eqref{neu6} and \eqref{neu7}, $C_q$ denotes a constant depending on  $\|f\|_\infty$, $J$ and $b$.
		Furthermore, we have that the moments of order $2$ are uniformly bounded if $\mu \in \mathcal{P}_2(\R_+)$. That is,
		$\sup_{k\geq r}\|\Theta^\pi_{t_r,t_k}(\mu)\|_2^2\leq e^{-b^*_1(t_k-t_r)}\|\mu\|_2^2+C.$
	\end{lemma}

	\begin{proof}
		
		First, notice that if $p=1$ then \eqref{neu1} follows from  Assumption \ref{assumptions_Neurons}, so the interesting case is when $p\geq 2$.

		Hereafter, for $q=1,\ldots, p$, $C_q$ will denote a constant which may vary from a line to another and which depends on $b$, $\|f\|_\infty$ and $J$. We also simplify the notation by using $ \X_{s,t}\equiv \X_{s,t}(x,\mu)$.
		
		By applying It\^o's formula and passing to the expectation, we easily get 
		\begin{align*}
			\E[\X_{s,t}^ q]
			&=x^ q+\int_s^t q\E[\X_{s,r}^ {q-1}](b(x)+J\mu(f))dr-\int_s^ t \E[\X_{s,r}^qf(\X_{s,r})]dr.
		\end{align*}
		Then,
		\begin{align*}
			\E[\X_{s,t}^ q]
			=&x^ q+\int_s^t \E\big[q\X_{s,r}^ {q-1}(b(\X_{s,r})+J\mu(f))-\X_{s,r}^qf(\X_{s,r})\big]dr\\
			&+q\int_s^ t \E\big[\X_{s,r}^{q-1}(b(x)-b(\X_{s,r}))\big]dr, 
		\end{align*}
		and by using \eqref{neu1} with $\mu(f)\le\|f\|_\infty$ ,
		\begin{align*}
			&\E[\X_{s,t}^ q]\leq x^q-\bar b_q x^q(t-s)+ C_q(1+x^{q-1})(t-s)+ I_1+I_2,\quad \text{ with}\\
			&I_1=\int_s^t \E\big[-\bar b_q(\X_{s,r}^ {q}-x^q)+C_q(\X_{s,r}^{q-1}-x^{q-1})\big]dr\\
			&I_2=q\int_s^ t \E\big[\X_{s,r}^{q-1}(b(x)-b(\X_{s,r}))\big]dr.
		\end{align*}
		We estimate now $I_1$ and $I_2$. We have
		\begin{align*}
			I_1\leq \bar b_q\int_s^t \big|\E\big[\X_{s,r}^ {q}\big]-x^q\big|dr +C_q\int_s^t \big|\E\big[\X_{s,r}^{q-1}]-x^{q-1}\big|dr,
		\end{align*}
		and by using \eqref{neu3}
		\begin{align*}
			I_1\leq C_q(1+x^q)(t-s)^2+C_q \big(1+x^{q-1})(t-s)^2.
		\end{align*}
		Concerning $I_2$, we first use the Lipschitz property of $b$ and the H\"older inequality: 
		\begin{align*}
			I_2
			\leq &C_q\int_s^ t \E\big[\X_{s,r}^{q-1}|\X_{s,r}-x|\big]dr 
			\leq C_q\int_s^ t \E\big[\X_{s,r}^{q}\big]^{\frac{q-1}{q}}\E\big[|\X_{s,r}-x|^q\big]^{\frac 1q}dr .
		\end{align*}
		Then, by applying \eqref{neu3} and \eqref{neu4}, we get
		\begin{align*}
			I_2
			\leq &C_q\int_s^t(1+x^q)(r-s)^{\frac 1q}dr\leq C_q(1+x^q)(t-s)^{1+\frac 1q},
		\end{align*}
		and, by rearranging all terms, we obtain
		\begin{align*}
			\E[\X_{s,t}^ q]
			\leq & x^q\Big(1-\big(\bar b_q-C_q(t-s)^{\frac 1q}\big)(t-s)\Big)+C_q(1+x^{q-1})(t-s).
		\end{align*}
	
		Now, we will use the above  Foster-Lyapunov type  condition \eqref{neu7} to prove the uniform boundedness of the  2-moments. 
		
		First,  we take  $r$ such that
		$$
		b^*_q=\bar{b}_q-C_q\gamma_r^{\frac 1q}>0 \text{ for every } q=1,\ldots, p.
		$$
		Then, for every $k\geq r$, recalling that $\Theta_k(\mu)=\mathcal{L}(\X_{t_k,t_{k+1}}(X))$, \eqref{neu7} gives
		\begin{equation}
			\label{neu7bis}
			\|\Theta_k(\mu)\|_q^q
			\leq (1-b^*_q\gamma_k)\|\mu\|_q^q +C_q(1+\|\mu\|_{q-1}^q)\gamma_k.
		\end{equation}
		Now the iteration on $q$ starts.
		
		First, consider $q=1$. Then, for every $k\geq r$, 
		$$
		\|\Theta_k(\mu)\|_1
		\leq (1-b^*_1\gamma_k)\|\mu\|_1 +C_1\gamma_k.
		$$
		and a standard application of Proposition \ref{lem:29} gives the uniform 1-moment property.
		Take now $q=2$. Then, for $k\geq r$, \eqref{neu7bis} gives
		\begin{equation*}
			\|\Theta_k(\mu)\|_2^2
			\leq (1-b^*_2\gamma_k)\|\mu\|_2^2 +C_2(1+\|\mu\|_1^2)\gamma_k.
		\end{equation*}
		So, 
		$$
		\|\Theta_k(\mu)\|_2^2
		\leq (1-b^*_2\gamma_k)\|\mu\|_2^2 +\mathrm{const}\times \gamma_k,
		$$
		where $\mathrm{const}= C_2(1+\sup_{k\geq r} \|\Theta^\pi_{t_r,t_k}(\mu)\|_1^2)$ and again Proposition \ref{lem:29} gives that\\
		$\sup_{k\geq r}\|\Theta^\pi_{t_r,t_k}(\mu)\|_2^2\leq e^{-b^*_2(t_k-t_r)}\|\mu\|_2^2+C$. 
	\end{proof}
	
	We now study an example.
	
	\begin{example}\label{example-neuron}\rm 
		Take
		$$
		b(x)=\frac{B}{1+x}\quad\text{and}\quad f(x)=\frac{ax}{1+x},
		$$
		where $a,B>0$.
		Then $b$,$f$ satisfy Assumption \ref{assumptions_Neurons}. Moreover, straightforward computations give that,
		for every $q\geq 1$,
		$$
		qx^{q-1}(b(x)+J\|f\|_\infty)-x^qf(x)\leq -ax^q+C_q(1+x^{q-1}),
		$$
with $C_q=\max\{q(B+Ja)+a , a\}$. Therefore, \eqref{neu1} holds for every $q\geq 1$.		
	\end{example}

We are now ready for the main result.
	
	\begin{theorem}
		Suppose that Assumption~\ref{assumptions_Neurons} holds with $\overline{b}-C_{f}J>0$. Suppose moreover that \eqref{neu1} is satisfied for every $q=1,\dots,p'$ with $p'\geq 2$. Let
		$\gamma _{k}=\frac{1}{(1+k)^{\beta }}$  with $%
		1\geq\beta >\frac{1}{2 }$,
		then, for every  
		\begin{equation*}
			\zeta \in \Big(0,1-\frac{C_fJ}{\bar{b}}\Big)\cap \Big(0,\frac{p'-1}{ 5 (p'-1)+4}\Big],
		\end{equation*}
		one has 
		\begin{align*}
			\mbox{if $\beta=1$:}&\quad \mathbb{E}[W_{1}(\rho _{k}(Y),\mu _{\ast
			})]\leq \frac{C}{(\ln k)^\zeta}, \\
			\mbox{if $1>\beta>\frac 1{2}$:} &\quad \mathbb{E}[W_{1}(\rho
			_{k}(Y),\mu _{\ast })]\leq \frac{C}{k^{(1-\beta)\zeta}}, \end{align*}
		$\mu_*$ denoting the invariant measure associated with the stochastic equation \eqref{SJE_Neuron}.
	\end{theorem}
	\begin{proof} Let $b_*=\overline{b}-C_{f}J>0$.
		It has been shown in~\cite[Theorem 5.3]{ABKH} that under Assumption~\ref{assumptions_Neurons},  $\theta$ and $%
		\Theta $ satisfy a Foster-Lyapunov criterion for  $\|\cdot \|_1$ and  are $(1,b_{\ast },1)$--coupled in the sense of {\cite[Definition 2.2]{ABKH}}, 	which exactly means that the condition (\ref{New6'}) holds with $p=1$ and $\varepsilon=1$. Consequently, by Proposition \ref{prop:commuting} the asymptotic commuting property \eqref{New6} holds with $\delta_{h,r,k}=\gamma_r$. From Lemma~\ref{lemma_neuron}, we get that Assumption  $\mathrm{\mathbf{A}}_1(\bar{b},C_{f}J,1)$  is satisfied.
		
		It is also shown in~\cite[Theorem 5.3]{ABKH} that $\theta$ has a unique invariant measure $\mu_*$.  Moreover, by \cite[Lemma~2.6]{ABKH}, we get that $W_1(\Theta^\pi_{t_0,t_k}(\mu),\mu_*)\to_{k\to \infty} 0$ for any $\mu \in \mathcal{P}_1(\R^d)$ (i.e. the approximation scheme converges well to the invariant measure of the continuous process). Now, we can apply Theorem~\ref{MAIN0} with $\hat{p}=\frac{p'-1}{\frac 52 (p'-1)+2}$, $b=\bar{b}$, $\alpha=C_fJ$ and $\varepsilon=1$ and get the claim.
	\end{proof}

	\appendix
	
	\section{On the asymptotic commutativity assumption}
	
	\label{app:C}
	
	We study here a sufficient condition in order to achieve the asymptotic
	commutativity property (\ref{New6}):
	
	\begin{assumption}
		\label{ass:C} There exists a continuous homogeneous flow $\theta$ such that
		for every $k\in \mathbb{N}$ and $\mu ,\nu \in \mathcal{P}_{p}(\mathbb{R}%
		^{d})$ 
		\begin{equation}
			W_{p}^{p}(\Theta _{k}(\mu ),\theta _{t_{k},t_{k+1}}(\nu ))\leq (1-b\gamma
			_{k})W_{p}^{p}(\mu ,\nu )+c_{\ast }\Gamma _{p}^{p}(\mu ,\nu )\gamma
			_{k}^{1+\varepsilon }.  \label{New6'}
		\end{equation}
	\end{assumption}
	
	Notice that, using a standard iteration argument, if $\theta$ satisfies
	Assumption \ref{ass:C} then, for every $0\leq r\leq k$, 
	\begin{equation}
		W_{p}^{p}(\Theta _{t_{r},t_{k}}^{\pi }(\mu ),\theta _{t_{r},t_{k}}(\nu
		))\leq e^{-b(t_{k}-t_{r})}W_{p}^{p}(\mu ,\nu )+C\Gamma _{p}^{p}(\mu ,\nu
		)\gamma _{{k}}^{\varepsilon }.  \label{eq:New6a}
	\end{equation}
	
	\begin{remark}
		Property \eqref{eq:New6a} is of the same nature as the self contraction
			assumption~(\ref{New2''}) but it uses $\Theta _{k}(\mu )$ and $\theta
			_{t_{k},t_{k+1}}(\nu )$ instead of $\Theta _{k}(\mu )$ and $\Theta _{k}(\nu
			).$ In the concrete examples discussed in Section \ref{sect:examples},
				we check that this hypothesis holds true.
			
	\end{remark}
	
	\begin{proposition}\label{prop:commuting}
		Let $\psi =(\psi _{k})_{k\in \mathbb{N}}$ be $(b,\alpha ,\varepsilon )$ self-coupled, with $\overline{b}:=b-\alpha >0$, and suppose
		that 
		\begin{equation}\label{uniform_bound_momp}
			\exists C\in \mathbb{R}_+, \forall \mu \in \mathcal{P}_p(\R^d), \forall 0\le k\le n, \  \|\Theta_{t_k,t_n}(\mu)\|_p^p+\|\theta_{t_k,t_n}(\mu)\|_p^p\le C (1+\|\mu\|_p^p).
		\end{equation}
		If
		Assumption \ref{ass:C} holds then the asymptotic commuting property %
		\eqref{New6} is satisfied for $\psi $ with $\delta _{h,k,n}=\gamma
		_{k}^{\varepsilon }$. In particular, there exists a $(\psi ,p)-$ stationary
		measure $\mu _{\ast }$ for $\psi .$ And, if $\theta $ is continuous in $(%
		\mathcal{P}_p(\mathbb{R} ^{d}),W_{p})$, then $\mu _{\ast }$ is an invariant
		measure for $\theta $ as well, that is, $\theta _{t_{k},t_{k+1}}(\mu _{\ast
		})=\mu _{\ast }.$
	\end{proposition}
	Let us note that the Foster-Lyapunov criterion in Proposition~\ref{lem:29} gives a sufficient condition to have~\eqref{uniform_bound_momp}.
	\begin{proof}Let $h\le r \le k$ be natural numbers. We first use the triangle inequality:
	\begin{align*}
		&W_{p}(\Theta _{t_{r},t_{k}}^{\pi }\circ \Theta
		_{t_{h},t_{r}}^{\pi }(\mu ),\Theta _{t_{h},t_{r}}^{\pi }\circ \Theta
		_{t_{r},t_{k}}^{\pi }(\mu )) \le W_{p}(\Theta _{t_{h},t_{r}}^{\pi }\circ \Theta _{t_{r},t_{k}}^{\pi
		}(\mu ),\Theta _{t_{h},t_{r}}^{\pi }\circ \theta _{t_{r},t_{k}}(\mu
		)) \\
		&+W_{p}(\Theta _{t_{h},t_{r}}^{\pi }\circ \theta _{t_{r},t_{k}}(\mu
		),\theta _{t_{h},t_{k}}(\mu )) 
		+W_{p}(\theta _{t_{h},t_{k}}(\mu ),\Theta _{t_{r},t_{k}}^{\pi }\circ
		\Theta _{t_{h},t_{r}}^{\pi }(\mu )).
	\end{align*}
	In order to prove \eqref{New6} we will analyze the order of
	each term in this decomposition :
	\begin{align*}
		I&= W_{p}^{p}(\Theta _{t_{h},t_{r}}^{\pi }\circ \Theta _{t_{r},t_{k}}^{\pi
		}(\mu ),\Theta _{t_{h},t_{r}}^{\pi }\circ \theta _{t_{r},t_{k}}(\mu
		)), \\II&=W_{p}^{p}(\Theta _{t_{h},t_{r}}^{\pi }\circ \theta _{t_{r},t_{k}}(\mu
		),\theta _{t_{h},t_{k}}(\mu )), \\
		III& =W_{p}^{p}(\theta _{t_{h},t_{k}}(\mu ),\Theta _{t_{r},t_{k}}^{\pi }\circ
		\Theta _{t_{h},t_{r}}^{\pi }(\mu )).
	\end{align*}%
	By the self coupling property, using~\eqref{estimate_Thetapi} and~\eqref{uniform_bound_momp}, we get
	\begin{align*}
I &\leq e^{-b(t_{r}-t_{h})}W_{p}^{p}(\theta _{t_{r},t_{k}}(\mu ),\Theta
_{t_{r},t_{k}}^{\pi }(\mu ))+C(1+\|\mu\|_p^p)\gamma _{r}^{\varepsilon } \\
&\leq C(1+\|\mu\|_p^p)(e^{-b(t_{r}-t_{h})}\gamma _{r}^{\varepsilon }+\gamma
_{r}^{\varepsilon })\leq C(1+\|\mu\|_p^p)\gamma _{r}^{\varepsilon }.
	\end{align*}
	For $II$, we use the commutativity of $\theta $,  (%
	\ref{eq:New6a}) and then~\eqref{uniform_bound_momp}  to get
	\begin{equation*}
		II\leq C \Gamma_p^p(\theta
		_{t_{r},t_{k}}(\mu ),\theta
		_{t_{r},t_{k}}(\mu )) \gamma_r^\varepsilon \le  C'\Gamma _{p}^{p}(\mu )\gamma
		_{r}^{\varepsilon }.
	\end{equation*}%
	We do the same  for $III$ and get
	\begin{equation*}
		III\leq e^{-bt_r}W_{p}^{p}((\Theta _{t_{r},t_k }^{\pi }(\mu ),\theta
		_{t_r,t_k  }(\mu )))+ C\Gamma _{p}^{p}((\Theta _{t_{r},t_k }^{\pi }(\mu ),\theta
		_{t_r,t_k  }(\mu )))\gamma _{r}^{\varepsilon }\leq C'\Gamma _{p}^{p}(\mu
		)\gamma _{r}^{\varepsilon }.
	\end{equation*}%
	From the above estimates and using the non-increasing property of the
	sequence $\gamma $, we obtain the condition (\ref{New6}) with $\delta_{h,r,k}=\gamma_r^\varepsilon$.

	\medskip By Theorem \ref{Thm_stationary}, a $(\psi ,p)-$ stationary measure $\mu
	_{\ast }$ for $\psi $ does exist. We prove now that $\mu _{\ast }$ is
	invariant for $\theta$, provided that $\theta$ is continuous. By (\ref%
	{eq:New6a}), 
	\begin{equation*}
		\lim_{n\rightarrow \infty }W_{p}^{p}(\theta _{t_{0},t_{n}}(\mu ),\mu _{\ast
		})=\lim_{n\rightarrow \infty }W_{p}^{p}(\mathcal{L}(X_{n}),\mu _{\ast })\leq
		\lim_{n\rightarrow \infty }C(e^{-\overline{b} t_{n}}+c_{\ast }\gamma
		_{n}^{\varepsilon })=0.
	\end{equation*}%
	Therefore by continuity%
	\begin{equation*}
		\theta _{0,t_{n+1}}(\mu _{\ast })=\lim_{m}\theta _{0,t_{n+1}}(\theta
		_{0,t_{m+1}}(\mu _{\ast }))=\mu _{\ast }.
	\end{equation*}%
	\end{proof}

	\section{Horowitz-Karandikar's estimate of the $W_{p}$ distance, and regularization of random measures}\label{app:HK}

	We prove the following generalization of the result by
	Horowitz and Karandikar \cite{HK} for $W_p$ which is used in the proof of  Theorem
	\ref{thm_ergo}.
	
	\begin{theorem}
		\label{thm_HKp} Let $p\ge 1$. Let $f,g:{\mathbb{R}}^d \to {\mathbb{R}}_+$ be
		probability density functions such that $\int_{{\mathbb{R}}^{d}}\left\vert
		x\right\vert ^{p}f(x)dx+\int_{{\mathbb{R}}^{d}}\left\vert x\right\vert
		^{p}g(x)dx<\infty$. Then, we have 
		\begin{equation}
			W_{p}^{p}(f(x)dx,g(x)dx)\leq 2^{p-1} \int_{{\mathbb{R}}^{d}}\left\vert
			x\right\vert ^{p}\left\vert f(x)-g(x)\right\vert dx.  \label{HKp}
		\end{equation}
	\end{theorem}
	
	Note that \cite[Lemma 2.2]{HK} proves the estimate for the $W_2$ distance and
	the Euclidean norm on~${\mathbb{R}}^d$ with a constant~$3$ instead of $%
	2^{2-1}=2$. Nonetheless their argument can be used to get~\eqref{HKp} for
	any norm on~${\mathbb{R}}^d$ as follows.
	
	\begin{proof}
		We consider the coupling used by~\cite{HK} and define 
		\begin{equation*}
			M(dx,dy)=f(x)\wedge g(x)\delta _{x}(dy)dx+\frac{(f(x)-f(x)\wedge
				g(x))(g(y)-f(y)\wedge g(y))}{1-\int_{{\mathbb{R}}^{d}}f(z)\wedge g(z)dz}dxdy,
		\end{equation*}
		with the convention $M(dx,dy)=f(x)\delta _{x}(dy)dx$ when $\int_{{\mathbb{R}}
			^{d}}f(z)\wedge g(z)dz=1$, i.e. when $f=g$ almost surely. We check easily
		that $M$ is a probability measure such that $\int_{y\in {\mathbb{R}}
			^{d}}M(dx,dy)=f(x)dx$ and $\int_{x\in {\mathbb{R}}^{d}}M(dx,dy)=g(y)dy$, so
		that 
		\begin{align*}
			W_{p}^{p}& (f(x)dx,g(x)dx)\leq \int_{{\mathbb{R}}^{d}\times {\mathbb{R}}
				^{d}}|x-y|^{p}M(dx,dy) \\
			& =\int_{{\mathbb{R}}^{d}\times {\mathbb{R}}^{d}}|x-y|^{p}\frac{
				(f(x)-f(x)\wedge g(x))(g(y)-f(y)\wedge g(y))}{1-\int_{{\mathbb{R}}
					^{d}}f(z)\wedge g(z)dz}dxdy \\
			& \leq 2^{p-1}\int_{{\mathbb{R}}^{d}\times {\mathbb{R}}^{d}}(|x|^{p}+\left
			\vert y\right\vert ^{p})\frac{(f(x)-f(x)\wedge g(x))(g(y)-f(y)\wedge g(y))}{
				1-\int_{{\mathbb{R}}^{d}}f(z)\wedge g(z)dz}dxdy \\
			& =2^{p-1}\int_{{\mathbb{R}}^{d}}|x|^{p}(f(x)-f(x)\wedge
			g(x))dx+2^{p-1}\int_{{\mathbb{R}}^{d}}|y|^{p}(g(y)-f(y)\wedge g(y))dy \\
			& =2^{p-1}\int_{{\mathbb{R}}^{d}}|x|^{p}|f(x)-g(x)|dx,
		\end{align*}
		since $f(x)-f(x)\wedge g(x)+g(x)-f(x)\wedge g(x)=|f(x)-g(x)|$.
	\end{proof}

	Using a regularization procedure, we now extend the estimate of Theorem~\ref%
	{thm_HKp} to general random measures - we follow here the ideas in \cite{DJL}. Let
	us be more precise and recall the notation~\eqref{HK2} and \eqref{def:HKp}. For two random probability measures $\mu (\omega ,dx)$ and $%
	\nu (\omega ,dx)$ and for
	a bounded measurable function $\varphi \in C_{b}^{0}({\mathbb{R}}^{d})$, we
	denote%
	\begin{equation*}
		S_{\mu ,\nu ,p}(\varphi )=\mathbb{E}\left[ \left\vert \int \varphi (x)\mu
		(\omega ,dx)-\int \varphi (x)\nu (\omega ,dx)\right\vert ^{p}\right]^{1/p}.
	\end{equation*}
	Let $\mathcal{A}$ be a subset of $C_{b}^{0}({\mathbb{R}}^{d})$. We define 
	\begin{equation*}
		\overline{W}_{1}^{\mathcal{A}}(\mu ,\nu )=\sup \{S_{\mu ,\nu ,1}(\varphi ):\varphi \in \mathcal{A}		\}.
	\end{equation*}%
	Let us note that when $\mu$ and $\nu$ are not random, $\overline{W}_{1}^{\mathcal{A}}(\mu ,\nu )$  can be seen as an example of probability metric as defined by Zolotarev~\cite{Zolotarev} or M\"uller~\cite{Muller}. 
	Note also that in the particular case $ \mathcal{A}=\{\varphi \in  C_{b}^{1}({\mathbb{R}}^{d})\text{ s.t. }\left\Vert \varphi \right\Vert
		_{\infty }\times \left\Vert \nabla \varphi \right\Vert _{\infty }\leq 1 \}$, 
	we get the distance $\overline{W}_{1}$ defined in~\eqref{def:HKp} with $p=1$.

	We will use regularisation and  consider in the next lemma  a given probability density $\phi \in
	C^{\infty}({\mathbb{R}}^{d})$ which has the support
	included in the unit ball. It is bounded with bounded derivatives, and we define $\phi _{r}(x)=r^{-d}\phi (\frac{x}{r})$. Then $\phi _{r}$ is bounded by $Cr^{-d}$, and $\partial _{i_1}\dots \partial_{i_m}\phi _{r}$ is
	bounded by $Cr^{-d-m}$ and $\phi _{r}(x)\neq 0$ implies $\left\vert
	x\right\vert \leq r.$  

	Last, we recall that we denote $\left\Vert \mu \right\Vert _{p}^{p}=\mathbb{E}%
	\left[\int_{{\ \mathbb{R}}^{d}}\left\vert x\right\vert ^{p}\mu (\omega ,dx)\right].$

\begin{lemma}
		\label{lemma_reg_wasserstein} Let $1\le p<p^{\prime }$. Suppose that $%
		\mu(\omega,dx)$ and $\nu(\omega,dx)$ are random probability measures on~${\ 
			\mathbb{R}}^d$ such that $\left\Vert \mu \right\Vert _{p^{\prime
		}}^{p^{\prime }}+\left\Vert \nu \right\Vert _{p^{\prime }}^{p^{\prime
		}}<\infty$. We assume that there exists an exponent $\ed>0$ such that 
		\begin{equation}\label{assump_calA} \exists C \in \R_+,\ \forall x \in \R^d, \forall r \in (0,1], \ S_{\mu ,\nu ,1}(\phi_r(x-\cdot)) \le C r^{-\ed}\overline{W}_{1}^{\mathcal{A}}(\mu ,\nu ).
		\end{equation}		
		Let $\hat{p}=\frac{p(p^{\prime }-p)}{
			(\ed+p)(p^{\prime }-p)+(d+p)p}$. Then, there exists a constant $C\in {\mathbb{R}}_+$ depending on 
		$p$, $p^{\prime }$ and $d$ such that  
		\begin{equation}
			\mathbb{E}[W_{p}^{p}(\mu ,\nu )]\leq C (1+ \left\Vert \mu \right\Vert
			_{p^{\prime }}^{p^{\prime }}+\left\Vert \nu \right\Vert _{p^{\prime
			}}^{p^{\prime }})^{1-\hat{p}} \overline{W}^{\mathcal{A}}_{1}(\mu ,\nu )^{\hat{p}}.
			\label{HK3}
		\end{equation}
		In particular, for  $ \mathcal{A}=\{\varphi \in  C_{b}^{1}({\mathbb{R}}^{d})\text{ s.t. }\left\Vert \varphi \right\Vert
		_{\infty }\times \left\Vert \nabla \varphi \right\Vert _{\infty }\leq 1 \}$, \eqref{assump_calA} holds with $\ed=d+1/2$. 
	\end{lemma}
	
	\begin{proof}
	
	\textbf{Step 1 (regularization)} Using the function $\phi_r$ above, we define 
	\begin{equation*}
		p_{\mu _{r}}(x)=\int_{{\mathbb{R}}^{d}}\phi _{r}(x-y)\mu (\omega ,dy)\quad
		\mu _{r}(\omega ,dx)=p_{\mu _{r}}(x)dx.
	\end{equation*}%
	We claim that 
	\begin{equation*}
		W_{p}^{p}(\mu (\omega ,\cdot ),\mu _{r}(\omega ,\cdot ))\leq r^{p}.
	\end{equation*}%
	Indeed, fix $\omega ,$ and take $X$ a random variable of law $\mu (\omega
	,dx)$ and another random variable $Y$ which is independent of $X$ and has
	law $\phi _{r}(x)dx$. Then it is easy to check that $X+Y$ has law $\mu
	_{r}(\omega ,dx)$ and consequently 
	\begin{equation*}
		W_{p}^{p}(\mu (\omega ,\cdot ),\mu _{r}(\omega ,\cdot ))\leq \mathbb{E}
		[\left\vert X+Y-X\right\vert ^{p}]=\mathbb{E}[\left\vert Y\right\vert ^{p}]\leq
		r^{p}.
	\end{equation*}%
	From the triangle inequality, we finally get for every $r>0$ 
	\begin{equation*}
		\mathbb{E}[W_{p}^{p}(\mu (\omega ,\cdot ),\nu (\omega ,\cdot ))]\leq
		3^{p-1}\left( 2r^{p}+\mathbb{E}[W_{p}^{p}(\mu _{r}(\omega ,\cdot ),\nu
		_{r}(\omega ,\cdot ))]\right) .
	\end{equation*}
	
	\textbf{Step 2. }We fix $R>r$ to be chosen in the following. Using Theorem~%
	\ref{thm_HKp}, we get 
	\begin{equation*}
		\mathbb{E}[W_{p}^{p}(\mu _{r}(\omega ,\cdot ),\nu _{r}(\omega ,\cdot ))]\leq
		2^{p-1} \mathbb{E}\int_{{\mathbb{R}}^{d}}\left\vert x\right\vert
		^{p}\left\vert \int_{{\mathbb{R}}^{d}}\phi_{r}(x-y) (\mu (\omega ,dy)-\nu
		(\omega ,dy))\right\vert dx=:J_{1}+J_{2},
	\end{equation*}%
	with $J_{1}$ being the integral on $\{\left\vert x\right\vert \leq R\}$ and $%
	J_{2}$ the integral on $\{\left\vert x\right\vert >R\}$.
	
	We estimate first $J_{2}$. We write%
	\begin{align*}
		\int_{\{\left\vert x\right\vert >R\}}\left\vert x\right\vert ^{p}&\left\vert
		\int_{{\mathbb{R}}^{d}}\phi _{r}(x-y)\mu (\omega ,dy)\right\vert dx = \int_{{%
				\ \mathbb{R}}^{d}}\mu (\omega ,dy)\int_{\{\left\vert x\right\vert
			>R\}}\left\vert x\right\vert ^{p}\phi _{r}(x-y)dx \\
		&\leq 2^{p-1}\int_{{\mathbb{R}}^{d}}\mu (\omega ,dy)\int_{\{\left\vert
			x\right\vert >R\}}(\left\vert y\right\vert ^{p}+\left\vert x-y\right\vert
		^{p})\phi _{r}(x-y)dx.
	\end{align*}%
	The function $\phi _{r}$ has the support included in the ball of radius $r$,
	so, if $\left\vert x\right\vert >R$ and $\phi _{r}(x-y)\neq 0$ then we have $%
	\left\vert y\right\vert \geq R-r$. We write first%
	\begin{align*}
\int_{{\mathbb{R}}^{d}}\mu (\omega ,dy)\int_{\{\left\vert x\right\vert
	>R\}}\left\vert y\right\vert ^{p}\phi _{r}(x-y)dx &=\int_{{\mathbb{R}}
	^{d}}\int_{{\mathbb{R}}^{d}}\left\vert y\right\vert ^{p}1_{\{\left\vert
	x\right\vert >R\}}\mu (\omega ,dy)\phi _{r}(x-y)dx \\
&\leq \int_{{\mathbb{R}}^{d}}\int_{\{\left\vert y\right\vert \geq
	R-r\}}\left\vert y\right\vert ^{p}1_{\{\left\vert x\right\vert >R\}}\mu
(\omega ,dy)\phi _{r}(x-y)dx \\
&\leq \int_{\{\left\vert y\right\vert \geq R-r\}}\left\vert y\right\vert
^{p}\mu (\omega ,dy) \\
&\leq (\int_{\{\left\vert y\right\vert \geq R-r\}}\left\vert y\right\vert
^{p^{\prime }}\mu (\omega ,dy))^{p/p^{\prime }}(\mu (\omega ,\{\left\vert
y\right\vert \geq R-r\}))^{1-p/p^{\prime }} \\
&\leq (\int_{\{\left\vert y\right\vert \geq R-r\}}\left\vert y\right\vert
^{p^{\prime }}\mu (\omega ,dy))\times \frac{1}{(R-r)^{p^{\prime }-p}},
	\end{align*}
	since $\mu (\omega ,\{\left\vert y\right\vert \geq R-r\})\le
	(R-r)^{-p^{\prime }} \int_{\{\left\vert y\right\vert \geq R-r\}}\left\vert
	y\right\vert ^{p^{\prime }}\mu (\omega ,dy) $ by Markov's inequality.
	Moreover, 
	\begin{align*}
		\mathbb{E}\int_{{\mathbb{R}}^{d}}\mu (\omega ,dy)\int_{\{\left\vert
		x\right\vert >R\}}\left\vert x-y\right\vert ^{p}\phi _{r}(x-y)dx &\leq  
	\mathbb{E}\int_{\{\left\vert y\right\vert >R-r\}} \mu (\omega ,dy)\int_{{\ 
			\mathbb{R}}^{d}}\left\vert x-y\right\vert ^{p}\phi _{r}(x-y)dx \\
	&\leq C r^{p}\times\mathbb{E} \mu (\omega ,\{\left\vert y\right\vert >R-r\} )\\
	&\leq C \frac{r^{p}}{(R-r)^{p^{\prime }}}\times \left\Vert \mu \right\Vert
	_{p^{\prime }}^{p^{\prime }}.
	\end{align*}
	The same estimates hold with $\nu $ instead of $\mu$, and we conclude that
	for $R>r$,
	\begin{equation*}
		J_{2}\leq (\left\Vert \mu \right\Vert _{p^{\prime }}^{p^{\prime
		}}+\left\Vert \nu \right\Vert _{p^{\prime }}^{p^{\prime
		}})\left((R-r)^{p-p^{\prime }}+\frac{r^{p}}{(R-r)^{p^{\prime }}} \right)\leq
		2(\left\Vert \mu \right\Vert _{p^{\prime }}^{p^{\prime }}+\left\Vert \nu
		\right\Vert _{p^{\prime }}^{p^{\prime }})(R-r)^{p-p^{\prime }} .
	\end{equation*}
	
	We now focus on $J_{1}$ and write
	\begin{align*}
		J_{1} &=\mathbb{E}\int_{\{\left\vert x\right\vert \leq R\}}\left\vert
	x\right\vert ^{p}\left\vert \int_{{\mathbb{R}}^{d}}\phi _{r}(x-y)(\mu
	(\omega ,dy)-\nu (\omega ,dy))\right\vert dx \\
	&=\int_{\{\left\vert x\right\vert \leq R\}}\left\vert x\right\vert
	^{p}S_{\mu ,\nu, 1 }(\phi _{r}(x-\cdot ))dx\leq C_d {R}^{d+p}\sup_{x
		\in {\mathbb{R}}^d}S_{\mu ,\nu,1 }(\phi _{r}(x-\cdot )) 
	\end{align*}
	where $C_d$ is a constant depending only on~$d$ (and related to the choice
	of the norm). Then, by using~\eqref{assump_calA}, we get 
	\begin{equation*}
		J_{1}\leq C r^{-\ed}{R}^{d+p}\times \overline{W}_{1}^{\mathcal{A}}(\mu ,\nu ),
	\end{equation*}%
and obtain finally  
	\begin{equation*}
		\mathbb{E}[W_{p}^{p}(\mu ,\nu )]\leq 3^{p-1}\left(2r^p+ 2(\left\Vert \mu
		\right\Vert _{p^{\prime }}^{p^{\prime }}+\left\Vert \nu \right\Vert
		_{p^{\prime }}^{p^{\prime }})(R-r)^{p-p^{\prime }} + C r^{-\ed}{R}
		^{d+p}\times \overline{W}_{1}^{\mathcal{A}}(\mu ,\nu )\right).
	\end{equation*}%
	We choose $R=r + r^{- \frac{p}{p^{\prime }-p}}$ so that $(R-r)^{p-p^{\prime
	}}=r^p$. Besides, we suppose that $r<\left( \frac 12 \right)^{\frac{
			p^{\prime }-p}{p^{\prime }}}$ in order to have $R-r>2r$. We also have $%
	R<2r^{- \frac{p}{p^{\prime }-p}}$ since $r<1<r^{- \frac{p}{p^{\prime }-p}}$. This gives 
	\begin{align*}
		\mathbb{E}[W_{p}^{p}(\mu ,\nu )]&\leq 3^{p-1}\left(2(1+\left\Vert \mu
		\right\Vert _{p^{\prime }}^{p^{\prime }}+\left\Vert \nu \right\Vert
		_{p^{\prime }}^{p^{\prime }})r^{p} + C r^{- \ed} \left(2 r^{- \frac{p}{
				p^{\prime }-p}}\right)^{d+p}\overline{W}_{1}^{\mathcal{A}}(\mu ,\nu )\right) \\
		&\leq C \left((1+\left\Vert \mu \right\Vert _{p^{\prime }}^{p^{\prime
		}}+\left\Vert \nu \right\Vert _{p^{\prime }}^{p^{\prime }})r^{p} + r^{-\frac{
				(p^{\prime }-p)\ed+(d+p) p}{p^{\prime }-p}} \overline{W}_{1}^{\mathcal{A}}(\mu ,\nu
		)\right).
	\end{align*}%
	We apply the next elementary lemma.
	
	\begin{lemma}
		Let $A,B,\alpha,\beta>0$. The function $h(r)=Ar^{\alpha}+Br^{-\beta}$, $r>0$
		reaches its minimum on ${\mathbb{R}}_+^*$ at $r^*=\left( \frac{B\beta }{A
			\alpha} \right)^{\frac 1{\alpha+\beta}}$ and $h(r^*)=A\left(1+ \frac \alpha
		\beta \right)(r^*)^\alpha=(1+\frac \alpha \beta )\left(\frac \beta \alpha
		\right)^{\frac{\alpha}{\alpha+\beta}} A^{\frac{\beta}{\alpha+\beta}}B^{\frac{
				\alpha}{\alpha+\beta}} $.
	\end{lemma}
	
	We optimize with respect to $0<r<\left( \frac 12 \right)^{\frac{p^{\prime
			}-p }{p^{\prime }}}$ and get
	\begin{equation*}
		\mathbb{E}[W_{p}^{p}(\mu ,\nu )]\leq C (1+ \left\Vert \mu \right\Vert
		_{p^{\prime }}^{p^{\prime }}+\left\Vert \nu \right\Vert _{p^{\prime
		}}^{p^{\prime }})^{\frac{\ed(p^{\prime }-p)+(d+p)p}{(\ed+p)(p^{\prime
				}-p)+(d+p)p}} \overline{W}_{1}^{\mathcal{A}}(\mu ,\nu )^{\frac{p(p^{\prime }-p)}{
				(\ed+p)(p^{\prime }-p)+(d+p)p}},
	\end{equation*}
	for a constant $C$ depending on $p$, $p^{\prime }$ and $d$, provided that 
	\begin{equation*}
		\frac{ \left((p^{\prime }-p) \ed+(d+p) p\right) \overline{W}_{1}^{\mathcal{A}}(\mu ,\nu
			) }{p (p^{\prime }-p) (1+\left\Vert \mu \right\Vert _{p^{\prime
			}}^{p^{\prime }}+\left\Vert \nu \right\Vert _{p^{\prime }}^{p^{\prime }}) }
		<\left( \frac 12 \right)^{\frac{(\ed+p)(p^{\prime }-p)+(d+p)p}{p^{\prime }}}.
	\end{equation*}
	Otherwise, we have $\overline{W}_{1}^{\mathcal{A}}(\mu ,\nu )\ge c (1+ \left\Vert \mu
	\right\Vert _{p^{\prime }}^{p^{\prime }}+\left\Vert \nu \right\Vert
	_{p^{\prime }}^{p^{\prime }})$ with $c>0$ depending on $p$, $p^{\prime }$
	and $d$. Therefore 
	\begin{align}
		\mathbb{E}[W_{p}^{p}(\mu ,\nu )] \leq& 2^{p-1}(\left\Vert \mu \right\Vert
		_{p}^{p}+\left\Vert \nu \right\Vert _{p}^{p})\le 2^{p-1}(\left\Vert \mu
		\right\Vert _{p^{\prime }}^{p^{\prime }}+\left\Vert \nu \right\Vert
		_{p^{\prime }}^{p^{\prime }}) \\
		\leq & C (1+ \left\Vert \mu \right\Vert _{p^{\prime }}^{p^{\prime
		}}+\left\Vert \nu \right\Vert _{p^{\prime }}^{p^{\prime }})^{1-\hat{p}} 
		\overline{W}_{1}^{\mathcal{A}}(\mu ,\nu )^{\hat{p}} .
	\end{align}

	Now, let us consider  $ \mathcal{A}=\{\varphi \in  C_{b}^{1}({\mathbb{R}}^{d})\text{ s.t. }\left\Vert \varphi \right\Vert
		_{\infty }\times \left\Vert \nabla \varphi \right\Vert _{\infty }\leq 1 \}$. For any $x \in \R^d$, we have $\phi_r(x-\cdot)\in \mathcal{A} $, $\left\Vert \phi_r \right\Vert_\infty \le C r^{-d}$ and $\left\Vert \partial_i \phi_r \right\Vert_\infty \le C r^{-d-1}$ for some constant $C>0$. Therefore, $\frac{\phi_r(x-\cdot)}{C r^{-(d+1/2)}} \in \mathcal{A}$ and thus $$ S_{\mu ,\nu ,1}(\phi_r(x-\cdot))=C r^{-(d+1/2)}S_{\mu ,\nu ,1}\left( \frac{\phi_r(x-\cdot)}{C r^{-(d+1/2)}}  \right)  \le C r^{-(d+1/2)}\overline{W}_{1}^{\mathcal{A}}(\mu ,\nu ).$$ 
		\end{proof}

	\begin{remark}\label{Remark_regularization} Let us note that Equation~\eqref{HK3} gives an upper bound for the $p$-Wasserstein distance in terms of the metric $\overline{W}^\mathcal{A}_{1}$ that are quite flexible to estimate. For example, we	 can use Lemma~\ref{lemma_reg_wasserstein} to get rates of convergence of
		the empirical measure. Namely, let $\mu \in \mathcal{P}_{p}(\mathbb{R}^{d})$
		be a probability measure, $X_{1},\dots ,X_{n}$ be i.i.d. sample according to
		this distribution and $\mu _{n}(dx)=\frac{1}{n}\sum_{i=1}^{n}\delta _{X_{i}}$
		be the (random) empirical measure. Assuming further that $\mu \in \mathcal{P}%
		_{p^{\prime }}(\mathbb{R}^{d})$ for $p^{\prime }>p$, Lemma~\ref%
		{lemma_reg_wasserstein} gives 
		\begin{equation*}
			\mathbb{E}[W_{p}^{p}(\mu ,\mu _{n})]\leq C(1+2\left\Vert \mu \right\Vert
			_{p^{\prime }}^{p^{\prime }})^{1-\hat{p}}\overline{W}^\mathcal{A}_{1}(\mu ,\mu _{n})^{%
				\hat{p}},
		\end{equation*}%
		with $\mathcal{A}=\{\varphi \in C^0_b(\R^d), \|\varphi\|_\infty\le 1 \}$ and $\ed=d$ since $\|\phi_r \|_\infty \le Cr^{-d}$. 
		We have $\overline{W}^\mathcal{A}_{1}(\mu ,\mu _{n})\leq \overline{W}^\mathcal{A}_{2}(\mu ,\mu
		_{n})\leq \frac{1}{\sqrt{n}}$ since 
		\begin{equation*}
			\mathbb{E}\left[ \frac{1}{n}\sum_{i=1}^{n}\varphi (X_{i})-\mathbb{E}[\varphi
			(X_{1})]\right] ^{2}=\frac{Var(\varphi (X_{1}))}{n}\leq \frac{1}{n}
		\end{equation*}%
		for $\Vert \varphi \Vert _{\infty }\leq 1$. When $p^{\prime }$ is
		asymptotically large, we get a rate of convergence in $O((1/n)^{\frac{p}{%
				2(p+d)}-})$, which is less accurate than the rate given in~\cite[Theorem 1]%
		{FG} that is roughly in $O((1/n)^{\min (p/d,1/2)})$.
		
		Note that it would be interesting to use the method developed by~\cite{DSS}
		in our context to get possibly higher order rate of convergence. The
		difficulty, however, is to get an estimate analogous to~\eqref{ineq_p}, but
		for $\varphi(x)=\mathbf{1}_{C}(x)$ where $C$ is a hypercube of~$\mathbb{R}%
		^d $. This is left for further research.
	\end{remark}
	
	\section{Proof of Lemma \ref{lemma:BoltzB}}\label{app:BoltzB}
	
	Let us stress that a standard application of  $L^p$ estimates for SDEs driven by Poisson noise (see e.g. \cite[Theorem 2.11]{Kunita2004}) would not yield the ``precise contraction'' we are looking for, so we need to apply It\^o's formula carefully, together with a Taylor expansion.
	\smallskip

	First, assume that the following inequality is satisfied for any even integer $q$, $q\le p$,
	\begin{equation}	\label{estX2}
		\begin{array}{rl}
			\E[ |\X_{s,t}(x,\mu)|^q]
			\leq &
			\displaystyle \Big(1-q\Big(\bar{b}-\delta-\frac 12\kappa_q\Big)(t-s)\Big)|x|^q
			+\frac{q}2\kappa_q\|\mu\|_q^q(t-s)\smallskip\\
			&\displaystyle+C_{q,b,\delta}(1+\|\mu\|_2^2)|x|^{q-2}(t-s)\smallskip\\
			&\displaystyle+C_{q}(1+|x|^q+\|\mu\|_q^q)(t-s)^{2},
		\end{array}
	\end{equation}
	where $\delta>0$ and $C_{q,b,\delta}>0$ is a positive constant depending only on $q,b,\delta$.

	Assuming the above inequality, we have that if we consider a r.v. $X$ which is $\mathcal{F}_s$ measurable and such that $\mathcal{L}(X)=\mu\in \mathcal{P}_p(\R^d)$. We use \eqref{estX2} with $q=p$ and we obtain
	\begin{equation*}
		\begin{array}{rl}
			\E[ |\X_{s,t}(X,\mu)|^p]
			\leq &
			\displaystyle \Big(1-p\Big(\bar{b}-\delta-\frac 12\kappa_p\Big)(t-s)\Big)\|\mu\|_p^p
			+\frac{p}2\kappa_p\|\mu\|_p^p(t-s)\smallskip\\
			&\displaystyle+C_{p,b,\delta}(1+\|\mu\|_2^2)\|\mu\|_{p-2}^{p-2}(t-s)\smallskip\\
			&\displaystyle+C_{p}(1+\|\mu\|_p^p)(t-s)^{2}.
		\end{array}
	\end{equation*}
	With this, \eqref{estX3}  follows by rearranging the terms and remarking that $1+(1+\|\mu\|_2^2)\|\mu\|_{p-2}^{p-2}\leq C_p(1+\|\mu\|_{p-2}^p)$.

	Therefore it only remain to prove \eqref{estX2}. This is done by induction on the even integers.
	
	We start with the case $q=2$.
	By using It\^o's formula and \eqref{vl2}, we have
	\begin{align*}
		\E[|\X_{s,t}(x,\mu)|^2]
		=&|x|^2+
		2\int_s^t\<b(x), \E[\X_{s,t}(x,\mu)]\>dr\\
		&+\int_s^t\int_{\R^d}\int_E \E[|c(v,z,x)|^2]d\mu(v)d\rho(z)dr\\
		\leq &|x|^2+
		2\int_s^t\<b(x), \E[\X_{s,t}(x,\mu)]\>dr+ 2Q_2(1+|x|^2+\|\mu\|_2^2)(t-s).
	\end{align*}
	Since $\E[\X_{s,r}(x,\mu)]=x+b(x)(r-s)$, using \eqref{VL4} and~\eqref{vl1}
	\begin{align*}
		&\<b(x),\E[\X_{s,r}(x,\mu)]\>
		\leq (-\bar{b}+\delta)|x|^2+L_b\delta^{-1}|b(0)|^2+C_b(1+|x|^2)(r-s).
	\end{align*}
	Therefore, by inserting in the above we obtain
	\begin{align*}
		\E[|\X_{s,t}(x,\mu)|^2]
		\leq &\Big(1-2\Big(\bar{b}-\delta-Q_2\Big)(t-s)\Big)|x|^2+
		Q_2\|\mu\|_2^2(t-s)\\
		&+C_{b,\delta}\big(1+|x|^2(t-s)\big)(t-s).
	\end{align*}
	This proves the statement for $q=2$.
	
	Now, we want to prove the statement for $q\geq 4$. The inductive hypothesis and H\"older's inequality give  
	\begin{equation}\label{est-p}
		\E[|\X_{s,t}(x,\mu)|^{q-2}]
		\leq |x|^{q-2}+C_{q-2}\big(1+|x|^{q-2}+\|\mu\|_{q-2}^{q-2}\big)(t-s).
	\end{equation}
	In fact, 
assuming that \eqref{estX2} holds for $q-2$ (i.e. inductive hypothesis), it follows that
$$
\begin{array}{rl}
\E[ |\X_{s,t}(x,\mu)|^{q-2}]
\leq &
\displaystyle \Big(1+C(t-s)\Big)|x|^{q-2} 
+C\|\mu\|_{q-2}^{q-2}(t-s)\smallskip\\
&\displaystyle+C_{q-2,b,\delta}(1+\|\mu\|_{2})^2|x|^{q-4}(t-s)\smallskip\\
&\displaystyle+C(1+|x|^{q-2}+\|\mu\|_{q-2}^{q-2})(t-s)\smallskip\\
\leq &
\displaystyle 
|x|^{q-2} + C(1+|x|^{q-2}+\|\mu\|_{q-2}^{q-2})(t-s)
\smallskip\\
&\displaystyle+C_{q-2,b,\delta}(1+\|\mu\|_{2})^2|x|^{q-4}(t-s).
\end{array}
$$
Consider now the last term. Since $(1+\|\mu\|_{2})^2|x|^{q-4}\le C_q((1+\|\mu\|_{2})^{q-2}+ |x|^{q-2})$ and $\|\mu\|_{2}\leq \|\mu\|_{q-2}$, we get~\eqref{est-p}.

	Set $f_q(x)=|x|^q$. By It\^o's formula,
	\begin{align*}
		\E[ |\X_{s,t}(x,\mu)|^q]
		=&|x|^q+ I_{s,t}(x,\mu)+J_{s,t}(x,\mu),\quad \mbox{where}\\
		I_{s,t}(x,\mu)
		=&\int_{s}^{t}\E \<\nabla f_q(\X_{s,r}(x,\mu)),b(x)\>dr\\
		J_{s,t}(x,\mu)
		=&\int_{s}^{t}\int_{\R^d\times E}\E\big[f_q\big(\X_{s,r}(x,\mu)+c(v,z,x)\big)-f_q(\X_{s,r}(x,\mu))\\
		&\qquad\qquad\quad-\<\nabla f_q(\X_{s,r}(x,\mu)),c(v,z,x)\>\big]d\mu(v)d\rho (z)dr.
	\end{align*}
	For simplicity, in the following computations we drop the dependence on $(x,\mu)$.
	
	We first study the term $J_{s,t}$. By Taylor’s formula with remainder in mean-value form, we have 
	$$
	f_q(\xi+c)-f_q(\xi)
	-\<\nabla f_q(\xi),c\>
	=\frac 12 \big\<\mathbf{H}_{f_q}(\xi+\theta c)c,c\big\>
	%
	%
	$$
	where $\mathbf{H}_{f_q}$ denotes the Hessian matrix of $f_q$ and $\theta\equiv\theta(\xi,c)\in (0,1)$. By a direct computation, 
	one gets
	\begin{align*}
		&f_q(\xi+c)-f_q(\xi)
		-\<\nabla f_q(\xi),c\>
		\leq  2^{q-4}q(q-1)(|\xi|^{q-2}+|c|^{q-2})|c|^2.
	\end{align*}
	Therefore,
	using \eqref{vlq}  and integrating with respect to $\mu$, we obtain 
	\begin{align*}
		J_{s,t}
		\leq &2^{q-3}q(q-1)Q_2(1+\|\mu\|_2^2+|x|^2)\int_s^t
		\E[|\X_{s,r}|^{q-2}]dr\\
		&+2^{2q-5}q(q-1)Q_q(1+\|\mu\|_q^q+|x|^q)(t-s).
	\end{align*}
	We now use \eqref{est-p} and  $\kappa_q=2^{q-3}(q-1)(2Q_2+2^{q-1}Q_q)$, to obtain
	\begin{equation}\label{J}
		\begin{array}{rl}
			J_{s,t}
			\leq 
			&\displaystyle
			\frac q2 \kappa_q(1+\|\mu\|_q^q+|x|^q)(t-s)
			+C_q|x|^{q-2}(1+\|\mu\|^2_2)(t-s)\smallskip\\
			&\displaystyle
			+C_q(1+\|\mu\|_q^q+|x|^q)(t-s)^2.
		\end{array}
	\end{equation}
	We now study $I_{s,t}$. 
	We write
	\begin{align*}
		I_{s,t}=&\int_{s}^{t}\<\nabla f_q(x),b(x)\>dr
		+\int_{s}^{t} \E\<\nabla f_q(\X_{s,r})-\nabla f_q(x),b(x)\>dr.
	\end{align*}
	Since $\nabla f_q(x)=q|x|^{q-2}x$, $\<\nabla f_q(x),b(x)\>=q|x|^{q-2}\<x,b(x)\>$, and using \eqref{vl1} we have
	\begin{align*}
		I_{s,t}\leq &q(-\bar{b}+\delta)|x|^q(t-s)+q C_{b,\delta} |x|^{q-2}(t-s) +H_{s,t}\quad \mbox{ where }\\
		H_{s,t}&=\int_{s}^{t}\big\<\E\big[\nabla f_q(\X_{s,r})-\nabla f_q(x)\big],b(x)\big\> dr.
	\end{align*}
	Consider $H_{s,t}$. To deal with the expectation inside the integral, we apply again It\^o's formula: 
	\begin{align*}
		&\<\E[\nabla f_q(\X_{s,r})-\nabla f_q(x)],b(x)\>
		=\int_s^r \E \Big[L_q(\X_{s,u},b(x))\Big]du\\
		&+
		\int_s^r \int_{\R^d\times E}\E\Big[M_q(\X_{s,u},c(v,z,x), b(x))\Big]d\mu(v)d\rho(z) du
	\end{align*}
	where
	\begin{align*}
		L_q(\xi,b)&=\<\mathbf{H}_{f_q}(\xi)b,b\>\\
		M_q(\xi, c,b)&=
		\<\nabla f_q(\xi+c)
		-\nabla f_q(\xi)-\mathbf{H}_{f_q}(\xi)c,b\>.
	\end{align*}
	Now, straightforward computations on the derivatives of $f_q$ give
	$$
	L_q(\xi,b) 
	\leq C_{q}|\xi|^{q-2}|b|^2,
	$$
	so that, by using \eqref{est-p} and \eqref{vl2}, we get
	$$
	\int_s^r \E [L_q(\X_{s,u},b(x))]du
	\leq C_q(1+|x|^q+\|\mu\|_q^q)(r-s).
	$$
	We study now the term $M_q$. Using that  $q\geq 4$ and a Taylor's expansion as before, we obtain
	\begin{align*}
		&M_q(\xi,c,b)
		\leq C_q(1+|\xi|^{q-3}+|c|^{q-3})|b| |c|^2,
	\end{align*}
	so that
	\begin{align*}
		&\E\big[M_q(\X_{s,u},c,b)\big]
		\leq C_q(1+\E\big[|\X_{s,u}|^{q-3}\big]+|c|^{q-3})|b| |c|^2.
	\end{align*}
	Since  $\|\cdot\|_{q-3}\leq \|\cdot\|_{q-2}$ one has
	$$
	\E[|\X_{s,u}|^{q-3}]\leq \E[|\X_{s,u}|^{q-2}]^{\frac {q-3}{q-2}}
	$$
	and by using \eqref{est-p},
	$$
	\E[|\X_{s,u}|^{q-3}]\leq C_q(1+|x|^{q-2}+\|\mu\|_{q-2}^{q-2})^{\frac {q-3}{q-2}}
	\leq C_q(1+|x|^{q-3}+\|\mu\|_{q-2}^{q-3}).
	$$
	From here, it follows as in the estimation for $J$,  that
	$$
	\int_s^r\int_{\mathbb{R}^d\times E}\E[M_q(\X_{s,u},c,b)]d\mu d\rho du
	\leq C_q(1+|x|^q+\|\mu\|_q^q)(r-s).
	$$
	Then, by inserting all estimates in the expression for $H_{s,t}$, we get
	\begin{align*}
		H_{s,t}\leq C_q(1+|x|^q+\|\mu\|_q^q)(t-s)^{2},
	\end{align*}
	so that 
	\begin{equation}\label{I}
		\begin{array}{rl}
			I_{s,t}\leq 
			&\displaystyle
			q(-\bar{b}+\delta)|x|^q(t-s)+q C_{b,\delta} |x|^{q-2}(t-s)\smallskip\\ 
			&\displaystyle
			+C_q(1+|x|^q+\|\mu\|_q^q)(t-s)^{2}.
		\end{array}
	\end{equation}
	Hence, using \eqref{J} and \eqref{I}, we finally obtain \eqref{estX2}.

		\section*{Acknowledgements} AA acknowledges the support of the ``Chaire Risques Financiers", Fondation du Risque.
	 AKH research was partially supported by KAKENHI 24K06789 and by MathRisk. LC acknowledges support from MathRisk, the MUR Excellence Project MatMod@TOV awarded to the Department of Mathematics, Tor Vergata University of Rome (CUP E83C23000330006) and the Research Project MLFGTSRL from Tor Vergata University of Rome (CUP E83C25000470005).
		\bibliographystyle{alpha}
		\bibliography{biblio_initials}
		
	\end{document}